\documentclass[reqno,11pt]{amsart}
\usepackage{amssymb,amsmath,amsthm}
\usepackage{amscd}
\usepackage{amsfonts}
\usepackage{amssymb}
\usepackage{latexsym}
\usepackage{color}
\usepackage{esint}
\usepackage{graphicx, subfigure}
\usepackage{caption}
\usepackage{amsthm}
\usepackage{hyperref}
\usepackage{tikz}
\usetikzlibrary{shapes.geometric}

 \usepackage[makeroom]{cancel}
 \usetikzlibrary{decorations.pathreplacing}

\let\hide\iffalse
\let\unhide\fi

\newtheorem{theorem}{Theorem}

\newtheorem{corollary}[theorem]{Corollary}
\newtheorem{definition}[theorem]{Definition}
\newtheorem{lemma}[theorem]{Lemma}
\newtheorem{proposition}[theorem]{Proposition}
\newtheorem{remark}[theorem]{Remark}

\let\p=\partial

\let\O=\Omega

\let\o=\omega

\let\b=\beta

\newcommand{\R}{\mathbb{R}}

\renewcommand{\S}{\mathbb{S}}

\newcommand{\be}{\begin{equation}}
\newcommand{\bm}{\begin{multline}}
\newcommand{\ee}{\end{equation}}

\newcommand{\xb}{x_{\mathbf{b}}}

\newcommand{\tb}{t_{\mathbf{b}}}

\newcommand{\X}{\mathbf{x}}
\newcommand{\V}{\mathbf{v}}
\newcommand{\T}{\Theta}

\newcommand{\bv}{\bar{v}}
\newcommand{\bx}{\bar{x}}
\newcommand{\tv}{\tilde{v}}
\newcommand{\tx}{\tilde{x}}

\newcommand{\Bes}{\begin{eqnarray*}}
\newcommand{\Ees}{\end{eqnarray*}}
\newcommand{\Be}{\begin{equation}}
\newcommand{\Ee}{\end{equation}}

 \numberwithin{equation}{section}
 \numberwithin{theorem}{section}

\def\p{\partial}

\def\O{\Omega}
\def\R{\mathbb{R}}

\def\B{\begin{equation}}
\def\E{\end{equation}}
\def\BN{\begin{eqnarray*}}
\def\EN{\end{eqnarray*}}

\usepackage{color}

\begin{document}

	\date{\today}
	
	\title
 {The Mixed Convex-Concave effect on the regularity of a Boltzmann solution}	
	\author{Gayoung An and Donghyun Lee}

	\begin{abstract} 
		The geometric properties of domains are well-known to be crucial factors influencing the regularity of Boltzmann boundary problems. In the presence of non-convex physical boundaries, highly singular behaviors are investigated including the propagation of discontinuities \cite{Kim11} and Hölder regularity \cite{CD2022}. In this paper, we focus on studying $C^{0,\frac{1}{4}-}$ H\"older type regularity (with singular weight) of the Boltzmann equation within concentric cylinders where both convex and concave features are present on the boundaries. Our findings extend the previous results of \cite{CD2022}, which primarily addressed concavity while considering the exterior of convex objects.
	\end{abstract}
	\maketitle
 \tableofcontents
 
 \section{Introduction}
 \subsection{Background and Main results}
 {\color{black}
The Boltzmann equation is a fundamental mathematical model for rarefied gases that accounts for collisions. It is expressed as a partial differential equation for the probability density function $F(t, x, v) \geq 0$, where $(t, x, v)$ represent time, space, and velocity, respectively. In terms of the probability density function $F(t, x, v)$, the Boltzmann equation is written as
\begin{equation} \label{Boltzmann}
		\p_{t}F + v\cdot\nabla_{x} F = Q(F,F), \ \   F(0,x,v)= F_{0}(x,v),
		\end{equation}
	where $Q(F,F)$ is nonlinear quadratic collisional operator. The collision operator is written by
	\begin{align}\label{Q}
	\begin{split}
	    Q(F_{1},F_{2})(v) &= \int_{\R^{3}}\int_{\S^{2}} B(v-u,\o) \left(F_{1}(u^{\prime}) F_{2}(v^{\prime}) - F_{1}(u) F_{2}(v)\right) d\o du \\
     &=Q_{\text{gain}}(F_{1}, F_{2})(v)
     -Q_{\text{loss}}(F_{1}, F_{2})(v).
	\end{split}
	\end{align} 
	(Here we abbreviate $(t,x)$ for notational convenience.) When two identical particle with velocities $u, v$ collide each other, we can describe post collisional velocities $u^{\prime}$ and $v^{\prime}$ as
 \begin{align*}
     u^{\prime} = u - ((v-u)\cdot \o)\o, \quad\text{and}\quad v^{\prime} = v + ((v-u)\cdot \o)\o
 \end{align*}
 for $\o\in\S^{2}$. We note that above expressions come from perfect elastic binary collision,
 \[
 	u+v = u'+v',\quad |u|^{2}+|v|^{2} = |u'|^{2}+|v'|^{2},
 \]
 which means conservation of momentum and kinetic energy, respectively. Meanwhile, there are many mathematical models for collision kernel $B(v-u,\o)$. In this paper, we consider hard sphere Boltzmann equation,
 \begin{align*}
     B(v-u,\o) = |(v-u) \cdot \o |.
 \end{align*} 
We note that our result might be easily adapted to the Boltzmann equation with hard potential or Maxwellian molecule, 
\[
	B(v-u,\o) = |v-u|^{\gamma}b(\theta),\quad 0 \leq \gamma < 1,\quad 0\leq b(\theta) \leq C|\cos\theta|
\] for a general constant $C>0$.  \\
	
Numerous articles have addressed the Boltzmann equation. In the seminal paper by Di Lion \cite{DiLion}, the renormalization solution of the Boltzmann equation was investigated. For the global well-posedness of the Boltzmann equation and related models, Guo developed a high-order energy method in the vicinity of equilibrium, as detailed in \cite{Guo_VPB, Guo_VMB, Guo_VPL}. In cases away from equilibrium, Desvillettes and Villani explored the convergence to equilibrium, assuming sufficiently smooth global solutions, as described in \cite{DV}. Additionally, we note the global well-posedness results presented in \cite{Strain} concerning the non-cutoff Boltzmann equation. \\
 
 Concerning general boundary condition problems, global well-posedness and regularity studies remained largely unknown until the work presented in \cite{Guo10}. We also refer \cite{SM2000, KS1978}. In this paper, an $L^{\infty}$ mild solution was proposed, and low regularity well-posedness was explored using a novel $L^{2}-L^{\infty}$ bootstrap argument. Subsequently, there has been widespread research on low regularity mild solution approaches. Collaborating with C. Kim, the second author in \cite{CD2018} extended the results of \cite{Guo10} from generalized analytic convex domains to general $C^{3}$ convex domains with specular boundary conditions. Moreover, the strong convexity assumptions on the boundary were relaxed to include some general non-convex domains in \cite{CD20182, KCL3D}. 
 
For the non-cutoff Boltzmann equation, \cite{AMST, DSS} examined low regularity solutions in the presence of boundary conditions. It is also noteworthy to mention the large amplitude problem, where the small $L^\infty$ initial conditions in the aforementioned works were replaced with small $L^p$-type initial data, allowing for arbitrarily large $L^\infty$ amplitudes. For cases involving torus and whole space, refer to \cite{DHWY}, and for boundary condition problems, see \cite{DKL2020, DW}. Regarding the BGK equation, which is an essential relaxed model for the Boltzmann equation, recent research on large amplitude solutions can be found in \cite{BGK2023}.


The solutions presented in the aforementioned works are all considered to be low regularity mild solutions. It is widely believed that a solution to general Boltzmann boundary problems does not possess high-order regularity unless the boundary is flat. However, investigating the regularity of the Boltzmann solution in bounded domains is an extremely challenging research area.

In \cite{YC2017}, Guo et al. examined the Boltzmann equation in a strictly convex domain with specular, bounce-back, and diffuse boundary conditions. They demonstrated that the first-order derivatives remain continuous away from the grazing set, where the backward-in-time linear trajectory tangentially intersects the boundary. In the context of non-convex domains, reference can be made to BV regularity results presented in \cite{GKTT2} with diffuse reflection boundary conditions. In such cases, the convex geometry of the boundary plays a crucial role because the characteristics are differentiable within its domain. \\


In the context of the Boltzmann equation, the study of regularity in non-convex domains under specular reflection (billiard-like boundary conditions) had remained an open problem due to its challenging singularity issues associated with grazing collisions of trajectories (or billiard maps). However, in \cite{CD2022}, Kim and Lee addressed the regularity of a Boltzmann solution in certain non-convex domains with specular reflection conditions. They established that a local-in-time solution outside a general convex object exhibits $C^{0,\frac{1}{2}-}_{x,v}$ regularity by investigating the averaging of grazing singularities in combination with the shift method. \\


The objective of our study in this paper is to investigate the regularity of a Boltzmann solution \eqref{Boltzmann} within a cylindrical annulus domain under the specular reflection boundary condition. In the prior work \cite{CD2022}, the domain considered was the outside of a convex obstacle, which, in geometrical terms, can be described as purely concave. Furthermore, when analyzing a billiard map in phase space, it is worth noting that at most one bounce can occur unless there is an external force.

In the context described above, the regularity result in \cite{CD2022} predominantly arises from the pure concavity of the domain. However, what happens when convexity is introduced into the domain? As demonstrated in \cite{YC2017, GD2022}, convexity ensures the differentiability of the billiard map, but the billiard map's escape behavior near the grazing regime is significantly slower compared to the concavity case. (It is important to mention that the rapid escape property of concavity is one of the key elements in \cite{CD2022}.) Therefore, it becomes a natural and intriguing topic to investigate the regularity of a Boltzmann solution when a domain contains both convex and concave features. 

For a general domain with both convexity and concavity, achieving precise derivative estimates appears to be a formidable task. Our chosen domain, `the concentric cylinder', is a simple yet nontrivial domain of significance that encompasses both convex and concave elements. While the analysis of derivatives of the given billiard map may require substantial effort, we will leverage some valuable results from \cite{YC2017, GD2022} regarding characteristics to aid in our study.  \\


First, we define concentric cylindrical annulus domain $\O$ and mild formulation of the Boltzmann equation as well as specular reflection boundary condition. 

\begin{definition}[cylindrical annulus domain] \label{def:domain}
	We define a cylindrical annulus domain $\Omega$,
	\begin{align}
	\Omega = \left\{(\rho \cos \theta, \rho \sin \theta,z)| r< \rho < R,\; -\pi \leq \theta < \pi,\; -\infty<z<+\infty \right\}, \notag
	\end{align}
	for some positive constants $R, r \in \mathbb{R}^{+}$. For convenience, we assume that $R-r > 1$, WLOG. Since the cross-section is circles, outward unit normal vector is defined by 
	\begin{align} \label{def:normal}
	n(x_1,x_2,x_3) = 
	\begin{cases}
	\frac{x_p}{|x|_p}
	\quad \text{for} \quad |x|_p=R,  \\
	-\frac{x_p}{|x|_p}
	\quad \text{for} \quad |x|_p=r.
	\end{cases}
	\end{align} 
	Note that we have used $\langle v \rangle = \sqrt{1+|v|^2},  v_p=(v_1,v_2,0), |v|_p = |v_p| = \sqrt{v_1^2+v_2^2}$, $v_{ver}=(0,0,v_3)$, and 
	\[
	\hat{v} = \begin{cases}
	\frac{v}{|v|},\quad v\neq 0, \\ \;0,\;\quad v=0
	\end{cases}
	\]
	for $v=(v_1,v_2,v_3)  \in \mathbb{R}^3$. \\
\end{definition} 

We impose specular reflection boundary condition on the boundary $\p\O$,
\begin{equation} \label{specular}
F(t,x,v) = F(t,x,R_{x}v) \  \textit{ for } x\in\p\O, \  \textit{ where } \ R_x v = v- 2 (n(x)\cdot v)n(x).
\end{equation}

Let us define global Maxwellian $\mu(v) = e^{-  \frac{|v|^2}{2} }$. We will rewrite the Boltzmann equation \eqref{Boltzmann} and the specular reflection boundary condition \eqref{specular} in terms of $f(t,x,v) = F(t,x,v)/\sqrt{\mu}$. \eqref{Boltzmann} and \eqref{specular} are written as
\Be\label{f_eqtn} 
\begin{split}
& \p_{t}f + v\cdot \nabla_{x} f  
= \Gamma(f,f) := \Gamma_{\text{gain}}(f,f) - \nu(f) f ,  
\ \ 
f|_{t=0} =f_0,  \\
& f(t,x,v)  = f(t,x,R_{x}v), \quad x\in \p\O,
\end{split}
\Ee
where
\begin{align}
&\Gamma_{\text{gain}}(f_1,f_2) = \frac{1}{\sqrt{\mu}}Q_{\text{gain}}(\sqrt{\mu}f_1,\sqrt{\mu}f_2),\label{Gamma}\\
&\nu(f) = \iint_{\mathbb{R}^{3}\times\mathbb{S}^{2}} |(v-u)\cdot\omega| \sqrt{\mu(u)} f(t,x,u) d\omega du. \label{nu}
\end{align}

\noindent Now, let us define mild solution of \eqref{f_eqtn}. We introduce characteristic $(X(s;t,x,v), V(s;t,x,v))$ which solves
%
\begin{equation}\label{E_Ham} 
\frac{d}{ds} (X(s;t,x,v), V(s;t,x,v)  )   =    (V(s;t,x,v), 0),  \ \  \ 
(X(s;t,x,v)  , V (s;t,x,v)  )|_{s=t}=(x,v).
\end{equation} 
Taking specular boundary condition \eqref{specular} on $\p\O$ into account, the explicit definition of $(X(s;t,x,v),V(s;t,x,v))$ will be given in Lemma \ref{lemma:traject} and Definition \ref{def:traject}. Along the characteristic $(X(s;t,x,v), V(s;t,x,v))$, we define mild solution of \eqref{f_eqtn} as the following,
\Be \label{f_expan}
\begin{split} 
	f(t,x,v)
	&=  
	e^{- \int^t_ 0 \nu(f) (\tau, X(\tau;t,x,v), V(\tau;t,x,v)) d \tau}
	f(0,X(0;t,x,v), V(0;t,x,v))\\
	& \ \ + \int^t_0
	e^{- \int^t_ s \nu(f) (\tau, X(\tau;t,x,v), V(\tau;t,x,v)) d \tau}
	\Gamma_{\text{gain}}(f,f)(s,X(s;t,x,v), V(s;t,x,v)) ds. 
\end{split}
\Ee 

\noindent We explain the Local in time existence of the mild solution \eqref{f_expan} in Lemma \ref{lem loc}.

 \begin{lemma}[Local existence] \label{lem loc}
	 Assume initial data $f_{0}$ satisfies $\|w_{0}(v)f_{0}\|_{\infty} := \|e^{\vartheta_{0}|v|^{2}}f_{0}\|_{\infty} < \infty$ for $0 < \vartheta_{0} < \frac{1}{4}$ and initial compatibility condition \eqref{specular}. Then there exists $T^{*} > 0$ and a local in time unique solution $f(t,x,v)$ of (\ref{f_expan}) for $0\leq t \leq T^{*}$. Moreover, the solution satisfies
	\begin{equation*}
	\sup_{0\leq s \leq T^{*}} \|w(v) f(s)\|_{\infty} :=
	\sup_{0\leq s \leq T^{*}} \|e^{\vartheta |v|^{2}}f(s)\|_{\infty} \lesssim \|e^{\vartheta_{0} |v|^{2}}f_{0}\|_{\infty},
	\end{equation*}
	for some $0 < \vartheta < \vartheta_{0}$. 
\end{lemma}
\begin{proof}
We refer \cite{YC2017}.	 
\end{proof}

Before introducing the Main Theorem, we define weight $h(y_p,v_p)$ 
\begin{align} \label{define:h}
	h(y_p,v_p) = \sqrt{1-\frac{|y|_p^2}{R^2}+\frac{(y_p \cdot \hat{v}_p)^2}{R^2} }>0,
\end{align} 
where $y \in \O, v \in \mathbb{R}^3$. The following main theorem states $C^{0,\frac{1}{4}-}$ H\"older type regularity (with singular weight) of the Boltzmann solution in concentric cylinder.

\begin{theorem}\label{theo:Holder}
	Suppose the domain $\O$ is given as in Definition \ref{def:domain}. Assume initial data $f_0$ satisfies compatibility condition \eqref{specular}, $\|e^{\vartheta_{0}|v|^{2}} f_0 \|_\infty< \infty$ for $0< \vartheta_{0} < \frac{1}{4}$, and  
 \begin{align} \label{thm_initial}
		\sup_{ \substack{ v\in \R^{3} \\ 0 < |x - \bx|\leq 1   } }
		\langle v \rangle  \frac{|f_{0}( x, v ) - f_{0}(\bx, v)|}{|x - \bx|^{2\b}}  
		+ \sup_{ \substack{ x\in \O \\ 0<|v - \bv|\leq 1  } }  \langle v \rangle^{2} \frac{|f_{0}( x, v ) - f_{0}( x, \bv)|}{|v - \bv|^{2\b}}  
		< \infty,\quad 0<\b < \frac{1}{4}.
 \end{align}
	Then there exists $0< T \ll 1$ such that we have a unique mild solution $f(t,x,v)$ of \eqref{f_eqtn} for $0 \leq t \leq T$ with $\sup_{0 \leq t \leq T}\| e^{\vartheta|v|^2} f (t) \|_\infty\lesssim \| e^{\vartheta_{0}|v|^2} f_0 \|_\infty$ for some $0\leq \vartheta <\vartheta_{0}$. Moreover $f(t,x,v)$ is H\"older continuous in the following sense:
		\begin{align} \label{est:Holder}
		\begin{split}
		&\sup_{0\leq t \leq T} \sup_{ \substack{ v\in \R^{3} \\ 0 < |x - \bx|\leq 1   } }
		\left| \max \left\{h^{2\b}(x_p,v_p), h^{2\b}(\bx_p,v_p)\right\} 
		\langle v_p \rangle^{-4\b} e^{-\varpi\langle v \rangle^{2}t}\frac{ | f(t,x,v) - f(t, \bx, v) | }{ |x-\bx|^{\b} }
		\right|    \\
           &\quad + \sup_{0\leq t \leq T} \sup_{ \substack{ x\in \O \\ 0<|v - \bv|\leq 1   } }
		\left| \min \left\{h^{\b}(x_p,v_p), h^{\b}(x_p,\bv_p) \right\} |v|_p^{2\b}
		\langle v_p \rangle^{-2\b} e^{-\varpi\langle v \rangle^{2}t}\frac{ | f(t,x,v) - f(t, x, \bv) | }{ |v-\bv|^{\b} }
		\right| \\
		&\lesssim_{\b} 
		\|e^{\vartheta_{0}|v|^{2}} f_{0}\|_{\infty}
		\left[
		\sup_{ \substack{ v\in \R^{3} \\ 0 < |x - \bx|\leq 1   } }
		\langle v \rangle  \frac{|f_{0}( x, v ) - f_{0}(\bx, v)|}{|x - \bx|^{2\b}}  
		+ \sup_{ \substack{ x\in \O \\ 0<|v - \bv|\leq 1   } }  \langle v \rangle^{2} \frac{|f_{0}( x, v ) - f_{0}( x, \bv)|}{|v - \bv|^{2\b}}  
		\right]
		+ \mathcal{P}_{3}(\|e^{\vartheta_{0}|v|^{2}} f_{0}\|_{\infty})
		\end{split}
		\end{align}
		for some $\varpi \gg_{f_0, \b} 1$, where $\mathcal{P}_{3}(s) = 1 + |s| + |s|^2+|s|^3$. 
\end{theorem}

H\"older type terms in LHS contain some singular weights such as 
\[
	\max \left\{h^{2\b}(x_p,v_p), h^{2\b}(\bx_p,v_p) \right\}	\langle v_p \rangle^{-4\b} e^{-\varpi\langle v \rangle^{2}t}
\] 
on LHS. We note that these weights are not purely tehnical one, but have some intrinsic physical implications as explained in the following remarks.

\begin{remark}
Assume that the particle, located at position $x$ with velocity $v$, hits the outer circle. Then $h(x_p,v_p)$ is the absolute value of the inner product between the unit normal vector at point $X_0(x_p,v_p)$ and the velocity $\hat{v}_p$, where $X_0(x_p,v_p)$  means the first forward in time bouncing point staring from $(x_p, v_p)$ and $|X_0(x_p,v_p)|=R$. (See Section  2 for the definition.) If the particle (or linear trajectory) hits the inner circle at least once, it holds that $h(x_p,v_p) \geq \sqrt{1-(r/R)^2}$.  
\end{remark}
\begin{remark}
	Singular weight $\max\{h^{2\b}(x_p,v_p), h^{2\b}(\bx_p,v_p)\}$ on RHS of \eqref{est:Holder} is very similar one as the singluar weight which appears in \cite{YC2017} for specular boundary condition case. Since the concenctric cylinder allows some characteristics to have consecutive bouncing near grazing regime $\O$, singular weight in terms of normal velocity on the boundary is natural. To understand this phenomenon, imagine a billliard particle (under specular reflection condition) hits the inner circle, by non-convexity, we observe that
    \begin{align*}
    	\sup_{(x,v)\neq (\bar{x}, \bar{v})}\frac{|X(s;t,\bx,\bv)-X(s;t,x,v)|}{|(\bx,\bv)-(x,v)|} = \infty,\quad \sup_{(x,v)\neq (\bar{x}, \bar{v})} \frac{|V(s;t,\bx,\bv)-V(s;t,x,v)|}{|(\bx,\bv)-(x,v)|} = \infty.
    \end{align*} 
    Instead, we have H\"older regularity
    \begin{align*}
         \sup_{(x,v)\neq (\bar{x}, \bar{v})} \frac{|X(s;t,\bx,\bv)-X(s;t,x,v)|}{|(\bx,\bv)-(x,v)|^{1/2}} \lesssim_{t,x,v, \bar{v}} 1,\quad \sup_{(x,v)\neq (\bar{x}, \bar{v})} \frac{|V(s;t,\bx,\bv)-V(s;t,x,v)|}{|(\bx,\bv)-(x,v)|^{1/2}} \lesssim_{t,x,v, \bar{v}} 1.
    \end{align*} 
    This is why we $C^{0,\beta}$ regularity in \eqref{est:Holder}, instead of $C^{0,2\beta}$ (See Section 6 for estimates of $X$ and $V$.) 
   This phenomenon also affects the estimates of the charactersitics which travel along outer part of $\O$ which is uniformly convex. In fact, such characteristics are differentiable (\cite{GD2022, YC2017}), but from the above observation, we should perform only $C^{0,\frac{1}{2}}$ H\"older estimates : 
    \begin{align*}
           &\frac{|X(s;t,\bx,\bv)-X(s;t,x,v)|}{|(\bx,\bv)-(x,v)|^{1/2}} \lesssim \frac{1}{n(X_0(x_p,v_p)) \cdot \hat{v}_p} = \frac{1}{h(x_p,v_p)} <\infty,
           \\
           &\frac{|V(s;t,\bx,\bv)-V(s;t,x,v)|}{|(\bx,\bv)-(x,v)|^{1/2}} \lesssim \frac{1}{(n(X_0(x_p,v_p)) \cdot \hat{v}_p)^{1/2}} =  \frac{1}{h^{1/2}(x_p,v_p)} < \infty.
    \end{align*}
    In conclusion, denominators of RHS above appear as singular weights in \eqref{est:Holder}. 
\end{remark}
\begin{remark}
	Velocity type weights (such as $\langle v_p \rangle^{-4\b} e^{-\varpi\langle v \rangle^{2}t}$) in LHS is very similar one as in \cite{CD2022}. This is also very natural and essential since derivative (or quotient) estimates for characteristic $X(s;t,x,v)$ and $V(s;t,x,v)$ must have growth in $|v|$ and travel length $|v|(t-s)$ under specular reflection effect. See Lemma \ref{disk} and Lemma \ref{x-tx:3} for details. \\
\end{remark}

\subsection{Preliminaries} 

\begin{definition} \label{def_k_c}
	For $c > 0$, we define
	\Be \notag
	\begin{split}
		k_{c}(v, v+\zeta) = \frac{1}{|\zeta|}e^{ - c|\zeta|^{2} - c \frac{ | |v|^{2}-|v+\zeta|^{2} |^{2} }{|\zeta|^{2}} },
	\end{split}
	\Ee
	and 
	\Be \notag
	\mathbf{k}_{c} (v, \bar v, \zeta) = k_{c}(v, v+\zeta) + k_{c}(\bv, \bv+\zeta).   \\
	\Ee
\end{definition}
It holds that
\Be \notag 
		k_{c}(v, v+\zeta) = k_{c}(R_{x}v, R_{x}v + R_{x}\zeta),\quad 	\mathbf{k}_{c} (v, \bar v, \zeta) = \mathbf{k}_{c} (R_{x}v, R_{x}\bar v, R_{x}\zeta),\quad x\in\p\O. 
\Ee
The following lemma summarizes some estimates of $\Gamma_{\text{gain}}(f, f)(t,x,v)$ and $\nu(f)(t, x, v)$ known in \cite{CD2022}.

 \begin{lemma} \label{lem_Gamma}
		Let $w(v) = e^{\vartheta|v|^{2}}$ for $0<\vartheta < \frac{1}{4}$ and 
		$x, \bx \in \O$, $v, \bv, \zeta \in\R^{3}$. For $|(x,v) - (\bx, \bv)| \leq 1$, we have the following estimates
		\Be \label{full k v}
		\begin{split}
			 \frac{| \Gamma_{\text{gain}}(f, f)(x,v) - \Gamma_{\text{gain}}(f, f)(x,\bar{v}) |}{|v-\bar{v}|^{2\b}}   
			\lesssim & \ \|wf\|_{\infty} \int_{\R^3} 
			\mathbf{k}_{c}(v, \bar v, u)
			\frac{ | f(x, v+u) - f(x, \bar{v}+u) | }{ |v-\bar{v}|^{2\b} }  du \\
			&
			+ \|wf\|_{\infty}^{2}  \min \left\{	 \langle v\rangle^{-1}	, \langle \bar v\rangle^{-1}		  \right \},
		\end{split}
		\Ee
	\Be \label{full k x} 
			 \frac{| \Gamma_{\text{gain}}(f , f )(x,v) - \Gamma_{\text{gain}}(f , f )(\bar{x},v) |}{|x-\bar{x}|^{2\b}}   
			 \lesssim \|wf \|_{\infty} \int_{\R^3} 	k_{c}(v, v+u)\frac{ |f (x, v+u) - f (\bar{x}, v+u)| }{ |x-\bar{x}|^{2\b} } du, 
		\Ee
		for some $c>0$, where $ \Gamma_{\text{gain}}(f,f)$ in \eqref{Gamma}. We have
  	\Be \label{full nu v}
		\begin{split}
			|\nu(f)(t, x, v) - \nu(f)(t, x, \bv)|  
			&\lesssim |v - \bv| \|f(t)\|_{\infty},
		\end{split}
		\Ee
		and
		\Be \label{full nu x} 
			|\nu(f)(t, x, v) - \nu(f)(t, \bx, v)|  
			\lesssim \langle v \rangle \int_{\R^{3}} k_{c}(0,u) |f(t, x, u) - f(t, \bx, u)| du,   \\
		\Ee for some $c>0$, where $\nu(f)$ in \eqref{nu}. Moreover, $\Gamma_\text{gain}(f,f)$ and $\nu(f)$ satisfy the specular reflection boundary condition (\ref{specular}),
		\Be \label{specular Gamma}
		\Gamma_{\text{gain}}(f,f)(t,x,v) = \Gamma_{\text{gain}}(f,f)(t,x,R_{x}v), \quad 
  	\nu(f)(t,x,v) = \nu(f)(t,x,R_{x}v)
		\Ee
		on $x\in\p\O$.
	\end{lemma}
\begin{proof}
    We refer Lemma 3.2, Lemma 3.3, and Lemma 3.5 of \cite{CD2022}.
\end{proof}

 	\begin{lemma}\label{lem_nega}
		 For $|\varpi s| < c$,
		\Be \label{ws nega}
		e^{-\varpi(1 + |v|^{2})s} e^{\varpi(1 + |v+\zeta|^{2})s} k_{c}(v, v+\zeta)  \leq k_{\frac{c}{2}}(v, v+\zeta).
		\Ee
		When $|v-\bar{v}| \leq 1$ and $\varpi s < (\sqrt{20}-4)\frac{c}{2}$,
		\Be \label{ws nega bar}
		e^{ -\varpi(1+|\bar{v}|^{2})s + \varpi(1+|\bar{v}+\zeta|^{2})s } k_{c}(v,v+\zeta) \lesssim k_{\frac{c}{2}}(v,v+\zeta).
		\Ee
		Moreover, from \eqref{ws nega} and \eqref{ws nega bar},
		\Be \label{bf nega}
		\begin{split}
			e^{ -\varpi(1+|v|^{2})s + \varpi(1+|v+\zeta|^{2})s } \mathbf{k}_{c}(v, \bv, \zeta) &\lesssim \mathbf{k}_{\frac{c}{2}}(v, \bv, \zeta), \\
			e^{ -\varpi(1+|\bv|^{2})s + \varpi(1+|\bv+\zeta|^{2})s } \mathbf{k}_{c}(v, \bv, \zeta) &\lesssim \mathbf{k}_{\frac{c}{2}}(v, \bv, \zeta),
		\end{split}
		\Ee
		hold when $|v-\bar{v}| \leq 1$ and $\varpi s < (\sqrt{20}-4)\frac{c}{2}$.
	\end{lemma}
 \begin{proof}
    We refer Lemma 3.4 of \cite{CD2022}.
\end{proof}

To define characteristic $(X(s;t,x,v), V(s;t,x,v))$ in a cylindrical annulus domain precisely, we define the following concepts.  
\begin{definition} \label{def:time} 
Given $(x,v) \in \O \times \mathbb{R}^3$, we define
\begin{align}
    &t_b(x,v) = \sup \big\{  s \geq 0 :  x-\tau v  \in \O  \ \ \text{for all}  \ \tau \in ( 0, s )  \big\}, \label{def:tb}\\
    &t_f(x,v) = \sup \big\{  s \geq 0 :  x+\tau v  \in \O  \ \ \text{for all}  \ \tau \in ( 0, s )  \big\}, \label{def:tf}
\end{align} 
\begin{align}\label{def:t*,l}
   t_*(x,v) = t_b(x,v) + t_f(x,v)\quad \text{and} \quad  l(x,v) = |v|_pt_*(x,v).
\end{align} 
We define
\begin{align} \label{def:X_0,1}
    X_0(x,v) = x+t_f(x,v)v,
    \quad X_1(x,v) = x-t_b(x,v)v, \quad V_0(x,v) = v.
\end{align}
\end{definition} 

\begin{definition}\label{defi:C123}
    We define 
\begin{align} \label{Case1,2 define}
\begin{split}
    &(x,v) \in \text{$\mathcal{C}_1$} \quad \text{if} \quad |X_0(x,v)|_p= R, \; |X_1(x,v)|_p= r,  \\
    &(x,v) \in \text{$\mathcal{C}_2$} \quad \text{if} \quad |X_0(x,v)|_p= r, \; |X_1(x,v)|_p= R,
\end{split}
\end{align} and
\begin{align} \label{case3 define}
      &(x,v) \in \text{$\mathcal{C}_3$} \quad \text{if} \quad |X_0(x,v)|_p= |X_1(x,v)|_p= R. 
\end{align} 
When $(x,v) \in \mathcal{C}_1$, we define
\begin{align} \label{def:angle:1}
    a(x,v) = \cos^{-1} \left(\widehat{X_0(x,v)_p}\cdot \hat{v}_p \right) \in [0,\pi/2) \text{\quad and \quad}
    b(x,v) = \cos^{-1}\left(\widehat{X_1(x,v)_p} \cdot \hat{v}_p \right)\in [0,\pi/2].
\end{align} When $(x,v) \in \mathcal{C}_2$, we define
\begin{align} \label{def:angle:2}
    a(x,v) = \cos^{-1}\left(\widehat{-X_1(x,v)_p}\cdot \hat{v}_p \right)\in [0,\pi/2) \text{\quad and \quad}
    b(x,v) = \cos^{-1}\left (\widehat{-X_0(x,v)_p} \cdot \hat{v}_p \right )\in [0,\pi/2].
\end{align} When $(x,v) \in \mathcal{C}_3$, we define
\begin{align} \label{def:angle:3}
    a(x,v) = \cos^{-1} \left (\widehat{X_0(x,v)_p}\cdot \hat{v}_p \right )\in [0,\pi/2). 
\end{align} (See Figure \ref{case2,3}). Also, we define
\begin{align}\label{def:graze}
     (x,v) \in \text{$\mathcal{G}$} \quad \text{if} \quad  X_1(x,v)_p \; \bot \; v_p \quad \text{or} \quad X_0(x,v)_p \; \bot \; v_p.
\end{align}
\end{definition} 

\begin{figure} [t] 
\centering
\includegraphics[width=11cm]{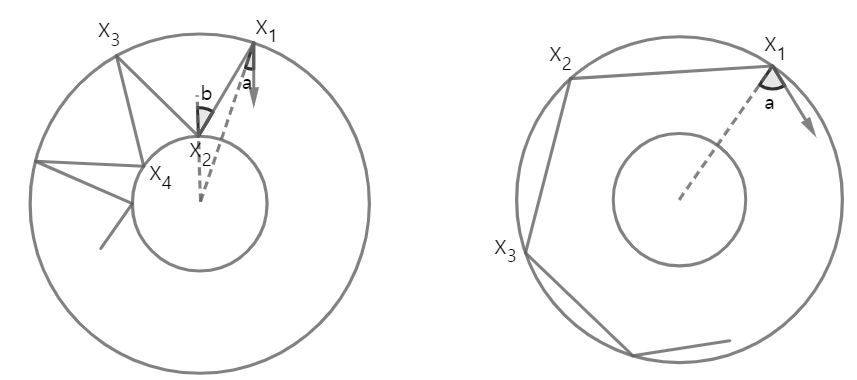}
\caption{$\mathcal{C}_2$ and $\mathcal{C}_3$.} \label{case2,3}
\end{figure}

\vspace{3mm}
When considering the difference between two trajectories, we will divide the difference into singular part and nonsingular parts. In this step, we need to define \textit{shifted} position and velocity, respectively.

\begin{lemma}\label{range of c}
    Let $U =\{(x_1,x_2)\in \mathbb{R}^2 :r < x^2_1+x^2_2 < R \}$. For given $p, q \in U$and $v \in \mathbb{R}^2$, we let
    \begin{align*}
        \tilde{p} = p + ((q-p)\cdot \hat{v}) \hat{v} \quad \text{and} \quad
        \tilde{q} = q + ((p-q)\cdot \hat{v}) \hat{v},
    \end{align*} and assume that $|p-q| < R-r$. Then it holds that $\tilde{p} \in U$ or $\tilde{q} \in U$.
\end{lemma}
\begin{proof}  Let $O_r =\left\{(x_1,x_2)\in \mathbb{R}^2 \big| x^2_1+x^2_2 \leq r \right\}$ and $O^{c}_R =\left\{(x_1,x_2)\in \mathbb{R}^2 \big| x^2_1+x^2_2  \geq R  \right\}$. We fix $p,q \in U$.
    If $\tilde{p},\tilde{q} \in O_r$ or $\tilde{p},\tilde{q} \in O^{c}_R$, then $p \notin U$ or $q \notin U$. Moreover, if $\tilde{p} \in O^{c}_R, \tilde{q} \in O_r$ or $\tilde{p} \in O_r, \tilde{q} \in O^{c}_R$, then $|\tilde{p}-\tilde{q}|=|p-q| \geq R-r$. It is a contradiction, and we conclude $\tilde{p} \in U$ or $\tilde{q} \in U$.\\
\end{proof}

\begin{definition} \label{def:para}
    We define shifted position $\tx(x,\bx,v)$ and velocity $\tv = \tv(v,\bv,\zeta),$ respectively. \\
    (i) Fix $x, \bx \in \O, v \in \mathbb{R}^3$, where $|x-\bx| \leq 1$. Let us  assume
	\begin{equation} \label{assume_x}
		 \text{ $x_p-\bx_p \neq 0$ is neither parallel nor anti-parallel to $v_p\neq 0$,  i.e., $(x_p-\bx_p)\cdot v_p \neq \pm |x_p-\bx_p||v|_p$. } \\
	\end{equation}
	In this case, we define shifted $\tilde{x}$ as
	\begin{align} \label{def:tx} 
		\tilde{x}_p =  \tilde{x}(x,\bx,v)_p= \bar{x}_p + \left( (x_p-\bar{x}_p)\cdot \frac{v_p}{|v|_p} \right) \frac{v_p}{|v|_p},
     \quad \tx_{ver}=x_{ver}.
	\end{align}
    We define 
    \begin{align} \label{def:x_para}
        \X(\tau) = \X(\tau;x,\bx,v)=(1-\tau){\tx}+\tau x, \quad \dot{\X}(\tau)=x_p-\tx_p
    \end{align} and $\X(0)=\tx, \; \X(1)=x.$ If $\tx(x,\bx,v)$ is not in $\O$, we rename $x$ as $\bx$, and $\bx$ as $x$. Then there exists $\tx(x,\bx,v) \in \O$ by Lemma \ref{range of c}.
    \\ (ii) For fixed $v, \bv, \zeta$, let us  assume
\begin{equation} \label{assume_v}
    \text{ $(v+\zeta)_p\neq 0$ is neither parallel nor anti-parallel to $(\bv+\zeta)_p\neq 0$, }
\end{equation} i.e., $ (v+\zeta)_p\cdot(\bv+\zeta)_p\neq \pm|v+\zeta|_p||\bv+\zeta|_p.$
	In this case, we define shifted $\tilde{v}$ as
	\begin{align} \label{def:tv} 
	(\tilde{v} + \zeta)_p = \left(\tilde{v}(v, \bv, \zeta) + \zeta \right)_p= |v+\zeta|_p \frac{(\bv+\zeta)_p}{|\bv+\zeta|_p},
    \quad (\tv+\zeta)_{ver}=(v+\zeta)_{ver}.
	\end{align} We define
     \begin{align} \label{def:v_para}
     	&\V(\tau) =\V(\tau; x, v , \bar v, \zeta) 
				=|v+\zeta|_p R_{(v,\bv,\zeta)} 
			\begin{bmatrix}
				\cos\T(\tau) \\ \sin\T(\tau) \\ 0
			\end{bmatrix}+(v+\zeta)_{ver}, \\
			&\T(\tau) = \tau\theta, \quad \dot{\T}(\tau) = \dot{\T} = \theta, \notag
    \end{align}  where $\theta = 
	   \cos^{-1} \left(\widehat{(v+\zeta)_p} \cdot \widehat{(\bv+\zeta)_p} \right)
	   \in[0,2\pi)$ is the angle between $(v+\zeta)_p$ and $(\bv+\zeta)_p$, and 
    \begin{align*}
			R_{(v,\bv,\zeta)} := 
			\begin{bmatrix}
				 & & \\
				\widehat{(\bv+\zeta)_p} & \widehat{(v+\zeta)_p} &  \widehat{(\bv+\zeta)_p}\times  \widehat{(v+\zeta)_p} \\
				 & & \\
			\end{bmatrix}
			\begin{bmatrix}
				1 & \cos\theta & 0 \\
				0 & \sin\theta & 0 \\
				0 & 0 & 1
			\end{bmatrix}^{-1}
    \end{align*}  
	at
    \begin{align*}
        \T(0) =0, \;  \T(1) =\theta \quad \text{and} \quad \V(0)=\tv+\zeta, \; \V(1)=v+\zeta.
    \end{align*}
\end{definition} 

To define \eqref{def:tx}, we assume that $|x-\bx| < R-r$, and this assumption is derived from Lemma \ref{range of c}. For convenience, we assume that $R-r > 1$ in Definition \ref{def:domain}.\\

Now, we consider $\tx \in \O$ and $v \in \mathbb{R}^3$ which satisfy $\tx(x,\tx,v) =\tx$ and $\tv(v,\tv,\zeta)=\tv$ in \eqref{def:tx}, \eqref{def:tv} for given $x \in \O$ and $v, \zeta \in \mathbb{R}^3$. We define the sets A and B, which consist of elements $(x,\tx,v)$, where $(\X(\tau),v) \in  \mathcal{C}_1$ for all $\tau \in [0,1]$ and $(\X(\tau),v) \in  \mathcal{C}_2$ for all $\tau \in [0,1]$, respectively. Here, A and B do not include all $(x,\tx,v) \in  \O \times  \O \times \mathbb{R}^3$. We also define the important quantities $\mathfrak{S}_{sp}(\tau; x, \tx, v)$ and $\mathfrak{S}_{vel}(\tau; x, v, \tilde v, \zeta)$, which are used in section 3.

\begin{definition} \label{def:g} 
For $\X(\tau)$ in \eqref{def:x_para}, we define 
\begin{align*}
  &A = \left \{(x,\tx,v) \big|(\X (\tau),v) \in \mathcal{C}_1  \text{\quad for\quad} \tau \in [0,1] \right\}, \\  
  &B = \left \{(x,\tx,v) \big|(\X (\tau),v) \in \mathcal{C}_2  \text{\quad for\quad} \tau \in [0,1] \right\}.
\end{align*}
We recall \eqref{def:X_0,1}, and define 
\begin{align*}
\frac{1}{\mathfrak{S}^{1}_{sp}(\tau; x, \tx, v)} =& \left|\frac{X_0(\X(\tau),v)_p\cdot \frac{\dot{\X}(\tau)}{|\dot{\X}(\tau)|}}{X_0(\X(\tau),v)_p \cdot v}\right| \mathbf{1}_{A}
+\left|\frac{X_1(\X(\tau),v)_p\cdot \frac{\dot{\X}(\tau)}{|\dot{\X}(\tau)|}}{X_1(\X(\tau),v)_p \cdot v}\right|\mathbf{1}_{B},\\
\frac{1}{\mathfrak{S}^{2}_{sp}(\tau; x, \tx, v)} =& \left|\frac{X_1(\X(\tau),v)_p\cdot \frac{\dot{\X}(\tau)}{|\dot{\X}(\tau)|}}{X_1(\X(\tau),v)_p \cdot v}\right|\mathbf{1}_{A} +\left|\frac{X_0(\X(\tau),v)_p\cdot \frac{\dot{\X}(\tau)}{|\dot{\X}(\tau)|}}{X_0(\X(\tau),v)_p \cdot v}\right|\mathbf{1}_{B},
\end{align*} and
\begin{align*}
    \frac{1}{\mathfrak{S}_{sp}(\tau; x, \tx, v)}
    =\frac{1}{\mathfrak{S}^{1}_{sp}(\tau; x, \tx, v)}+\frac{1}{\mathfrak{S}^{2}_{sp}(\tau; x, \tx, v)}.
\end{align*} 
For $\V(\tau)$ in \eqref{def:v_para}, we define
\begin{align*}
  &D = \left \{(x,v,\tv,\zeta) \big|(x,\V (\tau)) \in \mathcal{C}_1  \text{\quad for\quad} \tau \in [0,1] \right\}, \\  
  &E = \left \{(x,v,\tv,\zeta) \big| (x,\V (\tau)) \in \mathcal{C}_2  \text{\quad for\quad} \tau \in [0,1] \right\}.
\end{align*}
We recall \eqref{def:tb}, \eqref{def:tf}, and define 
\begin{align*} 
    \frac{1}{\mathfrak{S}_{vel}^{1}(\tau; x, v, \tilde v, \zeta )}&= t_f(x,\V(\tau))
    \left|\frac{X_0(x,\V(\tau))_p\cdot \frac{\dot{\V}(\tau)}{|\dot{\V}(\tau)|}}{X_0(x,\V(\tau))_p \cdot \V(\tau)}\right| \mathbf{1}_{D} 
    +t_b(x,\V(\tau))
    \left|\frac{X_1(x,\V(\tau))_p\cdot \frac{\dot{\V}(\tau)}{|\dot{\V}(\tau)|}}{X_1(x,\V(\tau))_p \cdot \V(\tau)}\right| \mathbf{1}_{E}, \\
    \frac{1}{\mathfrak{S}_{vel}^{2}(\tau; x, v, \tilde v, \zeta )}&= 
    t_b(x,\V(\tau))
    \left|\frac{X_1(x,\V(\tau))_p\cdot \frac{\dot{\V}(\tau)}{|\dot{\V}(\tau)|}}{X_1(x,\V(\tau))_p \cdot \V(\tau)}\right| \mathbf{1}_{D} +t_f(x,\V(\tau))
    \left|\frac{X_0(x,\V(\tau))_p\cdot \frac{\dot{\V}(\tau)}{|\dot{\V}(\tau)|}}{X_0(x,\V(\tau))_p \cdot \V(\tau)}\right| \mathbf{1}_{E}, 
\end{align*} and
\begin{align*}
      \frac{1}{\mathfrak{S}_{vel}(\tau; x, v, \tilde v, \zeta )}
      =  \frac{1}{\mathfrak{S}_{vel}^{1}(\tau; x, v, \tilde v, \zeta )}+  \frac{1}{\mathfrak{S}_{vel}^{2}(\tau; x, v, \tilde v, \zeta )}. 
\end{align*}
\end{definition}

In section 5, we estimate the seminorms below, which are used in the proof of Theorem \ref{theo:Holder}.  

\begin{definition}[seminorms] \label{def:seminorm}
    Let $x, \bx \in \O$ and  $v, \bv, \zeta \in\R^{3}$. For $\varpi > 0$, we define
	\begin{equation} \notag
	\begin{split}
	\mathfrak{H}^{2\b}_{sp}(s) &= \sup_{\substack{x\in {\O},  \\ 0<|v-\bv|\leq 1}} e^{ - \varpi \langle v  \rangle^{2} s } 
		\int_{\zeta} k_{c}(v, v+\zeta) \frac{ | f(s, x, v+\zeta) - f(s, \bx, v+\zeta) | }{ | x - \bx |^{2\b} } d\zeta,   \\
	\end{split}
	\end{equation}
	\begin{equation} \notag
	\begin{split}
		\mathfrak{H}^{2\b}_{vel}(s) &=  \sup_{\substack{v\in\R^{3},  \\ 0<|x-\bx|\leq 1}} e^{ - \varpi \langle v  \rangle^{2} s } 
	\int_{\zeta} \mathbf{k}_{c}(v, \bv, \zeta) \frac{ | f(s, x, v+\zeta) - f(s, x, \bv+\zeta) | }{ | v - \bv |^{2\b} }  d\zeta. \\
	\end{split}
	\end{equation}
\end{definition}

\subsection{Scheme of Proof and Organization of the Paper}
\subsubsection{Nonlocal to local iteration scheme}
We provide a brief overview of each section. From \eqref{f_expan} and a version of Carleman's representation and a priori $L^\infty$-bound of $f$, we can derive the difference quotients for $\beta$, 
\Be \label{lin_expan}
\begin{split}
	&\frac{ | f(t,x,v+\zeta) - f(t, \bar{x}, \bar{v}+\zeta) | }{ |(x,v)-(\bar{x}, \bar{v})|^{\b} }   \\
	&\lesssim      \int_{0}^{t} 
	\frac{  |  X(s) - \bar{X}(s) |^{2\beta} }{ |(x,v)-(\bar{x}, \bar{v})|^{\b} } 
	J_x[f(s)] (V(s);X(s),\bar{X}(s))
	ds  \\
	&\quad +   \int_{0}^{t}
	\frac{  |  V(s) - \bar{V}(s) |^{2\beta} }{ |(x,v)-(\bar{x}, \bar{v})|^{\b} } 
	J_v[f(s)] (X(s);V(s), \bar{V}(s)) ds
	+  good \  terms \\
    & \lesssim \sup_{s, v, x\neq \bar{x}}   J_{x}[f(s)] (v; x, \bar{x}) \times \int_{0}^{t}	\frac{  |  X(s) - \bar{X}(s) |^{2\beta} }{ |(x,v)-(\bar{x}, \bar{v})|^{\b} }  ds \\
    &\quad +  \sup_{s, x, v\neq \bar{v}}  J_{v} [f(s)](x; v, \bar{v})  \times \int_{0}^{t}
	\frac{  |  V(s) - \bar{V}(s) |^{2\beta} }{ |(x,v)-(\bar{x}, \bar{v})|^{\b} } ds +  good \  terms, 
\end{split}
\Ee
where $\mathbf{J}$-seminorms are defined as 
\Be\label{J_norm}
\begin{split}
	J_x[f(s)] (V(s);X(s),\bar{X}(s))& = \int_{|u|\lesssim 1} \frac{ |f(s,X(s),u+V(s)) - f(s,\bar{X}(s),u+V(s))| }{ |X(s)-\bar{X}(s)|^{2\beta} } du ,\\
	J_v[f(s)] (X(s);V(s), \bar{V}(s))&=
	\int_{|u|\lesssim 1} \frac{ |f(s,X(s),u+V(s)) - f(s,X(s),u+\bar{V}(s))| }{ |V(s)-\bar{V}(s)|^{2\beta} } du . 
\end{split}
\Ee
Note that $X(s)=X(s;t,x,v+\zeta)$, $V(s)=V(s;t,x,v+\zeta)$, $\bar{X}(s)=\bar{X}(s;t,\bar{x},\bar{v}+\zeta)$, and $\bar{V}(s)=\bar{V}(s;t,\bar{x},\bar{v}+\zeta)$. Using \eqref{f_expan}, we obtain the inequality about the $\mathbf{J}$-seminorms \eqref{J_norm} as
\begin{equation} \label{xiter}
\begin{split}
J_{x}[f(s)]( v; x, \bar{x})
&  \lesssim \  \int^t_0  
	\int_{|\zeta| \leq 1} \frac{|X(s) - \bar{X}(s)|^{2\beta}}{|x-\bar{x}|^{2\b}}  d\zeta 
ds \times
\sup_{s, v, x\neq \bar{x}}   J_{x}[f(s)] (v; x, \bar{x})   \\
&\quad +   \int^t_0 
	\int_{|\zeta|\leq 1} \frac{|V(s) - \bar{V}(s)|^{2\beta}}{|x-\bar{x}|^{2\b}}  d\zeta 
ds \times
 \sup_{s, x, v\neq \bar{v}}  J_{v} [f(s)](x; v, \bar{v}) + good \  terms\\
& \lesssim \sup_{\substack{v\in\R^{3} \\ 0 < |x - \bx|\leq 1}} 
	\langle v \rangle \frac{|f_{0}( x, v ) - f_{0}(\bx, v)|}{|x - \bx|^{2\b}}  
	+ \sup_{\substack{x\in\bar{\O} \\ 0 < |v - \bv|\leq 1}}  \langle v \rangle^{2} \frac{|f_{0}( x, v ) - f_{0}( x, \bv)|}{|v - \bv|^{2\b}}  
	+ \|w_{0} f_{0}\|_{\infty} \\
 &\quad + \frac{1}{\varpi} \Big[\sup_{s, v, x\neq \bar{x}}   J_{x}[f(s)] (v; x, \bar{x})+\sup_{s, x, v\neq \bar{v}}  J_{v} [f(s)](x; v, \bar{v}) \Big]\mathcal{P}_{3}(\|e^{\vartheta_{0}|v|^{2}} f_{0}\|_{\infty})
\end{split} 
\end{equation}
for $0<\beta<1/4$, and we also have a similar expression for $J_{v}[f(s)] (x; v, \bar{v})$. The key is to calculate the integral of the fractional trajectory in \eqref{xiter}. Then we obtain the uniform bound of $J_{x}[f(s)]( v; x, \bar{x})$ and  $J_{v}[f(s)] (x; v, \bar{v})$ for 
sufficiently large $\varpi \gg_{f_0, \b} 1$. Next, we use \eqref{lin_expan} to bound the H\"older seminorm. \\ 

Now, we briefly outline each section. In section 2, we obtain the explicit formula $(X(s;t,x,v),V(s;t,x,v))$. To get the uniform bound of $J_{x}[f(s)]( v; x, \bar{x})$ and $J_{v}[f(s)] (x; v, \bar{v})$ from \eqref{xiter}, we estimate 
\begin{align*}
 \frac{   |(X,V)(s;t,x,v)-(X,V)(s;t,\bx,\bv)|}{|(x,v)-(\bx, \bv)|}
\end{align*}
for $x, \bx \in \O$, $v, \bv \in \mathbb{R}^3$ in section 3,4, and we perform the integral estimate in section 5. Moreover, in section 6, we estimate 
\begin{align*}
     \frac{   |(X,V)(s;t,x,v)-(X,V)(s;t,\bx,\bv)|}{|(x,v)-(\bx, \bv)|^{1/2}},
\end{align*} and refer to these results as $C^{0,\frac{1}{2}}_{x,v}$ estimates of trajectory. Lastly, we apply the uniform bound of $J_{x}[f(s)]( v; x, \bar{x})$, $J_{v}[f(s)] (x; v, \bar{v})$ and $C^{0,\frac{1}{2}}_{x,v}$ estimates of trajectory to \eqref{lin_expan}, and prove the Theorem \ref{theo:Holder} in section 7. 

\subsubsection{The results of each section}

We summarize the results for each section in more detail. In section 2, we can formulate the trajectories satisfying the specular condition into three cases.(See Figure  \ref{case2,3}). For the particle that has velocity $v$ at space $x$, $a(x,v)$ is the angle between the normal vector at the point of collision and the velocity when the particle hits the outer circle, and $b(x,v)$ is the angle when the particle hits the inner circle. We derive
\begin{align} \label{a,b : intro}
&\cos a(x,v) = \sqrt{1-\frac{|x|_p^2}{R^2}+\frac{(x_p \cdot \hat{v}_p)^2}{R^2} }>0, \quad
    &\cos b(x,v) =\sqrt{1-\frac{|x|_p^2}{r^2}+\frac{(x_p \cdot \hat{v}_p)^2}{r^2} }\geq 0.
\end{align}

In section 3, we consider a trajectory that belongs to $\mathcal{C}_1$ or $\mathcal{C}_2$.  For $x,\bx \in \O$, $v,\bv,\zeta \in \mathbb{R}^3$, we assume that $(\X(\tau),v), (x,\V(\tau)) \in \mathcal{C}_1$ or  $(\X(\tau),v), (x,\V(\tau)) \in \mathcal{C}_2$ for all $\tau \in [0,1]$. Then, we have 
\begin{align*}
\begin{split}
 |(X,V)(s;t,x,v)-(X,V)(s;t,\bx,v)|  \lesssim_{|v|} |x-\bx|\int_0^{1}  \frac{1}{\mathfrak{S}_{sp}(\tau; x, \tx, v)}  \; d\tau,
\end{split}
\end{align*} and
\begin{align} 
\begin{split}
|(X,V)(s;t,x,v+\zeta)-(X,V)(s;t,x,\bv+\zeta)|
  \lesssim_{|v|} |v-\bv|\int_0^1 \frac{1}{\mathfrak{S}_{vel}(\tau; x, v, \tilde v, \zeta )} \; d\tau.
\end{split}
\end{align} We note that the above results are equivalent to the case where a particle hits once outside of a general convex domain. If $(x,v) \notin \mathcal{G}$, $(\tx,v) \in \mathcal{G}$, we obtain
\begin{align*}
    \int_0^{1}  \frac{1}{\mathfrak{S}_{sp}(\tau; x, \tx, v)}  \; d\tau \lesssim_{|v|} \frac{1}{\cos b(x,v)} 
\end{align*} by singularity averaging in Lemma \ref{lemma:g_sp}. Also, if $(x,v+\zeta) \notin \mathcal{G}$, $(x,\tv+\zeta) \in \mathcal{G}$, we obtain
\begin{align*}
    \int_0^1 \frac{1}{\mathfrak{S}_{vel}(\tau; x, v, \tilde v, \zeta )} \; d\tau\lesssim_{|v|} \frac{1}{\cos b(x,v+\zeta)} 
\end{align*} by singularity averaging in Lemma \ref{lemma:g_vel}. \\

In section 4, we consider $\mathcal{C}_3$ case. For $x,\bx \in \O$, $v,\bv \in \mathbb{R}^3$, we assume that $(\X(\tau),v), (x,\V(\tau)) \in \mathcal{C}_3$ for all $\tau \in [0,1]$. Then, we have
    \begin{align} \label{sec4:case3:x-x:x} 
    \begin{split}
       &|X(s;t,x,v)-X(s;t,\bx,v)|
     \lesssim_{|v|} |x-\bx|\max \left\{\frac{1}{\cos a(x,v)}, \frac{1}{\cos a(\bx,v)} \right\},
      \end{split}
      \end{align}
      \begin{align}\label{sec4:case3:x-x:v} 
    \begin{split}
        &|V(s;t,x,v)-V(s;t,\bx,v)|
      \lesssim_{|v|} |x-\bx| \max \left\{\frac{1}{\cos^2 a(x,v)}, \frac{1}{\cos^2 a(\bx,v)} \right\},
      \end{split}
      \end{align} and
\begin{align}\label{sec4:case3:v-v:x}
\begin{split}
      & |X(s;t,x,v)-X(s;t,x,\bv)| \lesssim_{|v|} |v-\bv|,
\end{split} 
\end{align}
\begin{align*}
\begin{split}
      &|V(s;t,x,v)-V(s;t,x,\bv)| \lesssim_{|v|} |v-\bv|\max \left\{\frac{1}{\cos a(x,v)}, \frac{1}{\cos a(x,\bv)} \right \}.
\end{split} 
\end{align*} For convex only cases  $\mathcal{C}_{3}$, we can adopt the sharp result of \cite{GD2022} to \eqref{sec4:case3:x-x:x}, \eqref{sec4:case3:x-x:v}, and \eqref{sec4:case3:v-v:x}. \\

In section 5, we get estimates for the trajectory differences regardless of the cases $\mathcal{C}_{1,2,3}$. Next, we obtain
    \begin{align}\label{int_sec5}
        \int \frac{e^{-c|\zeta|^2}}{|\zeta|}\frac{\langle v+\zeta \rangle^{r}}{|\cos b(x,v+\zeta)|^{2\beta}}\mathbf{1}_{\{(x,v+\zeta)|{(x,v+\zeta)}\in\text{$\mathcal{C}_{1,2}$}\}} \; d\zeta \lesssim_{\beta} \langle v \rangle^{r+1}
    \end{align} for $0<\beta<1/2$ and 
    \begin{align} \label{int_sec5_1}
        \int \frac{e^{-c|\zeta|^2}}{|\zeta|}\frac{\langle v+\zeta \rangle^{r}}{|\cos a(x,v+\zeta)|^{4\beta}}\mathbf{1}_{\{(x,v+\zeta)|{(x,v+\zeta)}\in\text{$\mathcal{C}_3$}\}} \; d\zeta \lesssim_{\beta} \langle v \rangle^{r+1}
    \end{align}  for $0<\beta<1/4$. From \eqref{int_sec5}, \eqref{int_sec5_1}, we estimate $\mathbf{J}$-seminorms for $0<\beta<1/4$. In particular, we should take the singular order of the RHS of \eqref{sec4:case3:x-x:v} into account since it is the most singular order among all fraction estimates of characteristics. We note that the range of $\beta\in(0,\frac{1}{4})$ comes from the singular order of $\cos a(x,v)$ in \eqref{sec4:case3:x-x:v}. \\ 

In section 6, we consider optimal $C^{0,\frac{1}{2}}_{x,v}$ regularity of trajectories in the domain. Let fix $x, \bx \in \O$, $v, \bv \in \mathbb{R}^3$. For  $(x,v),(\bx,v), (x,\bv) \in \mathcal{C}_1$ or  $(x,v),(\bx,v), (x,\bv) \in \mathcal{C}_2$, we obtain
\begin{align*}
      |(X,V)(s;t,x,v)-(X,V)(s;t,\bx,\bv)|
  \lesssim_{|v|} |(x,v)-(\bx, \bv)|^{1/2}.
\end{align*} For $(x,v),(\bx,v), (x,\bv) \in \mathcal{C}_3$, we obtain
   \begin{align} \label{sec6:case3:x-x:x} 
    \begin{split}
     & |(X,V)(s;t,x,v)-(X,V)(s;t,\bx,v)| \lesssim_{|v|} |x-\bx|^{1/2}\min \left\{\frac{1}{\cos a(x,v)}, \frac{1}{\cos a(\bx,v)} \right\}
      \end{split}
      \end{align}
       and
\begin{align}\label{sec6:case3:v-v:v}
\begin{split}
      &|V(s;t,x,v)-V(s;t,x,\bv)|\lesssim_{|v|} |v-\bv|^{1/2}\max \left\{\frac{1}{\cos^{1/2} a(x,v)}, \frac{1}{\cos^{1/2} a(x,\bv)}\right \}.
\end{split} 
\end{align}  \\

In the last section 7, we prove Theorem  \ref{theo:Holder} combining $C^{0,\frac{1}{2}}_{x,v}$ estimates of trajectory and $\mathbf{J}$-seminorms estimates. Due to the singularity of $C^{0,\frac{1}{2}}_{x,v}$ estimates of trajectory in \eqref{sec6:case3:x-x:x}, \eqref{sec6:case3:v-v:v}, there are the singularity of the form of $\cos a(x,v)$ in \eqref{a,b : intro} is multiplied by the H\"older norm in \eqref{est:Holder}. 

\section{Trajectory analysis}
\begin{lemma} \label{lemma:traject} Suppose the domain is given as in Definition \ref{def:domain}. For $(x,v) \in \O \times \mathbb{R}^3$ and $k \in \mathbb{N}\cup\{0\}$,
    we let 
    \begin{align} 
              &V_{k+1}(x,v) = V_{k}(x,v)-2(n(X_{k+1}(x,v))\cdot V_k(x,v))n(X_{k+1}(x,v)), \label{def:vk}\\ 
        &t_{b,k+1}(x,v) = t_b(X_{k+1}(x,v),V_{k+1}(x,v)), \quad t_{b,0}(x,v)=t_b(X_0(x,v),V_0(x,v))=t_*(x,v), \notag \\
        &X_{k+2}(x,v) = X_{k+1}(x,v)-t_{b,k+1}(x,v)V_{k+1}(x,v),\label{def:xk}
    \end{align}
 where $t_*(x,v)$ is defined in \eqref{def:t*,l}. Then it holds that
\begin{align} \label{lemma:|x|}
    t_{b,k+1}(x,v) = t_*(x,v), \quad |X_{2k}(x,v)|_p=|X_{0}(x,v)|_p, \quad |X_{2k+1}(x,v)|_p=|X_{1}(x,v)|_p,
\end{align} and 
\begin{align} \label{lemma:angle}
\begin{split}
&n(X_{2k}(x,v)) \cdot V_{2k}(x,v) =n(X_{0}(x,v)) \cdot V_{0}(x,v), \\
&n(X_{2k+1}(x,v)) \cdot V_{2k+1}(x,v)=n(X_{1}(x,v)) \cdot V_{1}(x,v).
\end{split}
\end{align}
\end{lemma}
\begin{proof}
Taking the inner product of $n(X_{k+1}(x,v))$ to \eqref{def:vk}, we get 
\begin{align} \label{inner}
\begin{split}
        &X_{k+1}(x,v)_p \cdot V_{k+1}(x,v)=-X_{k+1}(x,v)_p \cdot V_{k}(x,v), \\
    &|X_{k+1}(x,v)+\tau V_{k}(x,v)|_p^2 
   = |X_{k+1}(x,v)-\tau V_{k+1}(x,v)|_p^2.
\end{split}
\end{align} By \eqref{inner}, we obtain  $t_{b,k}(x,v)=t_{b,k-1}(x,v) \cdots =t_{b,0}(x,v)=t_*(x,v)$, where $t_b(x,v)$ in \eqref{def:tb}.
Using
 \begin{align*}
    & X_k(x,v) = X_{k+1}(x,v)+t_*(x,v)V_k(x,v),
    & X_{k+2}(x,v) = X_{k+1}(x,v)-t_*(x,v)V_{k+1}(x,v),
 \end{align*} we compute $|X_{k+2}(x,v)|_p=|X_k(x,v)|_p$. From \eqref{def:vk} and \eqref{def:xk}, we compute
 $ n(X_{k+2}(x,v))\cdot V_{k+2}(x,v)=n(X_k(x,v)) \cdot V_k(x,v)$. 
\end{proof}

\begin{definition} \label{def:traject}
    We define
\begin{align} \label{def:m}
    &m(s;t,x,v) :=[(t-t_b(x,v)-s)/t_*(x,v)]+ 1 \in \mathbb{N}\cup\{0\}.
\end{align} For  $k \in [0,m(0;t,x,v)]$, we define
\begin{align} \label{def:t_k}
    &t_{k}(t,x,v) = t-t_b(x,v)-(k-1)t_*(x,v).
\end{align} For $k=m(s;t,x,v)$, we define 
\begin{align*} 
    &X(s;t,x,v)=X_k(x,v)-V_k(x,v)(t_k(t,x,v)-s), \\
    &V(s;t,x,v)=V_k(x,v),
\end{align*} where  $X_k(x,v), V_k(x,v)$ are in \eqref{def:vk},\eqref{def:xk}.
\end{definition}

\begin{remark}[orientation] \label{ori}
Let us 
\begin{align*}
x=(|x|_p\cos\theta_x, |x|_p\sin\theta_x,x_3) \in \O, \quad
v=(|v|_p\cos\theta_v, |v|_p\sin\theta_v,v_3) \in  \mathbb{R}^3
\end{align*} for $\theta_x, \theta_v \in [-\pi,\pi]$. 
If $0 < \theta_x-\theta_v < \pi$ in mod $2\pi$, then $X(s;t,x,v)$ moves clockwise.
If $-\pi < \theta_x-\theta_v < 0$ in mod $2\pi$, then $X(s;t,x,v)$ moves counter-clockwise. If $\theta_x-\theta_v=0$ or $\pi$, then $X(s;t,x,v)$ moves repeatedly between some two points. In this paper, we consider that $X(s;t,x,v)$ is concluded in both moving clockwise case and moving counter-clockwise case if $\theta_x-\theta_v=0$ or $\pi$.
\end{remark}

\begin{lemma} \label{lemma :a,b}
Suppose that the domain is given as in Definition \ref{def:domain}. Let us 
\begin{align*}
x=(|x|_p\cos\theta_x, |x|_p\sin\theta_x,x_3) \in \O, \quad
v=(|v|_p\cos\theta_v, |v|_p\sin\theta_v,v_3) \in  \mathbb{R}^3
\end{align*} for $\theta_x, \theta_v \in [-\pi,\pi]$ and $0\leq \theta_x-\theta_v \leq \pi $ in mod $2\pi$. Then, we have
   \begin{align} \label{detail a,b:1}
   \begin{split}
        &a(x,v)=\sin^{-1}\Big(\frac{|x|_p}{R}\sin (\theta_x-\theta_v)\Big) \in [0,\pi/2), \\ &b(x,v)=\sin^{-1}\Big(\frac{|x|_p}{r}\sin (\theta_x-\theta_v)\Big)  \in [0,\pi/2],
   \end{split}
    \end{align}where $a(x,v), b(x,v)$ are in Definition \ref{def:time}.
\end{lemma} 

\begin{proof}
We assume that  $(x,v) \in \mathcal{C}_1$ with moving clockwise, and draw the points $X_{0}(x,v)_p$, $X_{1}(x,v)_p$, $x_p$, and $v_p$ in Figure \ref{angle}. By sine law,
    \begin{align} \label{sine law}
        \frac{|x|_p}{\sin a(x,v)} = \frac{R}{\sin(\theta_x-\theta_v)}
        \quad \text{and} \quad \frac{|x|_p}{\sin b(x,v)}=\frac{r}{\sin(\theta_x-\theta_v)}.
    \end{align}  Similarly, \eqref{detail a,b:1} is true when $(x,v) \in$ $\mathcal{C}_2$ or $\mathcal{C}_3$ with moving clockwise. Therefore, we also obtain 
    \begin{align} \label{a,b : sec2}
&\cos a(x,v) = \sqrt{1-\frac{|x|_p^2}{R^2}+\frac{(x_p \cdot \hat{v}_p)^2}{R^2} }>0, \quad
    &\cos b(x,v) =\sqrt{1-\frac{|x|_p^2}{r^2}+\frac{(x_p \cdot \hat{v}_p)^2}{r^2} }\geq 0.
\end{align}
\end{proof}

 \begin{figure}[t]
\centering
\includegraphics[width=5cm]{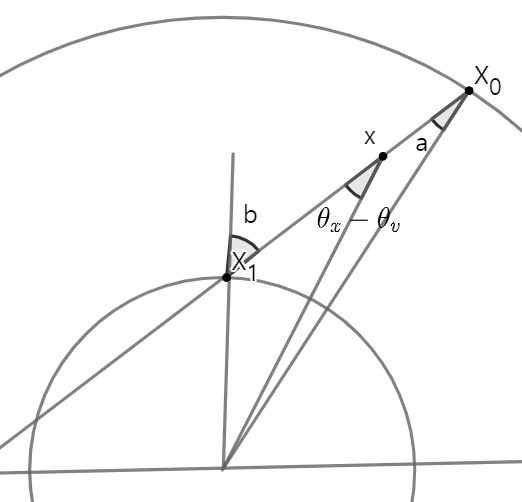}
\caption{the relationship between angles.}\label{angle}
\end{figure} 

\begin{lemma} \label{coor:X}
Suppose that the domain is given as in Definition \ref{def:domain}.Let us 
\begin{align*}
x=(|x|_p\cos\theta_x, |x|_p\sin\theta_x,x_3) \in \O, \quad
v=(|v|_p\cos\theta_v, |v|_p\sin\theta_v,v_3) \in  \mathbb{R}^3
\end{align*} for $\theta_x, \theta_v \in [-\pi,\pi]$ and $0\leq \theta_x-\theta_v \leq \pi $ in mod $2\pi$. For $k \in \mathbb{N}\cup\{0\}$, we let
\begin{align*}
    X_{k}(x,v) = \Big(|X_k(x,v)|_p \cos \theta_k (x,v), |X_k(x,v)|_p \sin \theta_k (x,v), x_3-v_3(t_b(x,v)+(k-1)t_*(x,v)) \Big).
\end{align*}
(1) Assume $(x,v) \in \mathcal{C}_1$. For $k \in \mathbb{N}\cup\{0\}$, we have
\begin{align}\label{coor1:X}
    \begin{split}
        |X_{2k}(x,v)|_p = R , \quad  |X_{2k+1}(x,v)|_p = r, \quad
         \theta_{X_k}(x,v) = \theta_v + kb(x,v) -(k-1)a(x,v), 
    \end{split}
\end{align} where $a(x,v), b(x,v)$ is in \eqref{def:angle:1}.\\
(2) Assume $(x,v) \in \mathcal{C}_2$. For $k \in \mathbb{N}\cup\{0\}$, we have
\begin{align}\label{coor2:X}
    \begin{split}
        |X_{2k}(x,v)|_p = r , \quad  |X_{2k+1}(x,v)|_p = R, \quad
        \theta_{X_k}(x,v) = \pi+\theta_v + (k-1)b(x,v) -ka(x,v), 
    \end{split}
\end{align} where $a(x,v), b(x,v)$ is in \eqref{def:angle:2}.\\
(3) Assume $(x,v) \in \mathcal{C}_3$. For $k \in \mathbb{N}\cup\{0\}$, we have
\begin{align}\label{coor3:X}
    \begin{split}
        |X_{k}(x,v)|_p = R , \quad
        \theta_{X_k}(x,v) = k\pi+\theta_v + (1-2k)a(x,v),
    \end{split}
\end{align} where $a(x,v)$ is in \eqref{def:angle:3}. 
\end{lemma} 
\begin{proof}
For  $k \in \mathbb{N}\cup\{0\}$, we write
\begin{align}\label{theta-theta}
\begin{split}
      &\quad \theta_{X_{k+1}}(x,v)-\theta_{X_{k}}(x,v) \\&=\cos^{-1}\left(\widehat{X_{k+1}(x,v)_p}\cdot \widehat{V_k(x,v)_p}\right)-\cos^{-1}\left(\widehat{X_{k}(x,v)_p}\cdot \widehat{V_k(x,v)_p} \right) \in [0,\pi/2).
\end{split}
\end{align}
When $(x,v) \in \mathcal{C}_1$, from \eqref{def:angle:1} and \eqref{theta-theta}, we have 
    \begin{align*}
        &\theta_{X_{k+1}}(x,v) = \theta_{X_{k}}(x,v)+b(x,v)-a(x,v), \quad \theta_{X_0}(x,v) = \theta_v +a(x,v).
    \end{align*}
When $(x,v) \in \mathcal{C}_2$, we have $\theta_{X_0}(x,v) = \theta_v +\pi-b(x,v)$. When $(x,v) \in \mathcal{C}_3$, we have
\begin{align*}
        &\theta_{X_{k+1}}(x,v) =\pi-2 a(x,v), \quad \theta_{X_0}(x,v) =  \theta_v + a(x,v).
    \end{align*}
\end{proof}

\begin{lemma} \label{coor:V}
 Suppose that the domain is given as in Definition \ref{def:domain}.Let us 
\begin{align*}
x=(|x|_p\cos\theta_x, |x|_p\sin\theta_x,x_3) \in \O, \quad
v=(|v|_p\cos\theta_v, |v|_p\sin\theta_v,v_3) \in  \mathbb{R}^3
\end{align*} for $\theta_x, \theta_v \in [-\pi,\pi]$ and $0\leq \theta_x-\theta_v \leq \pi $ in mod $2\pi$. For $k \in \mathbb{N}\cup\{0\}$, we let
\begin{align*}
     V_{k}(x,v) = (|v|_p \cos \theta_{V_{k}}(x,v), |v|_p \sin \theta_{V_{k}}(x,v), v_3).
\end{align*}
(1) Assume  $(x,v) \in \mathcal{C}_1$. For $k \in \mathbb{N}\cup\{0\}$,  
\begin{align} \label{coor1:V}
\begin{split}
   &\theta_{V_{2k}}(x,v) = \theta_v + 2k(b(x,v) -a(x,v)),  \\
    &\theta_{V_{2k+1}}(x,v) = \pi+\theta_v + (2k+2)b(x,v) -2ka(x,v), 
\end{split}
\end{align} where $a(x,v), b(x,v)$ are in \eqref{def:angle:1}.\\ 
(2) Assume  $(x,v) \in \mathcal{C}_2$. For $k \in \mathbb{N}\cup\{0\}$,  
\begin{align} \label{coor2:V}
\begin{split}
   &\theta_{V_{2k}}(x,v) = \theta_v + 2k(b(x,v) -a(x,v)), \\
    &\theta_{V_{2k+1}}(x,v) = \pi+\theta_v + 2kb(x,v) -(2k+2)a(x,v),
\end{split}
\end{align} where $a(x,v), b(x,v)$ are in \eqref{def:angle:2}. \\ 
(3) Assume  $(x,v) \in \mathcal{C}_3$. For $k \in \mathbb{N}\cup\{0\}$,  
\begin{align} \label{coor3:V}
    \theta_{V_{k}}(x,v) = k\pi+\theta_v -2ka(x,v) 
\end{align} where $a(x,v)$ is in \eqref{def:angle:3}.
\end{lemma}

\begin{proof} 
For $k \in \mathbb{N}\cup\{0\}$, we have
\begin{align} \label{theta-theta:V}
\begin{split}
        \theta_{V_k}(x,v) &=\theta_{X_{k+1}}(x,v)-\cos^{-1}\left(\widehat{X_{k+1}(x,v)_p}\cdot \widehat{V_k(x,v)_p} \right) \\
    &=\theta_{X_{k+1}}(x,v)-\pi+\cos^{-1}\left(\widehat{X_{k+1}(x,v)_p}\cdot \widehat{V_{k+1}(x,v)_p} \right).
\end{split}
\end{align}
(1) Using \eqref{coor1:X} and \eqref{theta-theta:V}, we get
\begin{align*}
   &\theta_{V_{2k}}(x,v)=\theta_{X_{2k+1}}(x,v)-b(x,v)=\theta_v+2k(b(x,v)-a(x,v)), \\
    &\theta_{V_{2k+1}}(x,v)=\theta_{X_{2k+2}}(x,v)-\pi+a(x,v)
    =-\pi+\theta_v + (2k+2)b(x,v) -2ka(x,v). 
\end{align*} By similar as in (1), we get \eqref{coor2:V} and \eqref{coor3:V}.
\end{proof}

\begin{figure}[t]
    \centering
    \includegraphics[width=5cm]{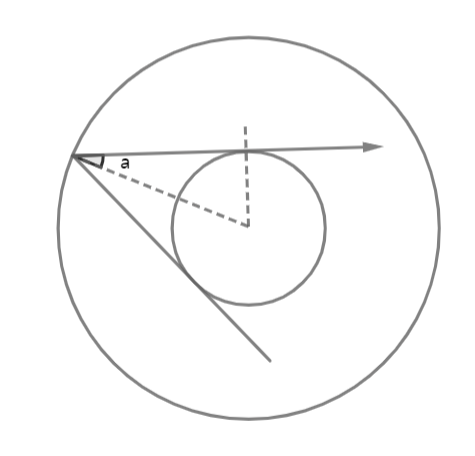}
    \caption{in the case of grazing.} \label{grazing}
    \end{figure} 
\begin{remark}[grazing] \label{b=pi/2}
  When $b(x,v)=\pi/2$ in \eqref{def:angle:1}, the trajectory can be expressed in two different ways by \eqref{coor1:X} or \eqref{coor3:X}.(See Figure \ref{grazing}.) \\ 
\end{remark} 

\section{Difference quotients estimates and Singularity Averaging for \texorpdfstring{$\mathcal{C}_1$}{} and \texorpdfstring{$\mathcal{C}_2$}{} }

In section 3, we let $f : \overline{\O} \times \R^{3} \rightarrow \R_{+} \cup \{0\}$ be a function which satisfies specular reflection \eqref{specular}, where $\O$ is a domain as in Definition \ref{def:domain} and $w(v) = e^{\vartheta|v|^{2}}$ for some $0 < \vartheta$. Let fix any point $t,s \in \mathbb{R}^{+}$, $v, \bv, \zeta \in \mathbb{R}^3$, $x,\bx \in \O$. We assume \eqref{assume_x}, \eqref{assume_v}, and recall $\tx(x,\bx,v+\zeta)$ in \eqref{def:tx}, $\X(\tau)$ in \eqref{def:x_para}, $\tv(v,\bv,\zeta)$ in \eqref{def:tv} and $\V(\tau)$ in \eqref{def:v_para}. 

\begin{definition} \label{x,y}
For $k, s \in \mathbb{R}^{+}, v, \zeta \in \mathbb{R}^3$, we define
\begin{align*}
 &\mathcal{X}(k,s,v,\zeta) =  \left[\frac{ e^{\varpi \langle v+\zeta \rangle^2 s}  }{ \langle v+\zeta \rangle}  
		\sup_{ \substack{ v\in \R^{3} \\0<|x-\bx|<1   } }\Big( e^{-\varpi \langle v \rangle^2 s}  
		\langle v \rangle \frac{|f(s, x, v ) - f(s, \bx, v)|}{|x - \bx|^{k}}\Big)+ \frac{\|wf(s)\|_{\infty}}{w(v+\zeta)} \right], \\
  &\mathcal{V}(k,s,v,\zeta) = \left[ 
		\frac{ e^{\varpi \langle v+\zeta \rangle^2 s} }{ \langle v+\zeta \rangle^{2}}	
		\sup_{ \substack{ x\in \O \\ 0<|v-\bv|<1  } } \Big(e^{-\varpi \langle v \rangle^2 s}  \langle v \rangle^{2} \frac{|f(s, x, v ) - f(s, x, \bv)|}{|v - \bv|^{\gamma}}\Big) + \frac{\|wf(s)\|_{\infty}}{w(v+\zeta)}   \right].
\end{align*}
\end{definition}

\subsection{Nonsingular parts of $\mathcal{C}_1$ and $\mathcal{C}_2$}
\begin{lemma}\label{tx-bx}
Suppose that $(\tx,v+\zeta), (\bx,v+\zeta) \in \mathcal{C}_1$ or
$(\tx,v+\zeta), (\bx,v+\zeta) \in \mathcal{C}_2$
 with moving clockwise. For $|x-\bx|\leq 1$, we have
    \begin{align} \label{f-f:bx}
    \begin{split}
          &|f(s, X(s;t,\tx,v+\zeta), V(s;t,\tx,v+\zeta)) - f(s, X(s;t, \bx, v+\zeta ), V(s;t, \bx, v+\zeta ))| \\
          &\lesssim \left(\mathcal{X}(\gamma,s,v,\zeta)  + |v+\zeta|_p^{\gamma} \mathcal{V}(\gamma,s,v,\zeta)  \right)|x-\bx|^{\gamma}.
    \end{split}
    \end{align}  
\end{lemma} 
\begin{proof} We let $v=v+\zeta$, $m=m(s;t,\bx,v+\zeta)$, $n=m(s;t,\tx,v+\zeta)$ and $m \geq n$.  By definition of $m(s;t,x,v)$ in \eqref{def:m}, $s\in [0,t] $ satisfies 
    \begin{align} \label{m-n:x:ns}
        t_{n+1}(t,\tx,v)< s \leq t_n(t,\tx,v) \quad \text{and} \quad 
        t_{m+1}(t,\bx,v)< s \leq t_m(t,\bx,v).
    \end{align} 
Using $t_{n+1}(t,\tx,v) < s \leq t_m(t,\bx,v)$ and $t_*(\tx,v)=t_*(\bx,v)$, we have
    \begin{align} \label{txx:Case1:contra}
        m-n-1 < \frac{t_b(\tx,v)-t_b(\bx,v)}{t_*(\tx,v)} \leq 1.
    \end{align}  We assume that $m=n$. Then, 
\begin{align} \label{v-v:m=n:x:ns:second term}
    |V(s;t,\tx,v)-V(s;t,\bx,v)| = |V_m(\tx,v)-V_m(\bx,v)|=0.
\end{align} and 
\begin{align} \label{case1:tvbv:m=n}
    |X(s;t,\tx,v)-X(s;t,\bx,v)| \lesssim |x-\bx|.
\end{align} Lasatly, we assume that $m=n+1$. We define $X_{n+1}^{*}(\tx,v)$ as
\begin{align*}
    X_{n+1}^{*}(\tx,v)_p= X_{n+1}(\tx,v)_p
    \quad \text{and} \quad X_{n+1}^{*}(\tx,v)_{ver}= X(s;t,\tx,v)_{ver}.
\end{align*} Since the specular boundary condition in \eqref{specular}, we can split
    \begin{align} \label{f-f:m=n+1:x:ns}
    \begin{split}
        &\quad |f(s, X(s;t,\tx,v), V(s;t,\tx,v)) - f(s, X(s;t, \bx, v ), V(s;t, \bx, v))|\\
        &\leq |f(s, X(s;t,\tx,v), V(s;t,\tx,v)) - f(s, X^{*}_{n+1}(\tx,v), V_{n}(\tx, v))|  \\
        &\quad + | f(s, X^{*}_{n+1}(\tx,v), V_{n+1}(\tx, v))- f(s, X(s;t, \bx, v ), V(s;t, \bx, v))|.
    \end{split}
    \end{align} By \eqref{m-n:x:ns},
    \begin{align} \label{x-x:m=n+1:x:ns:first term}
    \begin{split}
       |X(s;t,\tx,v)-X^{*}_{n+1}(\tx,v)| 
       \leq |t_{n+1}(t,\bx,v)-t_{n+1}(t,\tx,v)||v|_p \leq |x-\bx|.
    \end{split} 
    \end{align} By \eqref{case1:tvbv:m=n},
    \begin{align} \label{x-x:m=n+1:ns:x:second term}
    \begin{split}
     & \quad |X_{n+1}^{*}(\tx,v)-X(s;t, \bx, v )| \\ &\leq
        |X_{n+1}(\tx,v)-X_{n+1}(\bx,v)|_p+|X_{n+1}(\bx,v)-X(s;t, \bx, v )|_p+|x_3-\bx_3|  \\
        &\leq |t_{n+1}(t,\bx,v)-t_{n+1}(t,\tx,v)||v|_p+|x_3-\bx_3| \\
        &\leq |x-\bx|.
    \end{split}
    \end{align} Therefore, we can construct \eqref{f-f:bx} by applying \eqref{x-x:m=n+1:x:ns:first term} to the first term of \eqref{f-f:m=n+1:x:ns} and applying \eqref{v-v:m=n:x:ns:second term}, \eqref{x-x:m=n+1:ns:x:second term} to the second term of \eqref{f-f:m=n+1:x:ns}. 
\end{proof} 

\begin{lemma} \label{tv-bv}
Suppose that $(x,\tv+\zeta), (x,\bv+\zeta) \in \mathcal{C}_1$ or
$(x,\tv+\zeta), (x,\bv+\zeta) \in \mathcal{C}_2$ with moving clockwise. Then, we have
    \begin{align} \label{f-f:bv}
    \begin{split}
          &\quad|f(s, X(s;t,x,\tv+\zeta), V(s;t,x,\tv+\zeta)) - f(s, X(s;t, x, \bv+\zeta ), V(s;t, x, \bv+\zeta ))|  \\
          &\lesssim \;
          \left((t-s)^{\gamma}\mathcal{X}(\gamma,s,v,\zeta)  + \Big[1+(t-s)\max \{|v+\zeta|_p,|\bv+\zeta|_p\} \Big]^{\gamma} \mathcal{V}(\gamma,s,v,\zeta) \right)|v-\bv|^{\gamma}.
    \end{split}
    \end{align}
\end{lemma}

\begin{proof} We let $\tv:=\tv+\zeta,\; \bv:=\bv+\zeta,\; m=m(s;t,x,\bv+\zeta), n=m(s;t,x,\tv+\zeta)$ and $m \geq n$. By definition of $m(s;t,x,v)$ in \eqref{def:m}, $m \in \mathbb{N}$ satisfies
    \begin{align*}
        m-1 \leq \frac{t-s}{t_*(x,\bv)} = (t-s)\frac{|\bv|_p}{l(x,\bv)} \lesssim |\bv|_p(t-s).
    \end{align*} We first assume that $m \geq n+2$. Using $t_{n+1}(t,x,\tv) < t_m(t,x,\bv)$, we have
    \begin{align} \label{m-n-1:ns:v}
    \begin{split}
        m-n-1 &\leq \frac{ |t_b(x,\tv)-t_b(x,\bv)|}{t_*(x,\tv)} + (m-1)\frac{|t_*(x,\tv)-t_*(x,\bv)|}{t_*(x,\tv)} \\
        &\leq \frac{m}{t_*(x,\tv)}\Big|\frac{1}{|\tv|_p}-\frac{1}{|\bv|_p}\Big|
        \lesssim \frac{2(m-1)}{t_*(x,\tv)}\frac{|\tv-\bv|}{|\tv|_p|\bv|_p} \lesssim (t-s)|\tv-\bv|.
    \end{split}
    \end{align}
   Since $m-n-1 \geq 1$, we have
    \begin{align*}
        |V(s;t,x,\tv)-V(s;t,x,\bv)|\leq 2\text{max}\{|v|_p,|\bv|_p\}\cdot 1+|v_3-\bv_3| 
        \lesssim \Big(\text{max}\{|v|_p,|\bv|_p\}(t-s)+1\Big)|v-\bv|
    \end{align*} and
    \begin{align*}
        |X(s;t,x,\tv)-X(s;t,x,\bv)| \leq 2R \cdot 1+ (t-s)|v_3-\bv_3|\lesssim (t-s)|v-\bv|.
    \end{align*} 
    Next, we assume that $m=n \geq 1$. We have
\begin{align} \label{v-v:ns:v:m=n}
\begin{split}
       |V(s;t,x,\tv)-V(s;t,x,\bv)| &\leq |V_m(x,\tv)-V_m(x,\bv)|_p+|v_3-\bv_3| \leq 2|v-\bv|.
\end{split}
\end{align} 
Since $|t_m(x,\bv)-s| < t-s$,
\begin{align*} 
     &\quad |X(s;t,x,\tv)-X(s;t,x,\bv)| \notag \\
     &\leq |v|_p|t_m(t,x,\bv)-t_m(t,x,\tv)|+|V_m(x,\tv)-V_m(x,\bv)|_p|t_m(x,\bv)-s|+(t-s)|\tv-\bv|  \notag  \\
     &\leq |v|_p|t_b(x,\tv)-t_b(x,\bv)|+(m-1)|v|_p|t_*(x,\tv)-t_*(x,\bv)|+2(t-s)|\tv-\bv| 
     \\ &\lesssim (t-s)|v-\bv|.
\end{align*}
If $m=0$, we can check $|V(s;t,x,\tv)-V(s;t,x,\bv)| \leq |v-\bv|, |X(s;t,x,\tv)-X(s;t,x,\bv)|\leq (t-s)|v-\bv|$.
When $m= n+1$, we apply the specular boundary condition as \eqref{f-f:m=n+1:x:ns} in Lemma \ref{tx-bx}, and use  $n \lesssim |\bv|_p(t-s)$ and $n \lesssim |\tv|_p(t-s)$. 
\end{proof}

\subsection{Singular parts of $\mathcal{C}_1$ and $\mathcal{C}_2$}

\begin{lemma}  \label{est:time} Assume that $(\X (\tau),v+\zeta) \in \mathcal{C}_1$ with moving clockwise for all $\tau \in [0,1]$ or $(\X (\tau),v+\zeta) \in \mathcal{C}_2$ with moving clockwise for all $\tau \in [0,1]$.
For $\tx(x,\bx,v+\zeta)$ and $i=b,f,*$,
\begin{align} \label{t/x_C1,2}
    &|t_i(x,v+\zeta)-t_i(\tx,v+\zeta)| \lesssim |x-\tx| \int_{0}^{1}
		\frac{1}{\mathfrak{S}_{sp}(\tau; x, \tx, v+\zeta)}
		d\tau.
\end{align} Assume that $(x,\V(\tau)) \in \mathcal{C}_1$ with moving clockwise for all $\tau \in [0,1]$ or $(x,\V(\tau)) \in \mathcal{C}_2$ with moving clockwise for all $\tau \in [0,1]$.
For $\tv(v,\bv,\zeta)$ and $i=b,f,*$,
\begin{align} \label{t/v_C1,2}
    &|t_i(x,v+\zeta)-t_i(x,\tv+\zeta)| \lesssim |v-\tv| \int_{0}^{1}
		\frac{1}{\mathfrak{S}_{vel}(\tau; x, v, \tv, \zeta)}
		d\tau.
\end{align} 
\end{lemma}
\begin{proof}
  Assume that $(\X (\tau),v+\zeta) \in \text{$\mathcal{C}_1$}$ with moving clockwise for all $\tau \in [0,1]$. We let $\nabla_x = (\frac{\partial}{\partial x_1},\frac{\partial}{\partial x_2},0)^{T}$, $v=v+\zeta$, and $\xi(x)=|x|^2-r^2$.  We have
  \begin{align*}
      \nabla_x(\xi (X_1(x,v)_p)) &= (\nabla_x (X_1(x,v)_p))^{T} (\nabla_x \xi)(X_1(x,v)_p) \\ &=(\nabla_x(x_p-t_b(x,v)v_p))^{T}   (\nabla_x \xi)(X_1(x,v)_p) \\ &= 2\left(
      \begin{bmatrix}
				1 & 0 & 0 \\
				0 & 1 & 0 \\
				0 & 0 & 0
			\end{bmatrix}
      -v(\nabla_x t_b(x,v))^{T}\right)^{T} X_1(x,v)_p.
  \end{align*}
  Since $\xi (X_1(x,v)_p)=0$, we have
\begin{align*}
    \nabla_x t_b(x,v) = \frac{X_{1}(x,v)_{p} }{X_{1}(x,v)_p \cdot v} .
\end{align*} Thus, we have
    \begin{align*}
        |t_b(x,v)-t_b(\tx,v)| \leq \int^1_0 |\nabla_x t_b(\X(\tau),v) \cdot \dot{\X}(\tau)| d\tau 
        \leq |x-\tx| \int_0^1 \left|\frac{X_{1}(\X(\tau),v)_p \cdot \frac{\hat{\X}}{|\hat{\X}|}}{X_{1}(\X(\tau),v)_p \cdot v}\right| d\tau.
    \end{align*} Assume that $(x,\V(\tau)) \in$ $\mathcal{C}_1$ for all $\tau \in [0,1]$, and denote $v=v+\zeta$ and $\tv=\tv+\zeta$. We have
\begin{align*}
     \nabla_v(\xi (X_1(x,v)_p)) &= (\nabla_v (X_1(x,v)_p))^{T} (\nabla_x \xi)(X_1(x,v)_p) \\ &=(\nabla_v(x_p-t_b(x,v)v_p))^{T}   (\nabla_x \xi)(X_1(x,v)_p) \\ &= 2\left(-v_p(\nabla_v t_b(x,v))^{T}-t_b(x,v)\begin{bmatrix}
				1 & 0 & 0 \\
				0 & 1 & 0 \\
				0 & 0 & 0
			\end{bmatrix}\right)^{T} X_1(x,v)_p. 
\end{align*} Since $\xi (X_1(x,v)_p)=0$, we have
    \begin{align*}
    \nabla_v t_b(x,v) = -t_b(x,v)\nabla_x t_b(x,v) = -t_b(x,v) \frac{X_1(x,v)_p}{X_1(x,v)_p \cdot v}
\end{align*} Thus, we have
     \begin{align*}
        |t_b(x,v)-t_b(x,\tv)| \leq \int^1_0 |\nabla_v t_b(x,\V(\tau)) \cdot \dot{\V}(\tau)| d\tau 
        \leq |v-\tv| \int_0^1 t_b(x,\V(\tau))\left|\frac{X_{1}(x,\V(\tau))_p \cdot \frac{\hat{\V}}{|\hat{\V}|}}{X_{1}(x,\V(\tau))_p \cdot v}\right| d\tau.
    \end{align*} Similarly, we estimate \eqref{t/x_C1,2}, \eqref{t/v_C1,2} for $i=f$.
    And using $t_*(x,v)=t_f(x,v)+t_b(x,v)$, we estimate the inequalities for $i=*$.
\end{proof}

\begin{lemma} \label{est:angle} 
Assume that $(\X (\tau),v+\zeta) \in \mathcal{C}_1$ with moving clockwise for all $\tau \in [0,1]$ or $(\X (\tau),v+\zeta) \in \mathcal{C}_2$ with moving clockwise for all $\tau \in [0,1]$.
For $\tx(x,\bx,v+\zeta)$,
\begin{align}
    &|a(x,v+\zeta)-a(\tx,v+\zeta)| \lesssim |x-\tx|\left(1+|v+\zeta|_p\int_{0}^{1}
		\frac{1}{\mathfrak{S}^{1}_{sp}(\tau; x, \tx, v+\zeta)}
		d\tau  \label{est:a/x}\right),\\
  &|b(x,v+\zeta)-b(\tx,v+\zeta)|  \lesssim|x-\tx| \left(1+|v+\zeta|_p\int_{0}^{1}
		\frac{1}{\mathfrak{S}^{2}_{sp}(\tau; x, \tx, v+\zeta)}
		d\tau \right). \label{est:b/x}
\end{align} Assume that $(x,\V(\tau))) \in \mathcal{C}_1$ with moving clockwise for all $\tau \in [0,1]$ or $(x,\V(\tau))) \in \mathcal{C}_2$ with moving clockwise for all $\tau \in [0,1]$.
For $\tv(v,\bv,\zeta)$,
\begin{align}
    &|a(x,v+\zeta)-a(x,\tv+\zeta)| \lesssim |v-\tv|\left( \frac{1}{|v+\zeta|_p}+|v+\zeta|_p\int_{0}^{1}
		\frac{1}{\mathfrak{S}^{1}_{vel}(\tau; x, v, \tv, \zeta)}
		d\tau  \label{est:a/v} \right),\\
  &|b(x,v+\zeta)-b(x,\tv+\zeta)| \lesssim  |v-\tv|\left( \frac{1}{|v+\zeta|_p}+|v+\zeta|_p \int_{0}^{1}
		\frac{1}{\mathfrak{S}^{2}_{vel}(\tau; x, v, \tv, \zeta)}
		d\tau \right).  \label{est:b/v}
\end{align} 
\end{lemma}
\begin{proof}
 Assume that $(\X (\tau),v+\zeta) \in \text{$\mathcal{C}_1$}$ with moving clockwise for all $\tau \in [0,1]$, and denote $v=v+\zeta$. From \eqref{def:angle:1}, we get
    \begin{align} \label{cosb_eq:case1}
    \begin{split}
          r|v|_p \cos b(x,v) &= X_1(x,v)_p \cdot v_p
        =(x_p-t_b(x,v)v_p) \cdot v_p.
    \end{split}
    \end{align}  and
      \begin{align} \label{cosa_eq:case1:geo}
    \begin{split}
          R|v|_p \cos a(x,v) &= X_0(x,v)_p \cdot v_p
        =(x_p+t_f(x,v)v_p) \cdot v_p. 
    \end{split}
    \end{align}
    Applying \eqref{t/x_C1,2},
      \begin{align*} 
        r|v|_p|\cos b(x,v) - \cos b(\tx,v)| &\leq
       |v|_p^2|t_b(x,v)-t_b(\tx,v)| \lesssim  |v|_p^2|x-\tx| \int_0^1 \left|\frac{X_{1}(\X(\tau),v)_p \cdot \frac{\hat{\X}}{|\hat{\X}|}}{X_{1}(\X(\tau),v)_p \cdot v}\right| d\tau.
    \end{align*}
 From \eqref{detail a,b:1}, 
    \begin{align} \label{est:sinb:x}
    \begin{split}
        \quad |\sin b(x,v) - \sin b(\tx,v)|  
        \lesssim |x-\tx|, \;  |\sin a(x,v) - \sin a(\tx,v)|
        \lesssim |x-\tx|.
    \end{split}
    \end{align}
    Because $b(x,v), b(\tx,v) \in [0,\frac{\pi}{2}]$, we get
\begin{align}\label{b-b:x}
\begin{split}
    |b(x,v)-b(\tx,v)| &\leq \frac{\pi}{2}|\sin (b(x,v)-b(\tx,v))|  \\
    &\leq \frac{\pi}{2}|\sin b(x,v) - \sin b(\tx,v)| + \frac{\pi}{2}|\cos b(x,v) - \cos b(\tx,v)| \\
    &\lesssim |x-\tx|+|v|_p|x-\tx|\int_0^1 \left|\frac{X_{1}(\X(\tau),v)_p \cdot \frac{\hat{\X}}{|\hat{\X}|}}{X_{1}(\X(\tau),v)_p \cdot v}\right| d\tau.
\end{split}
\end{align}  
Assume that $(x,\V(\tau)) \in \text{$\mathcal{C}_1$}$ with moving clockwise for all $\tau \in [0,1]$, and denote $v = v+\zeta$ and $\tv = \tv +\zeta$. Using  \eqref{t/v_C1,2} and \eqref{cosb_eq:case1},
     \begin{align*}
          r|v|_p|\cos b(x,v) - \cos b(x,\tv)| &\leq
        |x|_p|v_p-\tv_p|+|v|_p^2|t_b(x,v)-t_b(x,\tv)| \notag \\ &\lesssim |v-\tv|\left(1+|v|_p^2\int_0^1 t_b(x,\V(\tau))\left|\frac{X_{1}(x,\V(\tau))_p \cdot \frac{\hat{\V}}{|\hat{\V}|}}{X_{1}(x,\V(\tau))_p \cdot v}\right| d\tau \right).
     \end{align*} From \eqref{detail a,b:1}, 
    \begin{align} \label{est:sinb:v}
    \begin{split}
        \quad |\sin b(x,v) - \sin b(x,\tv)|  
        \lesssim \frac{1}{|v|_p}|v-\tv|, \;  |\sin a(x,v) - \sin a(x,\tv)|
        \lesssim  \frac{1}{|v|_p}|v-\tv|.
    \end{split}
    \end{align} Similar to \eqref{b-b:x}, we get \eqref{est:b/v}. For $a(x,v)$ instead of $b(x,v)$, we can use similar arguments. 
\end{proof} 

\begin{lemma} \label{x-tx}
Assume that $(\X (\tau),v+\zeta) \in \mathcal{C}_1$ with moving clockwise for all $\tau \in [0,1]$ or $(\X (\tau),v+\zeta) \in \mathcal{C}_2$ with moving clockwise for all $\tau \in [0,1]$. For $|x-\bx|\leq 1$,
    \begin{align} \label{f-f:tx}
    \begin{split}
         &|f(s, X(s;t,x,v+\zeta), V(s;t,x,v+\zeta)) - f(s, X(s;t, \tilde{x}, v+\zeta ), V(s;t, \tilde{x}, v+\zeta ))| \\
          &\lesssim \left[\Big(1+|v+\zeta|_p(t-s)\Big)+\Big(|v+\zeta|_p+|v+\zeta|_p^2(t-s)\Big)\int_{0}^{1}
		\frac{1}{\mathfrak{S}_{sp}(\tau; x, \tx, v+\zeta)}
		d\tau\right]^{\gamma} \\ &\quad \times \Big(\mathcal{X}(\gamma,s,v,\zeta)+|v+\zeta|^{\gamma}_p \mathcal{V}(\gamma,s,v,\zeta) \Big) |x-\bx|^{\gamma}.
    \end{split}
    \end{align}
\end{lemma}
\begin{proof} We let $v=v+\zeta$, $m=m(s;t,\tx,v+\zeta), n=m(s;t,x,v+\zeta)$, and $m \geq n$.
    We first assume that $m \geq n+2$.
Using $ t_{n+1}(t,x,v) < t_m(t,\tx,v)$ and \eqref{t/x_C1,2}, we get 
    \begin{align*}
        m-n-1 &< \frac{|t_b(x,v)-t_b(\tx,v)|}{t_*(x,v)}+ (m-1)\frac{|t_*(x,v)-t_*(\tx,v)|}{t_*(x,v)} \\
        &\lesssim (t-s)|v|_p^2|x-\tx|  \int_{0}^{1}
		\frac{1}{\mathfrak{S}_{sp}(\tau; x, \tx, v+\zeta)}
		d\tau
    \end{align*} since $m\leq 2(m-1) \lesssim |v|_p(t-s)$ and $1/|v|_p\lesssim t_*(x,v) $. Using above equation and $m-n-1 \geq 1$, 
    \begin{align*}
        &|V(s;t,x,v) - V(s;t, \tilde x, v)|
        \lesssim (t-s)|v|_p^3|x-\tx|  \int_{0}^{1}
		\frac{1}{\mathfrak{S}_{sp}(\tau; x, \tx, v+\zeta)} 
		d\tau \\
    &|X(s;t,x,v) - X(s;t, \tilde x, v)| 
         \lesssim (t-s)|v|_p^2|x-\tx|  \int_{0}^{1}
		\frac{1}{\mathfrak{S}_{sp}(\tau; x, \tx, v+\zeta)}
		d\tau.
    \end{align*}
    Next, we assume that $m=n \geq 1$ and $(\X (\tau),v) \in \mathcal{C}_1$ with moving clockwise for all $\tau \in [0,1]$. Using \eqref{est:a/x}, \eqref{est:b/x} and $m \lesssim |v|_p(t-s)+1$,
    \begin{align} 
         &\quad |V(s;t,x,v) - V(s;t, \tilde x, v)| = 2|v|_p\Big|\sin \frac{1}{2}\Big(\theta_{V_m}(x,v)-\theta_{V_m}(\tx,v)\Big)\Big| \notag \\
         &\leq 2m|v|_p(|b(x,v)-b(\tx,v)|+|a(x,v)-a(\tx,v)|) \notag \\
         &\lesssim |x-\tx|\left[\Big(|v|_p+|v|_p^2(t-s)\Big)+\Big(|v|_p^2+|v|_p^3(t-s)\Big)\int_{0}^{1}
		\frac{1}{\mathfrak{S}_{sp}(\tau; x, \tx, v+\zeta)}
		d\tau\right], \label{v-v:m=n}
    \end{align} where $V(s;t,x,v)$ is in \eqref{coor1:V}. Similarly,
    \begin{align} 
        &\quad |X_m(x,v)-X_m(\tx,v)|_p \notag \\
        & \lesssim |x-\tx|\left[\Big(1+|v|_p(t-s)\Big)+\Big(|v|_p+|v|_p^2(t-s)|\Big)\int_{0}^{1}
		\frac{1}{\mathfrak{S}_{sp}(\tau; x, \tx, v+\zeta)}
		d\tau\right], \label{Xm-Xm:m=n}
    \end{align} where $X(s;t,x,v)$ is in \eqref{coor1:X}. 
      Applying \eqref{v-v:m=n}, \eqref{Xm-Xm:m=n} and $|t_m(t,\tx,v)-s| \leq t_*(\tx,v) \lesssim \frac{1}{|v|_p}$,
    \begin{align} 
        &\quad |X(s;t,x,v)-X(s;t,\tx,v)|\notag\\&=|X_m(x,v)-V(s;t,x,v)(t_m(t,x,v)-s)-X_m(\tx,v)+V(s;t,\tx,v)(t_m(t,\tx,v)-s)|_p \notag\\
       &\leq |X_m(x,v)-X_m(\tx,v)|_p+|V(s;t,x,v)-V(s;t,\tx,v)|_p|t_m(t,\tx,v)-s| \notag\\ & \quad +|v|_p|t_m(t,\tx,v)-t_m(t,x,v)| \notag\\
       &\lesssim
        |x-\tx|\left[\Big(1+|v|_p(t-s) \Big)  + \Big( |v|_p+|v|_p^2(t-s)\Big) \int_{0}^{1}
		\frac{1}{\mathfrak{S}_{sp}(\tau; x, \tx, v+\zeta)}
		d\tau\right].
    \end{align} If $m=n=0$, it hold that $|X(s;t,x,v)-X(s;t,\tx,v)|=|x-\tx|$ and $|V(s;t,x,v) - V(s;t, \tilde x, v)|=0$. 
    We assume that $m=n+1$.
    When $m= n+1$, we apply the specular boundary condition as \eqref{f-f:m=n+1:x:ns} in Lemma \ref{tx-bx} and $n \lesssim |v|_p(t-s)+1$ and \eqref{Xm-Xm:m=n}.
\end{proof}

\begin{lemma}  \label{v-tv}
  Assume that $(x,\V(\tau)) \in \mathcal{C}_1$ with moving clockwise for all $\tau \in [0,1]$ or $(x,\V(\tau)) \in \mathcal{C}_2$ with moving clockwise for all $\tau \in [0,1]$. For $|v-\bv| \leq 1$,
    \begin{align} \label{f-f:tv}
    \begin{split}
        &|f(s, X(s;t,x,v+\zeta), V(s;t,x,v+\zeta)) - f(s, X(s;t, x, \tv+\zeta ), V(s;t, x, \tv+\zeta ))|\\
        &\lesssim  
        \left[\Big(\frac{1}{|v+\zeta|_p}+(t-s)\Big)+\Big(|v+\zeta|_p+|v+\zeta|_p^2(t-s)\Big)\int_{0}^{1}
		\frac{1}{\mathfrak{S}_{vel}(\tau; x,v, \tv, \zeta)}
		d\tau
	\right]^{\gamma} \\ &\quad \times \Big(\mathcal{X}(\gamma,s,v,\zeta)+|v+\zeta|^{\gamma}_p\mathcal{V}(\gamma,s,v,\zeta)\Big) |v-\bv|^{\gamma}.
    \end{split}
    \end{align}
\end{lemma}
\begin{proof}
We can use similar arguments in Lemma \ref{x-tx}. Instead \eqref{t/x_C1,2},\eqref{est:a/x} and \eqref{est:b/x}, we use \eqref{t/v_C1,2}, \eqref{est:a/v} and \eqref{est:b/v}. 
\end{proof}

\subsection{Averaging Specular Singularity}

\begin{definition} \label{def:tau}
When $(\X(\tau),v) \in \text{$\mathcal{C}_1$}$ with moving clockwise for all $\tau \in [0,1]$, we define $\tau_{+}(x, \bx, v), \tau_{0}(x, \bx, v)$ as
\begin{equation} \label{tau+_1}
		\begin{split}
			&\hat{v}_p\cdot n(X_1(\X(\tau_{+}),v)) = 0, \quad \tau_{+} = \tau_{+}(x, \bx, v), \\ &\hat{v}_p\cdot n(X_1(\X(\tau_{0}),v)) = -1, \quad \tau_{0} = \tau_{0}(x, \bx, v)
		\end{split}
        \end{equation} in Figure \ref{tau_space}.
        When $(\X(\tau),v) \in \text{$\mathcal{C}_2$}$ with moving clockwise for all $\tau \in [0,1]$, we define $\tau_{+}(x, \bx, v), \tau_{0}(x, \bx, v)$ as
\begin{equation} \label{tau+_2}
		\begin{split}
			&\hat{v}_p\cdot n(X_0(\X(\tau_{+}),v)) = 0, \quad \tau_{+} = \tau_{+}(x, \bx, v), \\ &\hat{v}_p\cdot n(X_0(\X(\tau_{0}),v)) = 1, \quad \tau_{0} = \tau_{0}(x, \bx, v).
		\end{split}
        \end{equation}  
\end{definition}

\begin{figure} [t] 
\centering
\includegraphics[width=6cm]{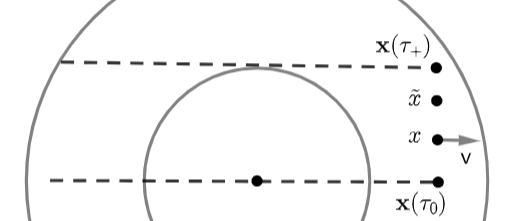}
\caption{$\X(\tau_{+})$ and $\X(\tau_{0})$ for given $v$.} \label{tau_space}
\end{figure}

\begin{lemma} \label{lemma:g_sp}
Assume $\tau_*\in \left[\min\{\tau_0, \tau_{+}\}, \max\{\tau_0, \tau_{+}\}\right]$, and $(\X(\tau),v) \in \text{$\mathcal{C}_1$}$ with moving clockwise for $\tau \in [0,1]$ or $(\X(\tau),v) \in \text{$\mathcal{C}_2$}$ with moving clockwise for $\tau \in [0,1]$. For $|x-\bx|\leq 1$, we have 
\begin{align} \label{g_sp:1}
    \left|\int^{\tau_{+}}_{\tau_{*}} \frac{1}{\mathfrak{S}_{sp}(\tau; x, \tx, v)} d\tau \right| \lesssim \left(\frac{1}{|v|_p}+\frac{1}{|v|_p\cos b(\X(\tau_*),v)}\right)|\tau_+-\tau_*|,
\end{align} where $\mathfrak{S}_{sp}(\tau; x, \tx, v)$ is in Definition \ref{def:g}, and $b(x,v)$ is in \eqref{def:angle:1} or \eqref{def:angle:2}.
\end{lemma}
\begin{proof}
    Assume $\tau_*\in [\text{min}\{\tau_0, \tau_{+}\}, \text{max}\{\tau_0, \tau_{+}\}]$ and $(\X(\tau),v) \in \text{$\mathcal{C}_1$}$ with moving clockwise for $\tau \in [0,1]$. For $a(x,v), b(x,v) \in [0,\pi/2]$, from \eqref{sine law}, we get 
    \begin{align} \label{sin<r/R}
       0 \leq \sin a(x,v) = \frac{r}{R}\sin b(x,v) \leq \frac{r}{R}.
    \end{align} We have
    \begin{align} \label{int_space_term1}
        \left|\frac{X_0(\X(\tau),v)_p\cdot \frac{\dot{\X}(\tau)}{|\dot{\X}(\tau)|}}{X_0(\X(\tau),v)_p \cdot v}\right| = \frac{\sin a(\X(\tau),v)}{|v|_p \cos a(\X(\tau),v)}\leq \frac{r}{|v|_p\sqrt{R^2-r^2}} \lesssim \frac{1}{|v|_p}.
    \end{align} 
Let $v=(v_1,0,v_3)$ and $v_1>0$ without loss of generality and $\X(\tau)_p = (|\X(\tau)|_p \cos \Theta(\tau), |\X(\tau)|_p \sin \Theta(\tau),0)=(\X(\tau)_1, \X(\tau)_2,0)$. From \eqref{detail a,b:1}, we have
\begin{align} \label{sin b(X,tau)}
    \sin b(\X(\tau),v) = \frac{|\X(\tau)|_p}{r}\sin(\Theta(\tau)-0) = \frac{\X(\tau)_2}{r}.
\end{align}
Recall $\X(\tau)_2=(1-\tau)\tx_2 + \tau x_2$ for $\tx=(\tx_1,\tx_2,\tx_3), x=(x_1,x_2,x_3)$ as in \eqref{def:x_para}. 
 From $\X(\tau_*)_2-\X(\tau_{+})_2=(\tau_{+}-\tau_{*})(\tx_2-x_2)$, we have
\begin{align} \label{1/x-x}
    \frac{1}{\tx_2-x_2} = \frac{\tau_{+}-\tau_{*}}{\X(\tau_*)_2-\X(\tau_{+})_2}= \frac{\tau_{+}-\tau_{*}}{r(\sin b(\X(\tau_*),v)-\sin b(\X(\tau_{+}),v))}= \frac{\tau_{+}-\tau_{*}}{r(\sin b(\X(\tau_*),v)-1)}.
\end{align}
Therefore, by \eqref{sin b(X,tau)} and \eqref{1/x-x}, we obtain
\begin{align} \label{int_space_term2}
\begin{split}
      \left |\int^{\tau_{+}}_{\tau_{*}} \left|\frac{X_1(\X(\tau),v)_p\cdot \frac{\dot{\X}(\tau)}{|\dot{\X}(\tau)|}}{X_1(\X(\tau),v)_p \cdot v} \right | d\tau \right | 
      &=  \left|\int^{\tau_{+}}_{\tau_{*}} \frac{\X(\tau)_2}{|v|_p \sqrt{r^2-\X(\tau)^2_2}} d\tau \right|
      =  \frac{r}{|v|_p|x_2-\tx_2| } \cos b(\X(\tau_*),v) \\
      &\lesssim \frac{1}{|v_p|} \frac{\cos b(\X(\tau_*),v)}{1-\sin b(\X(\tau_*),v)}|\tau_{+}-\tau_{*}|\lesssim \frac{1}{|v|_p\cos b(\X(\tau_*),v) }|\tau_{+}-\tau_{*}|
\end{split}
\end{align} since $v_p \;\bot\; \dot{\X}(\tau)$ and $\cos b(\X(\tau_{+},v))=0$. Combinig \eqref{int_space_term1} and \eqref{int_space_term2}, we deduce \eqref{g_sp:1}.
\end{proof}

\begin{definition} \label{def:tau:v}
When $X(s;t,x,\V(\tau))$ with moving clockwise for all $\tau \in [0,1]$, we define $\tau_{+}(x, v, \bv, \zeta)$ and $\tau_{0+}(x, v, \bv, \zeta)$ as
 \begin{align} \label{tau+_1:v}
		\begin{split}
			&\widehat{\V(\tau_{+})}\cdot n(X_1(x,\V(\tau_{+}))) = 0, \quad \tau_{+} = \tau_{+}(x, v, \bv, \zeta),\\
              &\widehat{\V(\tau_{0+})}\cdot n(X_1(x,\V(\tau_{0+}))) = -1, \quad \tau_{0+} = \tau_{0+}(x, v, \bv, \zeta),
		\end{split}
\end{align} where $(x,\V(\tau_{+})), (x,\V(\tau_{0+})) \in \text{$\mathcal{C}_1$}$.(See the left side of the Figure \ref{tau_vel}). 
 We define $\tau_{-}(x, v, \bv, \zeta)$ and $\tau_{0-}(x, v, \bv, \zeta)$ as
\begin{align} \label{tau+_2:v}
\begin{split}
   &\widehat{\V(\tau_{-})}\cdot n(X_0(x,\V(\tau_{-}))) = 0, \quad \tau_{-} = \tau_{-}(x, v, \bv, \zeta),\\
&\widehat{\V(\tau_{0-})}\cdot n(X_0(x,\V(\tau_{0-}))) = 1, \quad \tau_{0-} = \tau_{0-}(x, v, \bv, \zeta), 
\end{split}
\end{align} where $(x,\V(\tau_{-})),(x,\V(\tau_{0-})) \in \text{$\mathcal{C}_2$}$.(See the right side of the Figure \ref{tau_vel}).
\end{definition} 

\begin{figure} [t] 
\centering
\includegraphics[width=10cm]{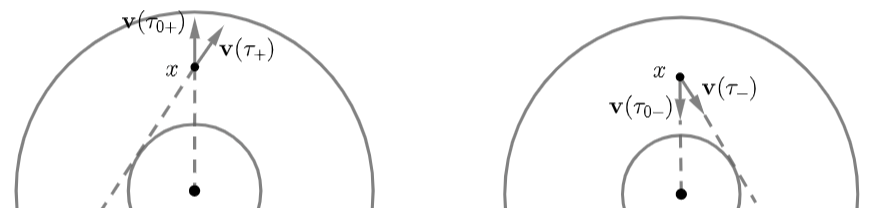}
\caption{$\V(\tau_{+}),\V(\tau_{0+}), \V(\tau_{-})$  and $\V(\tau_{0-})$ for given $x$.} \label{tau_vel}
\end{figure}

\begin{lemma}\label{lemma:g_vel}
(i) Assume $\tau_*\in \left[\min\{\tau_{0+}, \tau_{+}\}, \max\{\tau_{0+}, \tau_{+}\} \right]$ and $(x,\V(\tau)) \in \text{$\mathcal{C}_1$}$ with moving clockwise. Then we have 
\begin{align} \label{g_vel:1}
    \left|\int^{\tau_{+}}_{\tau_{*}} \frac{1}{\mathfrak{S}_{vel}(\tau;x,v, \tv, \zeta)} d\tau \right| \lesssim \left(\frac{1}{|v+\zeta|^2_p}+\frac{1}{\cos b(x,\V(\tau_{*}))}\frac{1}{|v+\zeta|^2_p}\right)|\tau_{+}-\tau_{*}|,
\end{align} where $\mathfrak{S}_{vel}(\tau;x,v, \tv, \zeta)$ is in Definition \ref{def:g}, and $b(x,v)$ is in \eqref{def:angle:1}. \\ \\ 
(ii) Assume $\tau_*\in [\min\{\tau_{0-}, \tau_{-}\}, \max\{\tau_{0-}, \tau_{-}\}]$ and $(x,\V(\tau)) \in \text{$\mathcal{C}_2$}$ with moving clockwise. Then we have 
\begin{align} \label{g_vel:2}
    \left|\int^{\tau_{-}}_{\tau_{*}} \frac{1}{\mathfrak{S}_{vel}(\tau;x,v, \tv, \zeta)} d\tau\right| \lesssim \left(\frac{1}{|v+\zeta|^2_p}+\frac{1}{\cos b(x,\V(\tau_{*}))}\frac{1}{|v+\zeta|^2_p}\right)|\tau_{-}-\tau_{*}|,
\end{align} where $b(x,v)$ is in \eqref{def:angle:2}.

\end{lemma}
\begin{proof}
    (i) First, by \eqref{sin<r/R}, we obtain
    \begin{align} \label{est:g:1:1:v}
    \begin{split}
         t_f(x,\V(\tau))
    \left|\frac{X_0(x,\V(\tau))_p\cdot \frac{\dot{\V}(\tau)}{|\dot{\V}(\tau)|}}{X_0(x,\V(\tau))_p \cdot \V(\tau)}\right| \lesssim  \frac{1}{|v+\zeta|^2_p} \frac{1}{\sqrt{1-(r/R)^2}}  \lesssim \frac{1}{|v|_p}.
    \end{split}
    \end{align} 
    We let $x=(0,x_2,x_3)$ and $x_2>0$ without loss of generality and 
    \begin{align*}
        &(\tv+\zeta)_p = (|v+\zeta|_p \cos \theta(\tv), |v+\zeta|_p \sin \theta(\tv),0), \\
        &(v+\zeta)_p = (|v+\zeta|_p \cos \theta(v), |v+\zeta|_p \sin \theta(v),0).
    \end{align*} Let us $\Theta(0)=\theta(\tv)$ and $\Theta(1)=\theta(v)$.
   We redefine $\V(\tau)_p =(|v+\zeta|_p \cos \Theta(\tau), |v+\zeta|_p \sin \Theta(\tau),0)$ for $\Theta(\tau) = (\Theta(1)-\Theta(0))\tau + \Theta(0)$ for convenience. From \eqref{detail a,b:1}, we have
\begin{align*}
    \sin b(x,\V(\tau)) = \frac{|x|_p}{r}\cos \Theta(\tau).
\end{align*}
If $t=|x|_p \cos \Theta(\tau)$, the Jacobian determinant is obtained by 
\begin{align*}
    \left|\frac{d\tau}{dt}\right|= \frac{1}{|x|_p \sin \Theta(\tau)}\frac{1}{|\Theta(1)-\Theta(0)|} \in (0,\infty).
\end{align*} 
Using $\V(\tau)_p \; \bot \; \dot{\V}(\tau)$ and performing change of variables, 
\begin{align} \label{est:g:temp}
\begin{split}
    &\left|\int^{\tau_{+}}_{\tau_{*}}  t_b(x,\V(\tau))
    \left | \frac{X_1(x,\V(\tau))_p\cdot \frac{\dot{\V}(\tau)}{|\dot{\V}(\tau)|}}{X_1(x,\V(\tau))_p \cdot \V(\tau)} \right | d\tau\right| \\
     &=\frac{1}{|v+\zeta|_p}\left|\int^{\tau_{+}}_{\tau_{*}}  t_b(x,\V(\tau))
    \frac{\sin b(x,\V(\tau))}{\cos b(x, \V(\tau))} d\tau\right| \\
    &= \frac{1}{|v+\zeta|_p|\Theta(1)-\Theta(0)|}\left|\int^{|x|_p \cos \Theta(\tau_{+})}_{|x|_p \cos \Theta(\tau_{*})} t_b(x,\V(\tau))\frac{1}{\sin \Theta(\tau)} \frac{t}{\sqrt{r^2-t^2}} dt \right|.
\end{split} 
\end{align}
From \eqref{tau+_1:v} and Figure \ref{tau_vel}, we have
\begin{align} \label{ave_sin_ine}
     \sin \Theta(\tau_{+})=\frac{\sqrt{|x|^2_p-r^2}}{|x|_p} \leq  \sin \Theta(\tau) \leq  \sin \Theta(\tau_{0+})=1
\end{align} since $\tau$ is in between $\tau_{+}$ and $\tau_{0+}$.
Also, we have 
\begin{align} \label{ave_tb_ine}
    t_b(x,\V(\tau)) \leq t_b(x,\V(\tau_{+})) \leq \frac{\sqrt{|x|^2_p-r^2}}{|v+\zeta|_p}
\end{align}  since $\tau$ is in between $\tau_{+}$ and $\tau_{0+}$. Then,  
\begin{align*}
    \eqref{est:g:temp} &\leq \frac{|x|_p}{|v+\zeta|^2_p|\Theta(1)-\Theta(0)|}\left |\int^{|x|_p \cos \Theta(\tau_{+})}_{|x|_p \cos \Theta(\tau_{*})}  \frac{t}{\sqrt{r^2-t^2}} dt \right| \\
  &\leq \frac{|x|_p}{|v+\zeta|^2_p|\Theta(1)-\Theta(0)|} \cos b(x,\V(\tau_*))
\end{align*} since $\cos \Theta(\tau_{+}) = r/ |x|_p$. It holds that
\begin{align*}
    &|(\Theta(1)-\Theta(0))(\tau_{*}-\tau_{+})|=|\Theta(\tau_{*})-\Theta(\tau_{+})|\\ &\geq |\cos \Theta(\tau_{*})-\cos \Theta(\tau_{+})|\geq \frac{r}{|x|_p} | \sin b(x,\V(\tau_{*}))-1|.
\end{align*} Therefore, we obtain 
\begin{align} \label{est:g:1:2:v}
\begin{split}
     &\left|\int^{\tau_{+}}_{\tau_{*}}  t_b(x,\V(\tau))
    \left|\frac{X_1(x,\V(\tau))_p\cdot \frac{\dot{\V}(\tau)}{|\dot{\V}(\tau)|}}{X_1(x,\V(\tau))_p \cdot \V(\tau)} \right| d\tau \right| \lesssim \frac{1}{|v+\zeta|_p^2}\frac{1}{\cos b(x,\V(\tau_{*})}|\tau_{+}-\tau_{*}|.
\end{split}
\end{align}  Combining \eqref{est:g:1:1:v} and \eqref{est:g:1:2:v}, we deduce \eqref{g_vel:1}. 
\end{proof}

\section{Difference quotients estimates for \texorpdfstring{$\mathcal{C}_3$}{}}

In section 4, we let $f : \overline{\O} \times \R^{3} \rightarrow \R_{+} \cup \{0\}$ be a function which satisfies specular reflection \eqref{specular}, where $\O$ is a domain as in Definition \ref{def:domain} and $w(v) = e^{\vartheta|v|^{2}}$ for some $0 < \vartheta$. Let fix any point $t,s \in \mathbb{R}^{+}$, $v, \bv, \zeta \in \mathbb{R}^3$, $x,\bx \in \O$. We assume \eqref{assume_x}, \eqref{assume_v}, and recall $\tx(x,\bx,v+\zeta)$ in \eqref{def:tx}, $\X(\tau)$ in \eqref{def:x_para}, $\tv(v,\bv,\zeta)$ in \eqref{def:tv} and $\V(\tau)$ in \eqref{def:v_para}. 
\subsection{Nonsingular part of $\mathcal{C}_3$}
\begin{lemma} \label{tx-bx:3}
Assume that $(\tx,v+\zeta), (\bx,v+\zeta) \in \mathcal{C}_3$ with moving clockwise. For $|x-\bx|\leq 1$,
 \begin{align}\label{f-f:bx:3}
 \begin{split}
          &|f(s, X(s;t,\tx,v+\zeta), V(s;t,\tx,v+\zeta)) - f(s, X(s;t, \bx, v+\zeta ), V(s;t, \bx, v+\zeta ))| \\
          &\lesssim \left(\mathcal{X}(\gamma,s,v,\zeta)  + |v+\zeta|_p^{\gamma} \mathcal{V}(\gamma,s,v,\zeta)  \right)|x-\bx|^{\gamma}.
    \end{split}
    \end{align} 
\end{lemma} 
\begin{proof}
    The proof is the same as in Lemma \ref{tx-bx}.
\end{proof}

\begin{lemma} \label{tv-bv:3}
  Assume that $(x,\tv+\zeta), (x,\bv+\zeta) \in \mathcal{C}_3$ with moving clockwise. For $|v-\bv| \leq 1$,
   \begin{align} \label{f-f:bv:3}
    \begin{split}
          &\quad|f(s, X(s;t,x,\tv+\zeta), V(s;t,x,\tv+\zeta)) - f(s, X(s;t, x, \bv+\zeta ), V(s;t, x, \bv+\zeta ))|  \\
          &\lesssim \;
          \Bigg(\Big[(t-s)+\max \{|v+\zeta|_p^{-1},|\bv+\zeta|_p^{-1}\}\Big]^{\gamma}\mathcal{X}(\gamma,s,v,\zeta) \\ &\quad\quad +\Big[1+(t-s) \max \{|v+\zeta|_p,|\bv+\zeta|_p\}\Big]^{\gamma}\mathcal{V}(\gamma,s,v,\zeta) \Bigg)|v-\bv|^{\gamma}
  \end{split}
    \end{align}
\end{lemma}
\begin{proof}
    We let $\tv=\tv+\zeta$, $\bv=\bv+\zeta$, $m(s;t,x,\bv)=m, m(s;t,x,\tv)=n$, and $m \geq n$. We have
    \begin{align} \label{t-t:tvbv:3}
    \begin{split}
        &|t_b(x,\tv)-t_b(x,\bv)|=|X_{1}(x,\tv)-x|_p\Big|\frac{1}{|\tv|_p}-\frac{1}{|\bv|_p}\Big| \lesssim \frac{|v-\bv|}{|\bv|_p|\tv|_p}, \\
        &|t_*(x,\tv)-t_*(x,\bv)|=|2R \cos a(x,\tv)|\Big|\frac{1}{|\tv|_p}-\frac{1}{|\bv|_p}\Big| \lesssim  \cos a(x,\tv) \frac{|v-\bv|}{|\bv|_p|\tv|_p}.
    \end{split}
    \end{align} By definition of $m(s;t,x,v)$ in \eqref{def:m}, $m,n \in \mathbb{N} \cup \{0\}$  satisfy
    \begin{align} \label{mn:tvbv:3}
    \begin{split}
        & m-1 \leq \frac{t-s-t_b(x,\bv)}{t_*(x,\bv)} = (t-s-t_b(x,\bv))\frac{|\bv|_p}{l(x,\bv)} \leq \frac{|\bv|_p(t-s-t_b(x,\bv))}{2R\cos a(x,\tv)}, \\
        & n > \frac{t-s-t_b(x,\tv)}{t_*(x,\tv)} = (t-s-t_b(x,\tv))\frac{|\tv|_p}{l(x,\tv)} \geq \frac{|\tv|_p(t-s-t_b(x,\tv))}{2R\cos a(x,\tv)}
    \end{split}
    \end{align} since $\cos a(x,\bv), \cos a(x,\tv)>0$. From  \eqref{coor3:X} and \eqref{coor3:V},
    \begin{align*}
        \theta_{X_m}(x,\bv)-\theta_{X_n}(x,\tv)=\theta_{V_m}(x,\bv)-\theta_{V_n}(x,\tv)=(m-n)(\pi-2a(x,\tv)).
    \end{align*}
     Thus, we have
    \begin{align} \label{V:tvbv:3}
        \quad |V(s;t,x,\bv)-V(s;t,x,\tv)| 
        \leq |\tv|_p \left|\sin \Big((m-n)\left(\frac{\pi}{2}-a(x,\tv)\right)\Big) \right|+|\tv-\bv|
    \end{align} and 
    \begin{align}\label{X:tvbv:3}
        &\quad |X(s;t,x,\bv)-X(s;t,x,\tv)| \notag \\
        &\leq |X_m(x,\bv)-X_n(x,\tv)|_p+|V_m(x,\bv)-V_n(x,\tv)|_p|t_n(t,x,\tv)-s|+|\bv|_p|t_n(t,x,\tv)-t_m(t,x,\bv)| \notag \\
        &\quad + |\tv_3-\bv_3|(t-s) \notag \\
        &\lesssim \left|\sin \Big( (m-n)\left(\frac{\pi}{2}-a(x,\tv) \right)\Big)\right |+|\bv|_p|t_b(x,\tv)-t_b(x,\bv)| \notag\\ &\quad +|\bv|_p|(n-1)(t_*(x,\tv)-t_*(x,\bv))|+(m-n)|\bv|_pt_*(x,\bv)+|v-\bv|\Big((t-s)+\frac{1}{|\tv|_p}\Big)
    \end{align} since $|t_n(t,x,\tv)-s|\leq t_*(x,\tv) \lesssim \frac{1}{|\tv|_p}$. \\ \\
    We first assume that $m \geq n+2$. Applying \eqref{t-t:tvbv:3} and \eqref{mn:tvbv:3}, 
     \begin{align} \label{m-n:tvbv:3}
          (m-n)\left(\frac{\pi}{2}-a(x,\tv)\right)  \leq 
         2(m-n-1)\left(\frac{\pi}{2}-a(x,\tv)\right) 
         \leq |v-\bv|\Big((t-s)+\frac{1}{|\bv|_p}\Big)
     \end{align} since $\cos a(x,\tv) \geq \frac{2}{\pi}\Big(\frac{\pi}{2}-a(x,\tv)\Big)$. Similarly,
      \begin{align}  \label{X-X:tvbv:3:sub2}
         (m-n)t_*(x,\bv)|\bv|_p
         \leq |v-\bv|\Big((t-s)+\frac{1}{|\bv|_p}\Big)
     \end{align} by using $t_*(x,\bv)=2R \cos a(x,\tv)/|\bv|_p$.
     Applying \eqref{m-n:tvbv:3} to \eqref{V:tvbv:3},
     \begin{align*}
        \quad |V(s;t,x,\bv)-V(s;t,x,\tv)| 
        \leq |v-\bv|\Big( \max\{|\bv|_p,|\tv|_p\}(t-s)+1 \Big).
     \end{align*}  Applying \eqref{t-t:tvbv:3}, \eqref{m-n:tvbv:3} and \eqref{X-X:tvbv:3:sub2} to \eqref{X:tvbv:3},
     \begin{align*}
        |X(s;t,x,\bv)-X(s;t,x,\tv)| 
        \lesssim |v-\tv|\left((t-s)+\max \left\{\frac{1}{|v|_p},\frac{1}{|\bv|_p} \right\}\right).
     \end{align*} 
     Next, we assume $m=n$. From \eqref{V:tvbv:3} and \eqref{X:tvbv:3}, 
     \begin{align}\label{tvv:3:m=n:v}
     \begin{split}
           &\quad |V(s;t,x,\bv)-V(s;t,x,\tv)| \leq |v-\bv|  \\
            &\quad |X(s;t,x,\bv)-X(s;t,x,\tv)|\leq |v-\bv|\left((t-s)+\max \left\{\frac{1}{|v|_p},\frac{1}{|\bv|_p}\right\}\right).
    \end{split}
     \end{align} 
  When $m= n+1$, we apply the specular boundary condition as \eqref{f-f:m=n+1:x:ns} in Lemma \ref{tx-bx}, and omit details. 
\end{proof} 
\subsection{Singular part of $\mathcal{C}_3$}

Ko and Lee obtained $\nabla_x X(0;t;x,v), \nabla_x V(0;t,x,v)$, $\nabla_v X(0;t,x,v)$, and $\nabla_v V(0;t,x,v)$ in a disk with the specular reflection boundary condition in \cite{GD2022}. Before starting this section, we recall Lemma 5.4 of \cite{GD2022}.

\begin{lemma}\label{disk} Let $(t,x,v) \in \mathbb{R}^{+}\times O \times \mathbb{R}^2$, where $O =\{x \in \mathbb{R}^2 | |x|=1\}$. We assume that $n(\xb) \cdot v \neq 0$, where $n(\xb)$ is outward unit normal vector at $\xb=x-\tb v \in \partial O$. Then we have estimates of derivatives for the characteristics $X(0;t,x,v)$ and $V(0;t,x,v)$
\begin{align*}
    &|\nabla_x X(0,t,x,v)| \lesssim \frac{|v|}{|v \cdot n(\xb)|}(1+|v|t), \\
    &|\nabla_v X(0,t,x,v)| \lesssim \frac{1}{|v|}(1+|v|t), \\
    &|\nabla_x V(0,t,x,v)| \lesssim \frac{|v|^3}{|v \cdot n(\xb)|^2}(1+|v|t), \\
    &|\nabla_v V(0,t,x,v)| \lesssim \frac{|v|}{|v \cdot n(\xb)|}(1+|v|t).
\end{align*} 
\end{lemma}
\begin{proof}
    We refer Lemma 5.4 of \cite{GD2022}.
\end{proof}

\begin{figure}[t]
\centering
\includegraphics[width=6cm]{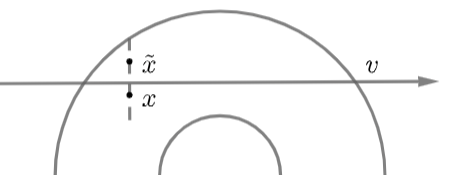}
\caption{ $(\X(\tau),v) \in \mathcal{C}_3$ for given $v$.}\label{sing:3:x}
\end{figure} 

\begin{lemma}\label{x-tx:3}
 Assume that $(\X(\tau),v) \in \mathcal{C}_3$ with moving clockwise for all $\tau \in [0,1]$. For $|x-\bx|\leq 1$,
\begin{align} \label{f-f:tx:3}
\begin{split}
          |X(s;t,x,v)-X(s;t,\tx,v)|
         \lesssim \max \left\{\frac{1}{\cos a(x,v)}, \frac{1}{\cos a(\tx,v)} \right\}(1+|v|_p(t-s))|x-\bx|, \\
         |V(s;t,x,v)-V(s;t,\tx,v)|
         \lesssim \max \left\{\frac{1}{\cos^2 a(x,v)}, \frac{1}{\cos^2 a(\tx,v)} \right \}(|v|_p+|v|_p^2(t-s))|x-\bx|, 
\end{split}
     \end{align} where $a(x,v)$ is in \eqref{def:angle:3}.
\end{lemma}
\begin{proof}
     We let $v=(|v|_p\cos \theta_v,|v|_p\sin \theta_v, v_3)$, and $\theta_v=0$ without loss of generality. 
     Next, we also denote $v=(v_1,0,v_3)$. From Lemma \ref{disk},
    \begin{align*}
        |\nabla_{x_1,x_2} X(s;t,x_p,v_p)| \lesssim \frac{|v|_p}{|v_p \cdot n({x}_{p,b})|}(1+|v|_p(t-s)) = \frac{1+|v|_p(t-s)}{\cos a(x,v)}.
    \end{align*} Since $X(s;t,x,v)_{ver}=(0,0,x_3-(t-s)v_3)$ and $V(s;t,x,v)=v_3$ for $x=(x_1,x_2,x_3), v=(v_1,v_2,v_3)$, 
    \begin{align*} 
         \nabla_{x_3} X(s;t,x,v)=1 \text{\quad and \quad}
          \nabla_{x_3} V(s;t,x,v)=0.
    \end{align*}
     Using \eqref{detail a,b:1} and Figure \ref{sing:3:x}, we get 
    \begin{align*}
        \max_{\tau}\left\{\frac{1}{\cos a(\X(\tau),v)}\right\} = \max \left\{\frac{1}{\cos a(x,v)}, \frac{1}{\cos a(\tx,v) } \right\}.
    \end{align*} Therefore, 
    \begin{align} \notag
		&|X(s;t,x,v )- X(s;t, \tx, v )| \leq \int^{1}_{0} 
		|\nabla_x X(s;t,\X(\tau),v ) \dot{\X}
		    |\;d\tau  \notag \\
		& \lesssim 
		\max \left\{\frac{1}{\cos a(x,v)}, \frac{1}{\cos a(\tx,v)} \right\}(1+|v|_p(t-s))|x-\tx|.
		\notag  
		\end{align} We can use the same arguments for the inequality of $|V(s;t,x,v)-V(s;t,\tx,v)|$ in \eqref{f-f:tx:3}.
\end{proof}

\begin{lemma} \label{v:est:angle,time:case3} Assume that $(x,v+\zeta),(x,\tv+\zeta) \in \mathcal{C}_3$ with moving clockwise, and $x_p \cdot v$ and $x_p \cdot \tv$ have the same sign. We have
    \begin{align} \label{v:est:t:case3}
            |t_i(x,v+\zeta)-t_i(x,\tv+\zeta)| \lesssim \frac{1}{|v+\zeta|^{2}_p} |v-\tv| 
    \end{align} for $i=b,f,*$, and
    \begin{align} \label{v:est:a:case3}
        \begin{split}
         |a(x,v+\zeta)-a(x,\tv+\zeta)|\lesssim \frac{1}{|v+\zeta|_p} |v-\tv|,
        \end{split} 
    \end{align} where $a(x,v)$ is in \eqref{def:angle:3}.
\end{lemma}
\begin{proof}
From $|X_0(x,v)|_p=|(x_1+t_f(x,v)v_1,x_2+t_f(x,v)v_2,0)|=R$, we compute
    \begin{align} \label{tf_eq:geo:case1,2}
         |v|_p^2 t_f(x,v) = -(x_p \cdot v_p) + \sqrt{(x_p \cdot v_p)^2+|v|_p^2(R^2-|x|_p^2)}.
    \end{align} We have
    \begin{align} \label{tf-tf:geo:Case3}
    \begin{split}
        &\quad |v|_p^2 |t_f(x,v)-t_f(x,\tv)| 
       \\ &\leq |(v_p-\tv_p)\cdot x_p|+\left|\sqrt{(x_p \cdot v_p)^2+|v|_p^2(R^2-|x|_p^2)}-\sqrt{(x_p \cdot \tv_p)^2+|\tv|_p^2(R^2-|x|_p^2)} \right|  \\
       &\leq |(v_p-\tv_p)\cdot x_p|+\left|\sqrt{(x_p \cdot v_p)^2}-\sqrt{(x_p \cdot \tv_p)^2} \right| 
       \lesssim |v-\tv|.
    \end{split} 
    \end{align} Using \eqref{cosa_eq:case1:geo} and \eqref{est:sinb:v}, we obtain \eqref{v:est:a:case3} directly from \eqref{v:est:t:case3}.
\end{proof}

\begin{lemma}\label{v-tv:3}
Assume that $(x,\V(\tau)) \in \mathcal{C}_3$ with moving clockwise for all $\tau \in [0,1]$.  For $|v-\bv| \leq 1$,
   \begin{align}  \label{f-f:tv:3}
   \begin{split}
          &|f(s, X(s;t,x,v+\zeta), V(s;t,x,v+\zeta)) - f(s, X(s;t, x, \tv+\zeta ), V(s;t, x, \tv+\zeta ))|\\
          &\lesssim
          \bigg(\left[\left(\frac{1}{|v+\zeta|_p}+(t-s)\right)  
	\right]^{\gamma}\mathcal{X}(\gamma,s,v,\zeta)\\
	&\quad+ \left[ \Big(1+|v+\zeta|_p (t-s)\Big) \times \max \left\{\frac{1}{\cos a(x,v+\zeta)},\frac{1}{\cos a(x,\tv+\zeta)} \right\}\right]^{\gamma}\mathcal{V}(\gamma,s,v,\zeta) \bigg) |v-\bv|^{\gamma}
    \end{split}
    \end{align}  where $a(x,v)$ is in \eqref{def:angle:3}.
\end{lemma}
\begin{proof}
We let $v=v+\zeta, \; \tv=\tv+\zeta,\; m=m(s;t,x,\tv+\zeta)$ and $ n=m(s;t,x,v+\zeta)$. We first assume that $x_p \cdot v$ and $x_p \cdot \tv$ have the same sign. By definition of $m(s;t,x,v)$ in \eqref{def:m}, $m,n \in \mathbb{N} \cup \{0\}$  satisfy
    \begin{align} \label{mn:tvv:3}
    \begin{split}
         &\quad \quad \quad m-1 \leq \frac{|v|_p(t-s-t_b(x,\tv))}{2R\cos a(x,\tv)}, \\
         &\frac{|v|_p(t-s-t_b(x,v))}{2R\cos a(x,v)} < n \leq  \frac{|v|_p(t-s-t_b(x,v))}{2R\cos a(x,v)}+1
    \end{split}
    \end{align} for $\cos a(x,v), \cos a(x,\tv) >0$.
From \eqref{coor3:V}, 
    \begin{align} \label{V-V:tvv:case3}
    \begin{split}
        &\quad |V(s;t,x,\tv)-V(s;t,x,v)| = 2|v|_p\Big|\sin \frac{1}{2} \left( \theta_{V_m}(x,\tv)-\theta_{V_n}(x,v) \right)\Big| \\
        &\leq |v-\tv|+2|v|_p\left|(m-n)\left(\frac{\pi}{2}-a(x,\tv)\right) \right| + 2|v|_p|n(a(x,v)-a(x,\tv))|.
    \end{split}
    \end{align} From \eqref{coor3:X},
    \begin{align} \label{theta-theta:tvv:3}
        \theta_{X_{m}}(x,\tv)-\theta_{X_{n}}(x,v)=\theta_{V_m}(x,\tv)-\theta_{V_n}(x,v)+a(x,\tv)-a(x,v).
    \end{align} Thus, we have
    \begin{align}\label{X:tvv:3}
    \begin{split}
        &\quad |X(s;t,x,\tv)-X(s;t,x,v)| \\
        &\leq 2|V_m(x,\tv)-V_n(x,v)|_p|v|_p^{-1}+|v|_p|t_n(t,x,v)-t_m(t,x,\tv)|+|a(x,v)-a(x,\tv)| \\
        &\leq 2|V_m(x,\tv)-V_n(x,v)|_p|v|_p^{-1}+|v|_p|t_b(x,\tv)-t_b(x,v)|\\ &\quad +|v|_p|(n-1)(t_*(x,\tv)-t_*(x,v))|+(m-n)|v|_pt_*(x,\tv)+|a(x,v)-a(x,\tv)|.
    \end{split}
    \end{align}  
    We assume that $m \geq n+2$. From \eqref{v:est:t:case3}, \eqref{v:est:a:case3}, and \eqref{mn:tvv:3},
     \begin{align} \label{a-a:tvv:3}
         n|a(x,\tv)-a(x,v)| \lesssim |v-\tv|\left(t-s+\frac{1}{|v|_p}\right)\frac{1}{\cos a(x,v)}
     \end{align} and
     \begin{align}
       |v|_p|(n-1)(t_*(x,\tv)-t_*(x,v))|\lesssim |v-\tv|\left(t-s+\frac{1}{|v|_p} \right)\frac{1}{\cos a(x,v)}. \label{t-t:tvv;3}
     \end{align} From \eqref{mn:tvv:3} and \eqref{a-a:tvv:3},
    \begin{align} \label{m-n:tvv:3}
    \begin{split}
         &(m-n)\left(\frac{\pi}{2}-a(x,\tv)\right) \leq 2(m-n-1)\left(\frac{\pi}{2}-a(x,\tv)\right) \\
         &\lesssim  
          |v-\tv|\left(t-s+\frac{1}{|v|_p} \right)\frac{1}{\cos a(x,v)}
    \end{split}
     \end{align} since $\cos a(x,\tv) \geq \frac{2}{\pi}\Big(\frac{\pi}{2}-a(x,\tv)\Big)$. And
     \begin{align} \label{X-X:tvv:3:sub2}
         &\quad (m-n)t_*(x,\tv)|v|_p \lesssim |v-\tv|  \left(t-s+\frac{1}{|v|_p} \right)\frac{1}{\cos a(x,v)}
     \end{align} since $t_*(x,\tv)=2R \cos a(x,\tv)/|v|_p$. Applying \eqref{a-a:tvv:3}, \eqref{m-n:tvv:3} to \eqref{V-V:tvv:case3},
     \begin{align} \label{V-V:tvv:m=n+1:case3}
         \quad |V(s;t,x,\tv)-V(s;t,x,v)|
        \lesssim |v-\tv|( 1+|v|_p(t-s)) \frac{1}{\cos a(x,v)}.
     \end{align} 
     Applying \eqref{t-t:tvv;3}, \eqref{X-X:tvv:3:sub2}, and \eqref{V-V:tvv:m=n+1:case3} to \eqref{X:tvv:3}, 
     \begin{align*}
        \quad |X(s;t,x,\tv)-X(s;t,x,v)| 
        \lesssim |v-\tv| \left(t-s+\frac{1}{|v|_p} \right) \frac{1}{\cos a(x,v)}.
     \end{align*} 
      Next, we assume $m=n$.
      From \eqref{V-V:tvv:case3} and \eqref{X:tvv:3}, we get the same results in above. 
    Lastly, when $m= n+1$, we apply the specular boundary condition as \eqref{f-f:m=n+1:x:ns} in Lemma \ref{tx-bx}, and omit details. Moreover, if we use 
    \begin{align*}
    \begin{split}
        &\quad |V(s;t,x,\tv)-V(s;t,x,v)| = |V_m(x,\tv)-V_n(x,v)|_p \\
        &\leq |v-\tv|+2|v|_p \left|(m-n)\left(\frac{\pi}{2}-a(x,v) \right) \right| + 2|v|_p|m(a(x,v)-a(x,\tv))|
    \end{split}
    \end{align*} instead of \eqref{V-V:tvv:case3} and  
     \begin{align*}
    \begin{split}
        &\quad |X(s;t,x,\tv)-X(s;t,x,v)| \\
        &\leq 2|V_m(x,\tv)-V_n(x,v)|_p|v|_p^{-1}+|v|_p|t_b(x,\tv)-t_b(x,v)|\\ &\quad +|v|_p|(m-1)(t_*(x,\tv)-t_*(x,v))|+(m-n)|v|_pt_*(x,v)+|a(x,v)-a(x,\tv)|
    \end{split}
    \end{align*} instead of \eqref{X:tvv:3}, then  we can replace $1/\cos a(x,v)$ by $1/\cos a(x,\tv) $ in the above results.
    Therefore, we conclude
     \begin{align} \label{sub:case3:Vv:sing}
    \begin{split}
        &|V(s;t,x,v )- V(s;t, x, \tv )| \\
	 &\lesssim 
		|v-\tv| (1+|v|_p (t-s)) \times \min \left\{\frac{1}{\cos a(x,v)},\frac{1}{\cos a(x,\tv)} \right\}.
	 \end{split}
     \end{align}

\begin{figure}[t]
\centering
\includegraphics[width=6cm]{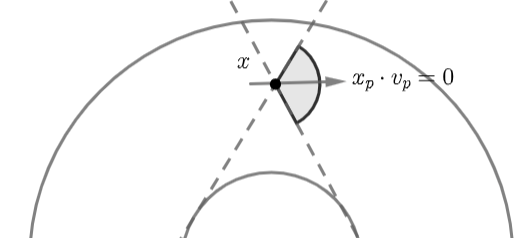}
\caption{ When $(x,v) \in \mathcal{C}_3$, the possible direction of velocity, $v$.}\label{sing_v}
\end{figure} 

\noindent Now, we assume that $x_p \cdot v$ and $x_p \cdot \tv$ have different signs. Then, there exists $\tau^{*}\in [0,1]$ such that $x_p \cdot \V(\tau^{*})=0$. (See Figure \ref{sing_v}). 
From \eqref{detail a,b:1}, we have $\sin a(x,\V(\tau))=\frac{|x|_p}{R}\cos \theta_{\V(\tau)}$ , where $\V(\tau)=(|v|_p\cos \theta_{\V(\tau)}, |v|_p\sin \theta_{\V(\tau)},v_3)$. Thus,
\begin{align*}
    \max \left\{\frac{1}{\cos a(x,v)},\frac{1}{\cos a(x,\tv)} \right\} \leq \frac{1}{\cos a(x,\V(\tau^{*}))}.
\end{align*} Using \eqref{sub:case3:Vv:sing}, 
\begin{align}\label{sub:case3:Vv:sing:2}
\begin{split}
      &|V(s;t,x,v )- V(s;t, x, \tv )| \\
      &\leq |V(s;t,x,v )- V(s;t, x, \V(\tau^{*} )|+|V(s;t,x, \tau^{*})- V(s;t, x, \tv )|\\
      &\lesssim |v-\tv| (1+|v|_p (t-s)) \\ &\quad\times \left[\min\left\{\frac{1}{\cos a(x,v)},\frac{1}{\cos a(x,\V(\tau_*))}\right\}+\min \left\{\frac{1}{\cos a(x,\tv)},\frac{1}{\cos a(x,\V(\tau_*))} \right\}\right]\\
      &\lesssim  |v-\tv|(1+|v|_p (t-s))\times \max \left\{\frac{1}{\cos a(x,v)},\frac{1}{\cos a(x,\tv)} \right\}.
\end{split}
\end{align} From Lemma \ref{disk}, 
    \begin{align*}
        |\nabla_{v_1,v_2} X(s;t,x_p,v_p)| \lesssim \frac{1}{|v|_p}(1+|v|_p(t-s)).
    \end{align*} 
Since $|\nabla_{v_3} X(s;t,x,v)|=t-s$,
    \begin{align} \label{sub:case3:Xv:sing}
    \begin{split}
        &|X(s;t,x,v )- X(s;t, x, \tv )| \\
	&\leq \int^{1}_{0} 
		|\nabla_v X(s;t,\V(\tau),v ) \dot{\V}
		    |d\tau  
	 \lesssim 
		|v-\tv|\frac{1}{|v|_p}(1+|v|_p(t-s)).
	 \end{split}
\end{align} 
From \eqref{sub:case3:Vv:sing}, \eqref{sub:case3:Vv:sing:2}, and \eqref{sub:case3:Xv:sing}, we construct \eqref{f-f:tv:3}.

\end{proof}

\section{\texorpdfstring{$\mathfrak{H}^{2\b}_{sp, vel}$}{} estimates} 

In section 5, we let $f : \overline{\O} \times \R^{3} \rightarrow \R_{+} \cup \{0\}$ be a function which satisfies specular reflection \eqref{specular}, where $\O$ is a domain as in Definition \ref{def:domain} and $w(v) = e^{\vartheta|v|^{2}}$ for some $0 < \vartheta$. Let fix any point $t,s \in \mathbb{R}^{+}$, $v, \bv, \zeta \in \mathbb{R}^3$, $x,\bx \in \O$. We assume \eqref{assume_x}, \eqref{assume_v}, and recall $\tx(x,\bx,v+\zeta)$ in \eqref{def:tx}, $\X(\tau)$ in \eqref{def:x_para}, $\tv(v,\bv,\zeta)$ in \eqref{def:tv} and $\V(\tau)$ in \eqref{def:v_para}. 

\subsection{Difference quotients estimates of all cases}
We split
\begin{align}
& |f(s, X(s;t,x,v+\zeta), V(s;t,x,v+\zeta)) - f(s, X(s;t, \bar{x}, \bar{v}+\zeta), V(s;t, \bar{x}, \bar{v}+\zeta)) | \notag \\
&\leq | f(s, X(s;t,x,v+\zeta), V(s;t,x,v+\zeta)) - f(s, X(s;t, \tilde{x}, v+\zeta ), V(s;t, \tilde{x}, v+\zeta )) |  \label{split1}   \\
&\quad + |f(s, X(s;t, \tilde{x}, v+\zeta ), V(s;t, \tilde{x}, v+\zeta ))  - f(s, X(s;t, \bar{x}, v+\zeta ), V(s;t, \bar{x}, v+\zeta ))|  \label{split2}   \\
&\quad + |f(s, X(s;t, \bar{x}, v+\zeta ), V(s;t, \bar{x}, v+\zeta )) - f(s, X(s;t, \bar{x}, \tv+\zeta ), V(s;t, \bar{x}, \tv+\zeta ))  | \label{split3}    \\
&\quad + |f(s, X(s;t, \bar{x}, \tv+\zeta ), V(s;t, \bar{x}, \tv+\zeta ))   - f(s, X(s;t, \bar{x}, \bv+\zeta ), V(s;t, \bar{x}, \bv+\zeta ))|,   \label{split4}        
\end{align}
\hide
\begin{align}
& f(s, X(s;t,x,v), V(s;t,x,v)) - f(s, X(s;t, \bar{x}, \bar{v}), V(s;t, \bar{x}, \bar{v}))  \\
&\leq  f(s, X(s;t,x,v), V(s;t,x,v)) - f(s, X(s;t, \tilde{x}, v ), V(s;t, \tilde{x}, v ))      \\
&\quad + f(s, X(s;t, \tilde{x}, v ), V(s;t, \tilde{x}, v ))  - f(s, X(s;t, \bar{x}, v ), V(s;t, \bar{x}, v ))       \\
&\quad + f(s, X(s;t, \bar{x}, v ), V(s;t, \bar{x}, v )) - f(s, X(s;t, \bar{x}, \tilde{v} ), V(s;t, \bar{x}, \tilde{v} ))       \\
&\quad + f(s, X(s;t, \bar{x}, \tilde{v} ), V(s;t, \bar{x}, \tilde{v} ))   - f(s, X(s;t, \bar{x}, \bar{v} ), V(s;t, \bar{x}, \bar{v} ))       
\end{align}
\unhide
and estimate each \eqref{split1}--\eqref{split4}. If $v_p = 0$, then we have
\begin{align*}
    |(X,V)(s;t,x,v)-(X,V)(s;t,\bx,\bv)| \leq |x-\bx|+(t-s)|v-\bv|.
\end{align*}  We only consider the case that $v_p \neq 0$ in the proof of Lemma \ref{lem:f-f:txbx}, Lemma \ref{lem:f-f:txx}, Lemma \ref{lem:f-f:tvbv}, and Lemma \ref{lem:f-f:tvv}.

\begin{figure} [t]
\centering
\includegraphics[width=6cm]{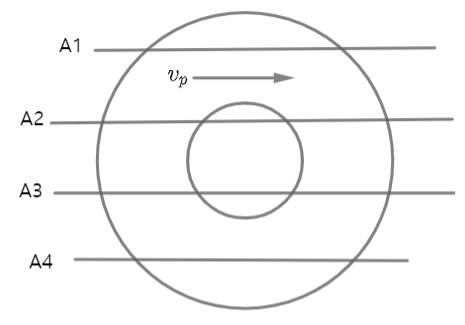}
\caption{ $A_1, A_2, A_3$ and $A_4$ for given $v$. }\label{fig:bx}
\end{figure}

\begin{lemma} \label{lem:f-f:txbx}
    For $|x-\bx|\leq 1$, we have
    \begin{align}  \label{f-f:txbx:sec4}
    \begin{split}
       \eqref{split2}
      &\lesssim \Big[\Big(1+|v+\zeta|_p(t-s)\Big)\mathcal{S}_{sp}(\bx, v+\zeta)\Big]^{2\b} \Big(\mathcal{X}(2\b,s,v,\zeta)+|v+\zeta|^{2\b}_p\mathcal{V}(2\b,s,v,\zeta)\Big)|x-\bx|^{2\b},
     \end{split}
     \end{align} where $\mathcal{S}_{sp}(\bx, v+\zeta)$ is defined as 
    \begin{align} \label{def:S_sp}
    \begin{split}
         \mathcal{S}_{sp}(\bx, v+\zeta)&=\frac{1}{\cos b(\bx,v+\zeta)} \mathbf{1}_{\{(\bx,v+\zeta)\in\text{$\mathcal{C}_1$or $\mathcal{C}_2$}\}}.  
    \end{split}
    \end{align}
\end{lemma} 
\begin{proof}

\begin{figure} [t]
\centering
\includegraphics[width=9cm]{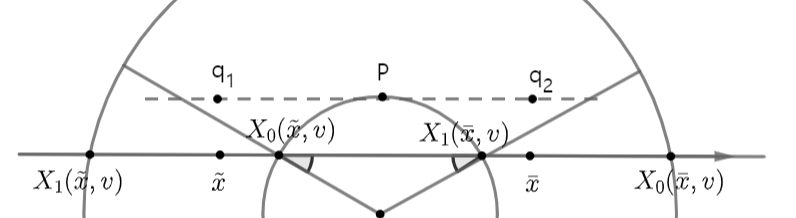}
\caption{$(\bx,v) \in \mathcal{C}_1$ and $(\tx,v) \in \mathcal{C}_2$.}\label{case1 and 2}
\end{figure}

    Let us $v+\zeta=(v_1,v_2,v_3)$. We assume $\zeta=0$ and $v_2=0$ without loss of generality, and it is expressed by $v+\zeta=v=(v_1,0,v_3)$. For $y:=(y_1,y_2,y_3)$, we classify $y \in A_1$ if $R \geq y_2 > r$, $\bx \in A_2$ if $r \geq y_2 \geq 0$, $\bx \in A_3$ if $0>y_2 \geq -r$, and $y \in A_4$ if $-r > y_2 \geq -R$. (See a Figure \ref{fig:bx}.) $A_1$ and $A_4$, and $A_3$ and $A_2$ are each $x$-axial symmetry.\\
    
     First, we assume that $\tx, \bx \in A_1$. Then we have $(\tx,v), (\bx,v) \in \text{$\mathcal{C}_3$}$ with moving clockwise, and we get \eqref{f-f:bx:3}. Next, assume that $\tx,\bx \in A_2$. If $(\tx,v), (\bx,v) \in \mathcal{C}_1$ or $(\tx,v), (\bx,v) \in \mathcal{C}_2$, then we have \eqref{f-f:bx}. We assume that $(\tx,v)\in \text{$\mathcal{C}_1$}, (\bx,v) \in \text{$\mathcal{C}_2$}$  with moving clockwise or $(\tx,v)\in \text{$\mathcal{C}_2$}, (\bx,v) \in \text{$\mathcal{C}_1$}$  with moving clockwise. (See Figure \ref{case1 and 2}). Let us $p=(0,r,0)+\tx_{ver}$, $q_1=\tx(\tx(x,\bx,v),p,v)$, and $q_2=\tx(\bx,p,v)$. We split 
    \begin{align}
        \eqref{split2} &\leq |f(s, X(s;t, \tilde{x}, v ), V(s;t, \tilde{x}, v ))  - f(s, X(s;t, q_1, v ), V(s;t, q_1, v ))|\label{split q1}\\
        &\quad+|f(s, X(s;t, q_1, v ), V(s;t, q_1, v ))-f(s, X(s;t, p, v ), V(s;t, p, v ))|\label{split p}\\
        &\quad+|f(s, X(s;t, p, v ), V(s;t, p, v ))-f(s, X(s;t, q_2, v ), V(s;t, q_2, v ))| \label{split q2}\\
        &\quad+|f(s, X(s;t, q_2, v ), V(s;t, q_2, v ))-f(s, X(s;t, \bx, v ), V(s;t, \bx, v ))|\label{split bx}
    \end{align}
We apply Lemma \ref{tx-bx} to \eqref{split p} and \eqref{split q2}. 
We apply Lemma \ref{x-tx} and Lemma \ref{lemma:g_sp} to \eqref{split q1}, \eqref{split bx}. Because $\cos b(\bx,v)=\cos b(\tx,v)$, we obtain \eqref{f-f:txbx:sec4}.
\end{proof}

\begin{figure}[t]
\centering
\includegraphics[width=4cm]{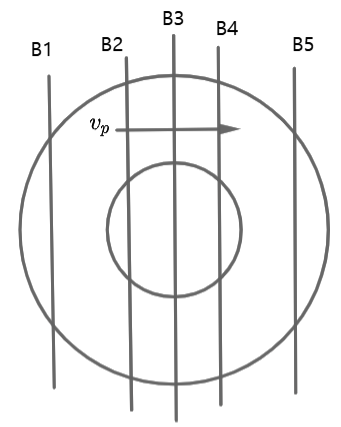}
\caption{$B_1, B_2, B_3, B_4$ and $B_5$ for given $v$. }\label{fig:tx}
\end{figure} 

\begin{lemma}\label{lem:f-f:txx}
    For $|x-\bx|\leq 1$, we have
    \begin{align}\label{f-f:txx:sec4}
    \begin{split}
         \eqref{split1}
         &\lesssim \bigg(\Big[\Big(1+|v+\zeta|_p(t-s)\Big)\mathcal{T}_{sp,1}(x, \tx, v+\zeta)\Big]^{2\b}   \mathcal{X}(2\b,s,v,\zeta) \\
	&\quad+ \Big[\Big(|v+\zeta|_p+|v+\zeta|_p^2(t-s)\Big) \mathcal{T}_{sp,2}(x, \tx, v+\zeta)\Big]^{2\b}    \mathcal{V}(2\b,s,v,\zeta) \bigg) |x-\bx|^{2\b},\\   
    \end{split}
    \end{align} where $\mathcal{T}_{sp,1}(x, \tx, v+\zeta)$ and $\mathcal{T}_{sp,2}(x, \tx, v+\zeta)$ are defined as 
    \begin{align} \label{def:T_sp_1}
    \begin{split}
         \mathcal{T}_{sp,1}(x, \tx, v+\zeta)&=\frac{1}{\cos b(x,v+\zeta)} \mathbf{1}_{\{(x,v+\zeta)\in\text{$\mathcal{C}_1$or $\mathcal{C}_2$}\}}+
         \frac{1}{\cos b(\tx,v+\zeta)} \mathbf{1}_{\{(\tx,v+\zeta)\in\text{$\mathcal{C}_1$or $\mathcal{C}_2$}\}}\\
         &\quad + \frac{1}{\cos a(x,v+\zeta)} \mathbf{1}_{\{{(x,v+\zeta)}\in\text{$\mathcal{C}_3$}\}}+\frac{1}{\cos a(\tx,v+\zeta)} \mathbf{1}_{\{{(\tx,v+\zeta)}\in\text{$\mathcal{C}_3$}\}}
    \end{split}
    \end{align}
    and
    \begin{align} \label{def:T_sp_2}
    \begin{split}
          \mathcal{T}_{sp,2}(x, \tx, v+\zeta)&=\frac{1}{\cos b(x,v+\zeta)} \mathbf{1}_{\{(x,v+\zeta)\in\text{$\mathcal{C}_1$or $\mathcal{C}_2$}\}}+
         \frac{1}{\cos b(\tx,v+\zeta)} \mathbf{1}_{\{(\tx,v+\zeta)\in\text{$\mathcal{C}_1$or $\mathcal{C}_2$}\}}\\
         &\quad + \frac{1}{\cos^2 a(x,v+\zeta)} \mathbf{1}_{\{{(x,v+\zeta)}\in\text{$\mathcal{C}_3$}\}}+\frac{1}{\cos^2 a(\tx,v+\zeta)} \mathbf{1}_{\{{(\tx,v+\zeta)}\in\text{$\mathcal{C}_3$}\}}.
    \end{split}
    \end{align}
\noindent Here, $b(x,v)$ is in \eqref{def:angle:1} or \eqref{def:angle:2}, and $a(x,v)$ is in \eqref{def:angle:3}.
\end{lemma}
    \begin{proof}
    We first let $v+\zeta=(|v|_p\cos \theta_v,|v|_p\sin \theta_v, v_3)$. Next, we let $\zeta=0$ and $\theta_v=0$ without loss of generality. Then, it is also expressed $v+\zeta=v=(v_1,0,v_3)$. We also let $z=(z_1,z_2,z_3)=(|z|_p\cos\theta_z, |z|_p\sin\theta_z, z_3).$ From \eqref{detail a,b:1}, we get   
    \begin{align} \label{size:sinb:case1}
         \sin b(z,v) = \frac{|z|_p}{r}\sin(\theta_z-\theta_v) = \frac{|z|_p}{r}\sin\theta_z=\frac{|z_2|}{r}, \quad b(z,v) \in [0, \pi/2],
     \end{align} where $(z,v) \in \mathcal{C}_1$ or $\mathcal{C}_2$ with moving clockwise, and 
     \begin{align} \label{size:sina:case3}
         \sin a(z,v) = \frac{|z|_p}{R}\sin(\theta_z-\theta_v) = \frac{|z|_p}{R}\sin\theta_z=\frac{|z_2|}{R}, \quad a(z,v) \in [0, \pi/2),
     \end{align} where $(z,v) \in \mathcal{C}_3$ with moving clockwise. \\

     For $y=(y_1,y_2,y_3) \in \O$, we classify $y \in B_1$ if $-r > y_1 \geq -R$, $y \in B_2$ if $0 > y_1 \geq -r$, $y \in B_3$ if $y_1 =0$, $y \in B_4$ if $r \geq y_1 > 0$, and $y \in B_5$ if $R \geq y_1 > r$. (See Figure \ref{fig:tx}.) $B_1$ and $B_5$, and $B_2$ and $B_4$ are each $y$-axial symmetry. Let us $x=(x_1,x_2,x_3)$ and $\tx=(\tx_1,\tx_2,\tx_3)$. We split into three cases: \\
     
     \textit{Case 1} First, assume that $x,\tx \in B_1$. If $R \geq x_2,\tx_2 \geq r$, then $(x,v),(\tx,v) \in \text{$\mathcal{C}_3$}$ with moving clockwise and we get \eqref{f-f:txx:sec4} by \eqref{f-f:tx:3}. If $R \geq x_2 \geq r$ and $r \geq \tx_2 \geq 0$,
     then there is $\tau_{+}(x,\bx,v)\in [0,1]$ in \eqref{tau+_2}. The  $(\X(\tau),v) \in \text{$\mathcal{C}_2$}$ with moving clockwise for $0 \leq \tau \leq \tau_{+}$, and $(\X(\tau),v) \in \text{$\mathcal{C}_3$}$ with moving clockwise for $\tau_{+} \leq \tau \leq 1$. We split
     \begin{align}
         &\quad | f(s, X(s;t,x,v), V(s;t,x,v)) - f(s, X(s;t, \tilde{x}, v ), V(s;t, \tilde{x}, v))| \notag \\
         &\leq | f(s, X(s;t,\tx,v), V(s;t,\tx,v)) - f(s, X(s;t, \X(\tau_{+}), v ), V(s;t, \X(\tau_{+}), v))| \label{split:B1:-1}\\
         &\quad+ |f(s, X(s;t, \X(\tau_{+}), v), V(s;t, \X(\tau_{+}), v))- f(s, X(s;t, x, v), V(s;t, x, v))|. \label{split:B1:0}
     \end{align}
      From \eqref{g_sp:1},
     \begin{align} \label{B1:case2}
         \left|\int^{\tau_{+}}_{0} \frac{1}{\mathfrak{S}_{sp}(\tau; x, \tx, v)} d\tau \right|  &\lesssim \left(\frac{1}{|v|_p}+\frac{1}{|v|_p\cos b(\tx,v)}\right)|\tau_+|.
     \end{align} Since $\sin a(\X(\tau_{+}),v)=r/R $,
     \begin{align} \label{B1:case3}
           \max \left\{\frac{1}{\cos^2 a(\X(\tau_{+}),v)}, \frac{1}{\cos^2 a(x,v)} \right\} \lesssim  \frac{1}{\cos^2 a(x,v)},
     \end{align} where $a(x,v)$ is in \eqref{def:angle:3}. We apply \eqref{f-f:tx}, \eqref{B1:case2} to \eqref{split:B1:-1}, and  \eqref{f-f:tx:3},\eqref{B1:case3} to \eqref{split:B1:0}. \\ If $r \geq x_2,\tx_2 \geq 0$, then $(\X(\tau),v) \in \text{$\mathcal{C}_2$}$ with moving clockwise for $0 \leq \tau \leq 1$. 
     From \eqref{size:sinb:case1} and \eqref{sin<r/R},
       \begin{align} \label{B1:same case}
       \begin{split}
             \frac{\mathbf{1}_{0 \leq \tau \leq 1}}{\mathfrak{S}_{sp}(\tau; x, \tx, v)} &\leq  
          \frac{1}{|v|_p} \frac{\mathbf{1}_{0 \leq \tau \leq 1}}{\cos b(\X(\tau),v)}+\frac{1}{|v|_p} \frac{\mathbf{1}_{0 \leq \tau \leq 1}}{\cos a(\X(\tau),v)} \\ &\lesssim \frac{1}{|v|_p}\left(\max \left\{\frac{1}{\cos b(x,v)}, \frac{1}{\cos b(\tx,v)} \right\}+1\right),
       \end{split}
     \end{align} where $a(x,v), b(x,v)$ is in \eqref{def:angle:2}. Using \eqref{B1:same case} and \eqref{f-f:tx}, we have \eqref{f-f:txx:sec4}. 
     If $r \geq x_2 \geq 0$, $0 > \tx_2 \geq -r$, then there is $\tau_0(x,\bx,v) \in [0,1]$ in \eqref{tau+_2}. We apply similar argument as \eqref{B1:same case} again and $\cos b(\X(\tau_0),v)=1$.
     If $R \geq x_2 \geq r$, $0 > \tx_2 \geq -r$, then there is $\tau_0(x,\bx,v), \tau_{+}(x,\bx,v) \in [0,1]$ in \eqref{tau+_2}. We use $\sin a(\X(\tau_{+}),v)=r/R$ and $\cos b(\X(\tau_0),v)=1$. If $R \geq x_2 \geq r$, $-r \geq \tx_2 \geq -R$, then there is $\tau_0(x,\bx,v) \in [0,1]$ in \eqref{tau+_2}, and repeat the previous work again. \\ 
     
    \textit{Case 2} Second, we assume $x,\tx \in B_2$. Next, we assume $r \geq x_2 \geq 0$, $0 \geq \tx_2 \geq -r$ or $R \geq x_2 \geq r$, $0 \geq\tx_2 \geq -r$ or $R \geq x_2 \geq r$, $-r \geq \tx_2 \geq -R$.(Otherwise, we can use the same arguments to $x, \tx \in B_1$.) We define
     \begin{align*}
         &p(x) = x \mathbf{1}_{0 \leq x_2 \leq r} +  \X(\tau_{+}) \mathbf{1}_{r < x_2 \leq R}=(p_1(x),p_2(x),p_3(x)),  \notag \\
              &q(\tx) = \tx \mathbf{1}_{-r \leq \tx_2 \leq 0} +  \X(\tau_{-}) \mathbf{1}_{-R < \tx_2 \leq -r}=(q_1(\tx),q_2(\tx),q_3(\tx)). \notag 
     \end{align*} Here, $\tau_{+}(x,\bx,v)$ is in \eqref{tau+_2}, and $\tau_{-}(x,\bx,v)$  satisfies  $\hat{v}_p\cdot n(X_0(\X(\tau_{-}),v)) = 0$. There is $\alpha_{*} \in \mathbb{R}^{+}$ such that $-r>p_{1}(x,v)-\alpha_{*}v_1>-R$, $-r>q_{1}(\tx,v)-\alpha_{*}v_1>-R$, and $|\alpha_{*}v_1|\leq |x-\tx|$. Let us 
     \begin{align*}
      p_*(x,v) = (p_1(x)-\alpha_* v_1,p_2(x),p_3(x)) \text{ \;and \;}
        q_*(\tx,v) = (q_1(\tx,v)-\alpha_* v_1,q_2(x),q_3(x)). 
     \end{align*} Then we have
     \begin{align*} 
         b(p(x),v)=b(p_*(x,v),v)  \text{ \;and \;}
         b(q(\tx),v)=b(q_*(\tx,v),v),
     \end{align*} where $b(x,v)$ in \eqref{def:angle:2}. We split
     \begin{align*}
        &\quad | f(s, X(s;t,x,v), V(s;t,x,v)) - f(s, X(s;t, \tilde{x}, v ), V(s;t, \tilde{x}, v))| \\
        &\leq | f(s, X(s;t,x,v), V(s;t,x,v)) - f(s, X(s;t,p(x), v ), V(s;t,p(x), v))|  \\
         &\quad+ | f(s, X(s;t,p(x),v), V(s;t,p(x),v)) - f(s, X(s;t,p_*(x,v), v ), V(s;t,p_*(x,v), v))| \\
         &\quad+|f(s, X(s;t,p_*(x,v), v ), V(s;t,p_*(x,v), v))-f(s, X(s;t,q_*(\tx,v), v ), V(s;t,q_*(\tx,v), v))| \\
         &\quad+ |f(s, X(s;t,q_*(\tx,v), v ), V(s;t,q_*(\tx), v))- f(s, X(s;t, q(\tx), v ), V(s;t, q(\tx), v))|  \\
         &\quad+ |f(s, X(s;t,q(\tx,v), v ), V(s;t,q(\tx,v), v))- f(s, X(s;t, \tx, v ), V(s;t, \tx, v))|.
     \end{align*} Then $p_*(x,v),q_*(\tx,v) \in B_1$, and we can also estimate other terms, and omit details. \\ 
     
     \textit{Case 3} Lastly, we assume $x,\tx \in B_3$. We use the same arguments 
      $x,\tx \in B_1$ or  $x,\tx \in B_2$. Therefore, we conclude \eqref{f-f:txx:sec4} for any $x,\bx \in \O$ when $v \in \mathbb{R}^3$ is fixed.
     
    \end{proof}

    \begin{lemma} \label{lem:f-f:tvbv}
   We have
     \begin{align}  \label{f-f:tvbv:sec4}
      \begin{split}
        \eqref{split4}
        &\lesssim  
           \bigg(\Big[(t-s)+\max \left\{\frac{1}{|v+\zeta|_p},\frac{1}{|\bv+\zeta|_p} \right \}\Big]^{2\b}\mathcal{X}(2\b,s,v,\zeta) \\
        &\quad+\Big[1+(t-s)\max\left\{|v+\zeta|_p,|\bv+\zeta|_p \right\} \Big]^{2\b} \mathcal{V}(2\b,s,v,\zeta) \bigg)|v-\bv|^{2\b}
     \end{split}
     \end{align}
\end{lemma}
\begin{proof}
       Because $(\bx,\tv+\zeta), (\bx,\bv+\zeta) \in \text{$\mathcal{C}_1$}$ or $(\bx,\tv+\zeta), (\bx,\bv+\zeta) \in \text{$\mathcal{C}_2$}$ or $(\bx,\tv+\zeta), (\bx,\bv+\zeta) \in \text{$\mathcal{C}_3$}$, we have \eqref{f-f:tvbv:sec4} directly from Lemma \ref{tv-bv} and Lemma \ref{tv-bv:3}.
\end{proof}
    
\begin{lemma}\label{lem:f-f:tvv}
We have
    \begin{align}\label{f-f:tvv:sec4}
    \begin{split}
           \eqref{split3}
         &\lesssim \bigg(\Big[\frac{1}{|v+\zeta|_p}+(t-s)\Big]^{2\b} \mathcal{X}(2\b,s,v,\zeta)\\
	&\quad+  \Big[\Big(1+|v+\zeta|_p(t-s)\Big)\mathcal{T}_{vel}(\bx,v,\tv,\zeta)\Big]^{2\b}    \mathcal{V}(2\b,s,v,\zeta) \bigg)|v-\bv|^{2\b},
    \end{split}
    \end{align} where $\mathcal{T}_{vel}(\bx,v,\tv,\zeta)$ is defined as 
    \begin{align} \label{def:T_vel}
    \begin{split}
         \mathcal{T}_{vel}(\bx,v,\tv,\zeta)&=\frac{1}{\cos b(\bx,v+\zeta)} \mathbf{1}_{\{{(\bx,v+\zeta)}\in\text{$\mathcal{C}_1$or $\mathcal{C}_2$}\}}  +\frac{1}{\cos b(\bx,\tv+\zeta)} \mathbf{1}_{\{{(\bx,\tv+\zeta)}\in\text{$\mathcal{C}_1$or $\mathcal{C}_2$}\}} \\ 
         &\quad +\frac{1}{\cos a(\bx,v+\zeta)} \mathbf{1}_{\{{(\bx,v+\zeta)}\in\text{$\mathcal{C}_3$}\}}+\frac{1}{\cos a (\bx,\tv+\zeta)} \mathbf{1}_{\{{(\bx,\tv+\zeta)}\in\text{$\mathcal{C}_3$}\}}.
    \end{split}
    \end{align} Here, $b(x,v)$ is in \eqref{def:angle:1} or \eqref{def:angle:2}, and $a(x,v)$ is in \eqref{def:angle:3}.
\end{lemma}
\begin{proof}
We first let $\bx=(x_1,x_2,x_3)=(|x|_p\cos\theta_x, |x|_p\sin\theta_x,x_3)$. Next, we let
$x_1=0$ and $x_2 > 0$ without loss of generality. For $w=(|w|_p\cos\theta_w, |w|_p\sin\theta_w, w_3)$, from \eqref{detail a,b:1}, we have
   \begin{align} \label{sinb:tvv}
      \sin b(\bx,w) = \frac{|x|_p}{r}\sin (\theta_x-\theta_w) = \frac{|x|_p}{r}\cos\theta_w.
   \end{align} \\ 
   Let $\theta_*=\sin^{-1}(r/x_2) \in [0, \pi/2]$. For $w=(|w|_p\cos\theta_w,|w|_p\sin\theta_w,w_3)$, we classify $w \in D_1$ if $-\pi/2 \leq \theta_w \leq -\pi/2+\theta_*$, and $w\in D_2$ if $-\pi/2+\theta_* \leq \theta_w \leq \pi/2-\theta_*$, and $w \in D_3$ if $\pi/2-\theta_* \leq \theta_w \leq \pi/2$, and $w \in D_4$ if $\pi/2 \leq \theta_w \leq \pi/2+\theta_*$, and $w \in D_5$ if $\pi/2+\theta_* \leq \theta_w \leq \pi$ or $-\pi \leq \theta_w \leq -\pi/2-\theta_*$, and $w \in D_6$ if $-\pi/2-\theta_* \leq \theta_w \leq -\pi/2$.  (See Figure \ref{fig:tv}.) $D_1$ and $D_6$, and $D_2$ and $D_5$, $D_3$ and $D_4$ are each $y$-axial symmetry. Let $v+\zeta=(|v|_p\cos\theta_v,|v|_p\sin\theta_v,v_3)$ and $\tv+\zeta=(|v|_p\cos\theta_{\tv},|v|_p\sin\theta_{\tv},v_3)$ 
   for $\theta_v, \theta_{\tv} \in [-\pi,\pi]$. Now, we split into four cases:  \\
   \begin{figure}[t]
    \centering
    \includegraphics[width=5cm]{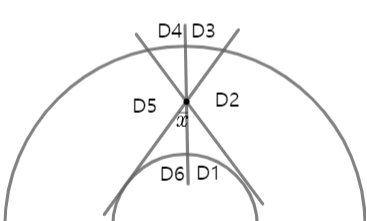}
    \caption{$D_1, D_2, D_3, D_4, D_5$ and $D_6$ for given $\bx$.}\label{fig:tv}
    \end{figure} 
    
   \textit{Case 1} First, assume that both $v+\zeta,\; \tv+\zeta \in D_1$. Because
$\max_{0\leq \tau \leq 1} \cos \T(\tau) = \max\{\cos \theta_v, \cos \theta_{\tv}\}$ for $\V(\tau), \T(\tau)$ in \eqref{def:v_para} and \eqref{sinb:tvv}, we get
\begin{align*}
       \max_{0\leq \tau \leq 1} \sin b(\bx,\V(\tau)) = \max\{\sin b(\bx,v+\zeta), \sin b(\bx, \tv+\zeta)\}, 
   \end{align*} where  $b(x,v) \in [0, \frac{\pi}{2}]$. By above result and \eqref{sin<r/R},
   \begin{align} \label{g_est:C1}
   \begin{split}
       \frac{\mathbf{1}_{0 \leq \tau \leq 1}}{\mathfrak{S}_{vel}(\tau; \bx, v, \tv, \zeta)} &\leq \frac{ t_b(\bx,\V(\tau))\mathbf{1}_{0 \leq \tau \leq 1}}{|v+\zeta|_p\cos a(\bx,\V(\tau))}+\frac{ t_f(\bx,\V(\tau))\mathbf{1}_{0 \leq \tau \leq 1}}{|v+\zeta|_p\cos b(\bx,\V(\tau))}\\ &\lesssim \frac{1}{|v+\zeta|^{2}_p}\left(1+\max \left \{\frac{1}{\cos b(\bx,v+\zeta)}, \frac{1}{\cos b(\bx,\tv+\zeta)} \right\}\right).
   \end{split}
   \end{align} Here, $a(x,v),b(x,v)$ are in \eqref{def:angle:2}. We apply \eqref{g_est:C1} to \eqref{f-f:tv}, and get \eqref{f-f:tvv:sec4}. We can use the same arguments when $v+\zeta,\;\tv+\zeta$ are located in $D_3$. Next, assume that both  $v+\zeta,\;\tv+\zeta \in D_2$. Then we get \eqref{f-f:tvv:sec4} from \eqref{f-f:tv:3} directly.\\ 
   
   \textit{Case 2} Second, we assume that $\tv+\zeta \in D_1$, and $v+\zeta \in D_2$. Then there is $\tau_{-}(\bx,v,\bv,\zeta)\in[0,1]$ in \eqref{tau+_2:v}. We can split
   \begin{align}
          &\quad | f(s, X(s;t,\bx,\tv+\zeta), V(s;t,\bx,\tv+\zeta)) - f(s, X(s;t, \bx, v+\zeta ), V(s;t, \bx, v+\zeta))| \notag \\
         &\leq | f(s, X(s;t,\bx,\tv+\zeta), V(s;t,\bx,\tv+\zeta)) - f(s,X(s;t,\bx,\V(\tau_{-})),V(s;t,\bx,\V(\tau_{-})))|\label{split:C12:1} \\
         &\quad+|f(s,X(s;t,\bx,\V(\tau_{-})),V(s;t,\bx,\V(\tau_{-})))-f(s, X(s;t, \bx, v+\zeta ), V(s;t, \bx, v+\zeta))|. \label{split:C12:2} 
   \end{align} We apply \eqref{f-f:tv},\eqref{g_vel:2} to \eqref{split:C12:1}. And we apply \eqref{f-f:tv:3} and  $\cos a(\bx,\V(\tau_{-})) = \sqrt{1-(r/R)^2}$ to \eqref{split:C12:2}. Then we get \eqref{f-f:tvv:sec4}. If $\tv+\zeta \in D_2, v+\zeta \in D_3$, we consider $\mathcal{C}_1$ instead of $\mathcal{C}_2$ in above. Next, if $\tv+\zeta \in D_6, v+\zeta \in D_1$ or $\tv+\zeta \in D_3, v+\zeta \in D_4$, there is $\tau_{0-}(\bx,v,\bv,\zeta) \in [0,1]$ in \eqref{tau+_2:v}. We can use similar arguments above. \\
  
\textit{Case 3} Assume $\tv+\zeta \in D_1$ and $v+\zeta \in D_3$. There are $\tau_{-}(\bx, v, \bv, \zeta)\in [0,1]$ in \eqref{tau+_2:v} and $\tau_{+}(\bx, v, \bv, \zeta)\in[0,1]$ in \eqref{tau+_1:v}, and it holds that $0 \leq \tau_{-}(\bx, v, \bv, \zeta)\leq \tau_{+}(\bx, v, \bv, \zeta) \leq 1$. We apply \eqref{v-tv}, \eqref{g_vel:1} and \eqref{g_vel:2} for $\tau \in [0,\tau_{-}]$ and $\tau \in [\tau_{+},1]$. And we apply 
$\sin a(\bx,\V(\tau_{-}))=\sin a(\bx,\V(\tau_{+}))= r/R$ and \eqref{v-tv:3} for $\tau \in [\tau_{-},\tau_{+}]$. Next, assume $\tv+\zeta \in D_2$ and $v+\zeta \in D_4$. There are $\tau_{+}(\bx, v, \bv, \zeta)$ and $ \tau_{0+}(\bx, v, \bv, \zeta) \in [0,1]$ in \eqref{tau+_1:v} for $0 \leq \tau_{+}(\bx, v, \bv, \zeta)\leq \tau_{0+}(\bx, v, \bv, \zeta) \leq 1$. We apply $\sin a(\bx,\V(\tau_{+}))=r/R$ and \eqref{tv-bv:3} for $\tau \in [0,\tau_{+}]$, and  $\cos b(\bx,\V(\tau_{0+}))=1$ and \eqref{g_vel:1} for $\tau \in [\tau_{+},\tau_{0+}]$. After reflecting $X(s;t,\bx,\V(\tau))$ about y-axis for $\tau_{0+} \leq \tau \leq 1$, we apply \eqref{sin<r/R} and \eqref{sinb:tvv} like \eqref{g_est:C1}. If $\tv+\zeta \in {D_3}$ and $v+\zeta \in {D_5}$, then we reflect $v+\zeta,\;\tv+\zeta$ about y-axis. If $\tv+\zeta \in D_2, v+\zeta \in D_6$, we consider $\mathcal{C}_1$ instead of $\mathcal{C}_2$ in above. If $\tv+\zeta \in {D_5}, v+\zeta \in {D_1}$ or $\tv+\zeta \in {D_3}, v+\zeta \in {D_5}$, then we reflect $v+\zeta,\;\tv+\zeta$ about y-axis. \\

\textit{Case 4} Lastly, if $\tv+\zeta \in D_1$ and $v+\zeta \in D_4$, there are $\tau_{-}(\bx, v, \bv, \zeta)\in [0,1]$ in \eqref{tau+_2:v} and $\tau_{+}(\bx, v, \bv, \zeta), \tau_{0+}(\bx, v, \bv, \zeta) \in [0,1]$ in \eqref{tau+_1:v} for $0 \leq \tau_{-}(\bx, v, \bv, \zeta) \leq \tau_{+}(\bx, v, \bv, \zeta)\leq \tau_{0+}(\bx, v, \bv, \zeta) \leq 1$. If 
 $\tv+\zeta \in D_2$ and $v+\zeta \in D_5$, there are $\tau_{+}(\bx, v, \bv, \zeta)$ and $\tau_{0+}(\bx, v, \bv, \zeta) \in [0,1]$ in \eqref{tau+_1:v} for $0 \leq \tau_{+}(\bx, v, \bv, \zeta)\leq \tau_{0+}(\bx, v, \bv, \zeta) \leq 1$. For each interval, we use similar arguments in  \textit{Case 1},  \textit{Case 2} and  \textit{Case 3}. Therefore, we conclude \eqref{f-f:tvv:sec4} for any $v+\zeta, \bv+\zeta \in \mathbb{R}^3$ when $\bx \in \O$ is fixed.
\end{proof}
\subsection{Integrability}
\begin{lemma}\label{int:case1,2}
For $x \in \O$, $v \in \mathbb{R}^3, \; c > 0$ and $0<\beta<1/2$, we have
    \begin{align} \label{int:case1,2:1}
        \int \frac{e^{-c|\zeta|^2}}{|\zeta|}\frac{\langle v+\zeta \rangle^{r}}{|\cos b(x,v+\zeta)|^{2\beta}}\mathbf{1}_{\{{(x,v+\zeta)}\in\text{$\mathcal{C}_{1,2}$}\}} \; d\zeta \lesssim_{\beta} \langle v \rangle^{r+1}
    \end{align} and
    \begin{align} \label{int:case1,2:2}
        \int \frac{e^{-c|\zeta|^2}}{|\zeta|}\frac{\langle v+\zeta \rangle^{r}}{|\cos b(x,v+\zeta)|^{2\beta}}\frac{1}{|v+\zeta|_p^{2\beta}}\mathbf{1}_{\{{(x,v+\zeta)}\in\text{$\mathcal{C}_{1,2}$}\}} \; d\zeta \lesssim_{\beta} \langle v \rangle^{r+1-2\beta},
    \end{align} where $b(x,v)$ is in \eqref{def:angle:1} or \eqref{def:angle:2}.
\end{lemma}
\begin{proof}
We prove \eqref{int:case1,2:2} at first. We have
\begin{align}
       &\int \frac{e^{-c|\zeta|^2}}{|\zeta|}\frac{\langle v+\zeta \rangle^{r}}{|\cos b(x,v+\zeta)|^{2\beta}}\frac{1}{|v+\zeta|_p^{2\beta}}\mathbf{1}_{\{{(x,v+\zeta)}\in\text{$\mathcal{C}_{1}$}\}} \; d\zeta  \notag \\
       &=
        \int \frac{e^{-c|v-\zeta|^2}}{|v-\zeta|}\frac{\langle \zeta \rangle^{r}}{|n(X_1(x,\zeta))\cdot \widehat{\zeta}|^{2\beta}}\frac{1}{|\zeta|^{2\b}}\mathbf{1}_{\{{(x,\zeta)}\in\text{$\mathcal{C}_{1}$}\}} \; d\zeta. \label{int_C1}
\end{align}
We change the coordinate system, considering the $x$-axis as the $z$-axis in Definition \ref{def:domain}. We consider spherical coordinate of $\zeta=|\zeta|(\cos \varphi_{\zeta} \sin\theta_{\zeta},\sin \varphi_{\zeta} \sin\theta_{\zeta},\cos \theta_{\zeta})\in\R^{3}$. Let fix $x=(0,0,x_3) \in \O$ for $R>x_3>r$ without loss of generality.  We define the plane $S_{\varphi}=\{(x\cos \varphi \sin\theta,x\sin \varphi \sin\theta,x\cos \theta)| 0 \leq x < \infty, \theta \in [0,\pi]\}$ for $\varphi \in [0,2\pi]/\{ \frac{1}{2}\pi, \frac{3}{2}\pi\}$. The cross section $\partial \O \cap S_{\varphi}$ consists of two ellipses, with one smaller ellipse within the larger one, and we only consider the cases where the particle hits the inner ellipse. Next, we define $\zeta_{g,\varphi}=|\zeta_{g,\varphi}|(\cos \varphi \sin\theta_{g}(\varphi),\sin \varphi \sin\theta_{g}(\varphi),\cos \theta_{g}(\varphi))$ such that $n(X_1(x,\zeta_{g,\varphi}))\cdot \zeta_{g,\varphi}=0$ for given $\varphi$.  
Then we have
\begin{align*}
    \eqref{int_C1}
    &\lesssim \int_{[0,2\pi]/\{ \frac{1}{2}\pi, \frac{3}{2}\pi\}}   \int_{0}^{\infty} \int_{0}^{\theta_{g}(\varphi)} 
		\frac{e^{-c|v-\zeta|^{2}}}{|v-\zeta|} \frac{  \langle \zeta \rangle^{r} }{|n(X_1(x,\zeta_{g,\varphi}))\cdot \hat{\zeta}|^{2\b}} 
		|\zeta|^{2-2\b} \sin\theta_{\zeta} \; d|\zeta|d\theta_{\zeta} d\varphi_{\zeta}   \\
    &\lesssim 
    \sup_{\substack{ \hat{v} = (\theta_{v}, \varphi_{v}) \\
		0\leq \theta_{v} \leq \theta_{g}(\varphi)}}
		\int_{[0,2\pi]/\{ \frac{1}{2}\pi, \frac{3}{2}\pi\}}   \int_{0}^{\infty} \int_{0}^{\theta_{g}(\varphi)} 
		\frac{e^{-c|v-\zeta|^{2}}}{ |v-\zeta| } \frac{  \langle \zeta \rangle^{r} }{ \sin^{2\b}(\theta_{g}(\varphi) - \theta_{\zeta})} |\zeta|^{2-2\b}\sin\theta_{\zeta} \; d|\zeta|d\theta_{\zeta} d\varphi_{\zeta}
\end{align*} since $|n(X_1(x,\zeta_{g,\varphi}))\cdot \hat{\zeta}| \leq |n(X_1(x,\zeta))\cdot \hat{\zeta}|$. We obtain optimal $\theta_{g}(\varphi) =\frac{\pi}{2}$, where $x\in\p\O$. Therefore, for $x \in \partial\O$,
\begin{align}
    &\lesssim \sup_{\substack{  \hat{v} = (\theta_{v}, \varphi_{v}) \\ \hat{v}\cdot n(x) \leq 0}}
	\int_{ \{\zeta\cdot n(x) \leq 0\} }
	\frac{e^{-c|v-\zeta|^{2}}}{ |v-\zeta| } \frac{  \langle \zeta \rangle^{r} }{ |n(x) \cdot \hat{\zeta}| } \frac{1}{|\zeta|^{2\b}} d\zeta. \label{sup_int_case1}
\end{align}

Now, we change the sphere coordinate system with $y$-axis as the $z$-axis, $x$-axis as the $y$-axis for convenience. (return to the original setting). We fix $x=(0,r,0) \in \partial \O$, and let\\ $\zeta=|\zeta|(\cos \varphi_{\zeta} \sin\theta_{\zeta},\sin \varphi_{\zeta} \sin\theta_{\zeta},\cos \theta_{\zeta})$ and $v=|v|(\sin \theta_v \cos \varphi_{v},\sin \theta_v \sin \varphi_{v},\cos\theta_v)$ for $\theta_v \in [0,\pi], \varphi_{v}\in [0,2\pi]$. We have 
    \begin{align} \label{int:case3:1:sub}
    \begin{split}
		|v-\zeta|_{\varphi_{v}} &:= \sqrt{|v|^{2} +|\zeta|^{2} - 2|v||\zeta|\cos(\theta_{\zeta} - \theta_{v})} \leq |v-\zeta|,
     \end{split}
    \end{align}where $|v-\zeta|_{\varphi_{v}}$ is 2D distance in a fixed $\varphi_{v}$ plane, when $v$ and $\zeta$ have coordinate $(\theta_{v}, \varphi_{v})$ and $(\theta_{\zeta}, \varphi_{v})$, respectively. Making a change of variable, 
    \begin{align} \label{int:case3:2:sub}
        d\zeta =|\zeta|^2 \sin \theta_{\zeta} \; d|\zeta|d\theta_{\zeta}d\varphi_{\zeta}
        =|\zeta|\sin \theta_{\zeta} d\varphi_{\zeta} \times |\zeta| d|\zeta| d\theta_{\zeta}
        =|\zeta|\sin \theta_{\zeta} d\varphi_{\zeta} \; dA.
    \end{align} We treat $v$ and $\zeta$ as like 2D vectors in fixed $\varphi_{v}$ plane. Applying \eqref{int:case3:1:sub} and \eqref{int:case3:2:sub} to \eqref{sup_int_case1},
    \begin{align*}
        \eqref{sup_int_case1} &\leq \int \frac{1}{\sin^{2\b} \varphi_{\zeta}} d\varphi_{\zeta} \int \frac{e^{-c|v-\zeta|_{\varphi_{v}}^2}}{|v-\zeta|_{\varphi_{v}}} \langle \zeta \rangle^{r} \sin^{1-2\b}\theta_{\zeta} \frac{1}{|\zeta|^{2\b-1}} dA 
        \lesssim_{\beta} \langle v \rangle^{r +1-2\b}
    \end{align*} for $0<\beta <1/2$. We also obtain
    \begin{align*}
          &\int \frac{e^{-c|\zeta|^2}}{|\zeta|}\frac{\langle v+\zeta \rangle^{r}}{|\cos b(x,v+\zeta)|^{2\beta}}\mathbf{1}_{\{{(x,v+\zeta)}\in\text{$\mathcal{C}_1$}\}} \; d\zeta \notag \\ 
          &\leq \int \frac{e^{-c|\zeta|^2}}{|\zeta|}\frac{\langle v+\zeta \rangle^{r+2\b}}{|\cos b(x,v+\zeta)|^{2\beta}}\frac{1}{|v+\zeta|_p^{2\b}}\mathbf{1}_{\{{(x,v+\zeta)}\in\text{$\mathcal{C}_1$}\}} \; d\zeta 
          \lesssim_{\beta} \langle v \rangle^{r+1}
    \end{align*} 
 for $0<\beta<1/2$. We use \eqref{int:case1,2:2} in the last step. If ${(x,v+\zeta)}\in\text{$\mathcal{C}_2$}$, we use the same arguments.
\end{proof}

\begin{lemma} \label{int:case3}
For $x \in \O$, $v \in \mathbb{R}^3, \; c > 0$ and $0<\beta<1/4$, we have
    \begin{align} \label{int:case3:1}
        \int \frac{e^{-c|\zeta|^2}}{|\zeta|}\frac{\langle v+\zeta \rangle^{r}}{|\cos a(x,v+\zeta)|^{4\beta}}\mathbf{1}_{\{{(x,v+\zeta)}\in\text{$\mathcal{C}_3$}\}} \; d\zeta \lesssim_{\beta} \langle v \rangle^{r+1}.
    \end{align} For $x \in \O$, $v \in \mathbb{R}^3, \; c > 0$ and $0<\beta<1/2$, we have
    \begin{align} \label{int:case3:2}
        \int \frac{e^{-c|\zeta|^2}}{|\zeta|}\frac{\langle v+\zeta \rangle^{r}}{|\cos a(x,v+\zeta)|^{2\beta}}\frac{1}{|v+\zeta|_{p}^{2\beta}}\mathbf{1}_{\{{(x,v+\zeta)}\in\text{$\mathcal{C}_3$}\}} \; d\zeta \lesssim_{\beta} \langle v \rangle^{r+1-2\beta},
    \end{align}where $a(x,v)$ in \eqref{def:angle:3}.
\end{lemma}
\begin{proof}
     We assume that $(x,v+\zeta) \in \text{$\mathcal{C}_3$}$. For $x=(|x|_p\cos \theta_x, |x|_p\sin \theta_x, x_3)$, we let $\theta_x=\pi/2$ without loss of generality. Then we have $x=(0,|x|_p,x_3)$. Let $\zeta=|\zeta|(\sin \theta_{\zeta}\cos \varphi_{\zeta},\sin \theta_{\zeta}\sin \varphi_{\zeta}, \cos\theta_{\zeta})=(|\zeta|_p \cos \lambda_{\zeta},|\zeta|_p \sin \lambda_{\zeta}, \zeta_3)$ for $\theta_{\zeta} \in [0,\pi]$ and $\varphi_{\zeta},\lambda_{\zeta} \in [0,2\pi]$. From \eqref{detail a,b:1}, we have 
    \begin{align*}
        \sin^2 a(x,\zeta) = \frac{|x|_p^2}{R^2}\sin^2(\theta_x-\lambda_{\zeta})
        \leq \sin^2(\theta_x-\lambda_{\zeta}) = \cos^2 \lambda_{\zeta}
    \end{align*} and 
    \begin{align*}
        \cos^2 a(x,\zeta) \geq \sin^2 \lambda_{\zeta} = \left(\frac{|\zeta|}{|\zeta|_p}\sin \theta_{\zeta} \sin \varphi_{\zeta}\right)^2,
    \end{align*} where $a(x,v)$ is in \eqref{def:angle:3}.  We use \eqref{int:case3:1:sub} and \eqref{int:case3:2:sub}. Then we have
    \begin{align*}
         &\int \frac{e^{-c|\zeta|^2}}{|\zeta|}\frac{\langle v+\zeta \rangle^{r}}{|\cos a(x,v+\zeta)|^{2\beta}}\frac{1}{|v+\zeta|_p^{2\beta}}\mathbf{1}_{\{{(x,v+\zeta)}\in\text{$\mathcal{C}_3$}\}} \; d\zeta  \notag \\ 
         &\leq \int_{\zeta \in \mathbb{R}^3} \frac{e^{-c|v-\zeta|^2}}{|v-\zeta|}\frac{\langle \zeta \rangle^{r}}{|\sin \theta_{\zeta}\sin \varphi_{\zeta}|^{2\beta}}\frac{1}{|\zeta|^{2\beta}}  \; d\zeta \\
         &\leq \int \frac{1}{\sin^{2\b} \varphi_{\zeta}} d\varphi_{\zeta} \int \frac{e^{-c|v-\zeta|_{\varphi_{v}}^2}}{|v-\zeta|_{\varphi_{v}}} \langle \zeta \rangle^{r} \sin^{1-2\b}\theta_{\zeta} \frac{1}{|\zeta|^{2\b-1}} dA 
            \lesssim_{\beta} \langle v \rangle^{r +1-2\b}
    \end{align*} for $0<\beta <1/2$. Similarly, we have
    \begin{align*}
            &\int \frac{e^{-c|\zeta|^2}}{|\zeta|}\frac{\langle v+\zeta \rangle^{r}}{|\cos a(x,v+\zeta)|^{4\beta}}\mathbf{1}_{\{{(x,v+\zeta)\}}\in\text{$\mathcal{C}_3$}} \; d\zeta \notag \\
            &\leq \int \frac{1}{\sin^{4\b} \varphi_{\zeta}} d\varphi_{\zeta} \int \frac{e^{-c|v-\zeta|_{\varphi_{v}}^2}}{|v-\zeta|_{\varphi_{v}}} \langle \zeta \rangle^{r} \sin^{1-4\b}\theta_{\zeta} \frac{|\zeta|^{2\b}_p}{|\zeta|^{2\b-1}} dA \\
            &\leq \int \frac{1}{\sin^{4\b} \varphi_{\zeta}} d\varphi_{\zeta} \int \frac{e^{-c|v-\zeta|_{\varphi_{v}}^2}}{|v-\zeta|_{\varphi_{v}}} \langle \zeta \rangle^{r} \sin^{1-4\b}\theta_{\zeta} |\zeta| dA
            \lesssim_{\beta}\langle v \rangle^{r+1}
    \end{align*} for $0<\beta <1/4$.
\end{proof} 
\begin{corollary}\label{lem_int T} Recall $k_{c}(v, v+\zeta)$,  $\mathbf{k}_{c}(v, \bv, \zeta)$ in Definition \ref{def_k_c}. For $x, \bx \in \O$, $v,\bv \in \mathbb{R}^3, \; c > 0$ and $0<\beta<1/4$, we have
        \begin{align*}
            \int_{\zeta} k_{c}(v, v+\zeta)   \langle v+\zeta \rangle^{r} \mathcal{T}^{2\b}_{sp,2}(x, \tx, v+\zeta) d\zeta  
		&\lesssim_{\beta} \langle v \rangle^{r+1}.
        \end{align*} For $x,\bx \in \O$, $v,\bv \in \mathbb{R}^3, \; c > 0$ and $0<\beta<1/2$, we have
         \begin{align*}
            \int_{\zeta} k_{c}(v, v+\zeta)   \langle v+\zeta \rangle^{r} \left(\mathcal{S}^{2\b}_{sp}(\bx, v+\zeta)+\mathcal{T}^{2\b}_{sp,1}(x, \tx, v+\zeta) \right) d\zeta  
		&\lesssim_{\beta} \langle v \rangle^{r+1}
  \end{align*} and
        \begin{align*}
            \int_{\zeta} \mathbf{k}_{c}(v, \bv, \zeta) \frac{ \langle v+\zeta \rangle^{r}}{|v+\zeta|_p^{2\b}}   \mathcal{T}^{2\b}_{vel}(\bx, v, \tv, \zeta) d\zeta  
		&\lesssim_{\beta} \langle v \rangle^{r+1-2\beta}.
        \end{align*}
\end{corollary}
\begin{proof}
We apply Lemma \ref{int:case1,2} and Lemma \ref{int:case3}.
\end{proof}

\subsection{Uniform estimates for $\mathfrak{H}^{2\b}_{sp, vel}$}
Fix $t>0,\;x,\bx \in \O$ and $v,\bv,\zeta \in \mathbb{R}^3$.
Since $|e^{-a}-e^{-b}| \leq |a-b|$ for $a\geq b \geq 0$, we obtain the following inequality from \eqref{f_expan}, 
 \begin{align} 
& |f(t,x,v+ \zeta)-f(t,\bar{x}, v+ \zeta)| \notag \\
&\leq   
|f(0,X(0 ), V(0 ))-f(0,{X}^{*}(0 ), {V}^{*}(0 ))| 
 \label{bbasic f-f1} \\
&\quad + \int^t_0 
\left|\Gamma_{\text{gain}}(f,f)(s,X(s ), V(s ))
-\Gamma_{\text{gain}}(f,f)(s,{X}^{*}(s ), {V}^{*}(s )) \right|  ds 
 \label{bbasic f-f2} \\
&\quad +  \|w_{0}f_{0}\|_{\infty} \frac{1}{w_{0}(\bv+\zeta)} 
\int_{0}^{t } | \nu(f) (s, X(s), V(s ))  - \nu(f) (s, {X}^{*}(s), {V}^{*}(s)) | ds
 \label{bbasic f-f3} \\
&\quad + t \sup_{0\leq s \leq t} \|wf(s)\|^{2}_{\infty}
\frac{1}{\sqrt{w(\bv+\zeta)}}  
\int_{0}^{t } | \nu(f) (s, X(s), V(s ))  - \nu(f) (s, {X}^{*}(s), {V}^{*}(s)) | ds , \label{bbasic f-f4}  
\end{align}
where 
\begin{align*}
	X(s)= X(s;t, x, v+\zeta),\quad V(s)= V(s;t, x, v+\zeta),\quad {X}^{*}(s) = X(s;t, \bx, v+\zeta), \quad {V}^{*}(s) = V(s;t, \bx, v+\zeta).
\end{align*} 

Now, we estimate the seminorms in Definition \ref{def:seminorm} by using Lemma  \ref{lem:f-f:txbx} - \ref{lem:f-f:tvv} and  Corollary \ref{lem_int T}.

\begin{proposition}\label{prop_unif H} 
There exists $\varpi \gg_{f_{0}, \b} 1$ such that
\begin{align*}
\begin{split}
	&\left[ \sup_{0\leq s \leq T}\mathfrak{H}_{sp}^{2\b}(s) + \sup_{0\leq s \leq T}\mathfrak{H}_{vel}^{2\b}(s) \right]  \\
	&\lesssim_{\b} \sup_{\substack{v\in\R^{3} \\ 0 < |x - \bx|\leq 1}} 
	\langle v \rangle \frac{|f_{0}( x, v ) - f_{0}(\bx, v)|}{|x - \bx|^{2\b}}  
	+ \sup_{\substack{x\in\bar{\O} \\ 0 < |v - \bv|\leq 1}}  \langle v \rangle^{2} \frac{|f_{0}( x, v ) - f_{0}( x, \bv)|}{|v - \bv|^{2\b}}  
	+ \|w_{0} f_{0}\|_{\infty}
\end{split}
\end{align*}
for sufficiently small $T>0$ such that $\varpi T \ll 1$ and $0<\beta <1/4$.
\end{proposition}
\begin{proof} First, we estimate $\mathfrak{H}_{sp}^{2\b}(s)$. To estimate \eqref{bbasic f-f1}, we replace $s=0$ in \eqref{split1} and \eqref{split2}. Next, we apply Lemma \ref{lem:f-f:txbx}, Lemma \ref{lem:f-f:txx} and Corollary \ref{lem_int T}. For $0<\beta<1/4$, we have
	\begin{align} \label{pro s=0}
	\begin{split}
		&e^{-\varpi \langle v \rangle^{2} t}  
		\int_{\R_{\zeta}^{3}} k_{c}(v, v+\zeta) \frac{ |\eqref{bbasic f-f1}| }{|x-\bx|^{2\b}} d\zeta \\
		&\lesssim e^{-\varpi \langle v \rangle^{2} t}  
		\int_{\R_{\zeta}^{3}} k_{c}(v, v+\zeta)  
		\left[ ( 1 + |v+\zeta| t ) \Big(\mathcal{T}_{sp,1}(x, \tx, v+\zeta)+\mathcal{S}_{sp}(\bx, v+\zeta)\Big)
		\right]^{2\b}  \mathcal{X}(2\b,0,v,\zeta) \; d\zeta
		\\
		&\quad+ e^{-\varpi \langle v \rangle^{2} t}  
		\int_{\R_{\zeta}^{3}} k_{c}(v, v+\zeta) 
		\Big[ (|v+\zeta| +  |v+\zeta|^{2}t) \Big(\mathcal{T}_{sp,2}(x, \tx, v+\zeta)+\mathcal{S}_{sp}(\bx, v+\zeta)\Big)
		\Big]^{2\b}   \mathcal{V}(2\b,0,v,\zeta) \; d\zeta \\
		&\lesssim_{\b}  
		\sup_{\substack{v\in\R^{3} \\ 0 < |x - \bx|\leq 1}} 
		\langle v \rangle  \frac{|f_{0}( x, v ) - f_{0}(\bx, v)|}{|x - \bx|^{2\b}}  
		+ \sup_{ \substack{ x\in \O \\ 0 < |v - \bv|\leq 1   } }  \langle v \rangle^{2} \frac{|f_{0}( x, v ) - f_{0}( x, \bv)|}{|v - \bv|^{2\b}}  
		+ \|w_{0} f_{0}\|_{\infty}.
	\end{split}
	\end{align}
    To estimate the difference of $\Gamma_{\text{gain}}(f,f)(s,{X}(s ), {V}(s ))$ in \eqref{bbasic f-f2}, we use \eqref{full k v}, \eqref{full k x} in Lemma \ref{lem_Gamma} and Lemma \ref{lem_nega}. For $0<\beta<1/4$, we have
	\begin{equation*}
	\begin{split}
	& e^{-\varpi \langle v \rangle^{2} t} 
	\int_{\R_{\zeta}^{3}} k_{c}(v, v+\zeta) \frac{ |\eqref{bbasic f-f2} | }{|x-\bx|^{2\b}} d\zeta    \\
	&\lesssim \int_{0}^{t} e^{-\varpi \langle v \rangle^{2} (t-s)}  
	\int_{\R_{\zeta}^{3}} k_{\frac{c}{2}}(v, v+\zeta)
	\Big[ ( 1 + |v+\zeta| t ) \Big(\mathcal{T}_{sp,1}(x, \tx, v+\zeta)+\mathcal{S}_{sp}(\bx, v+\zeta)\Big)
	\Big]^{2\b}  
	\\
	&\quad\quad 
	\times  
	\left[  
		\|wf(s)\|_{\infty}
		\mathfrak{H}_{sp}^{2\b}(s)
		+ \frac{\langle v+\zeta \rangle}{w(v+\zeta)}\|w f(s)\|_{\infty} \right] 
	d\zeta ds  \\
	&\quad+ \int_{0}^{t} e^{-\varpi \langle v \rangle^{2} (t-s)}  
	\int_{\R_{\zeta}^{3}} k_{\frac{c}{2}}(v, v+\zeta)
	\Big[ (|v+\zeta| +  |v+\zeta|^{2}t) \Big(\mathcal{T}_{sp,2}(x, \tx, v+\zeta)+\mathcal{S}_{sp}(\bx, v+\zeta)\Big)
	\Big]^{2\b}    \\
	&\quad\quad 
	\times \left[
		\|wf(s)\|_{\infty}
		\mathfrak{H}_{vel}^{2\b}(s)
		+ \left(  \frac{1}{\langle v+\zeta \rangle}  + \frac{\langle v+\zeta \rangle}{w(v+\zeta)} \right) \|wf(s)\|^{2}_{\infty}  \right]
	d\zeta ds \\
	&\lesssim_{\b} \frac{1}{\varpi} 
	\left[ \sup_{0\leq s \leq T}\mathfrak{H}_{sp}^{2\b}(s) + \sup_{0\leq s \leq T}\mathfrak{H}_{vel}^{2\b}(s) \right] 
	\mathcal{P}_{2}( \|w_{0}f_{0}\|_{\infty}),
	\end{split}
	\end{equation*} where $\mathcal{P}_{2}(s) = |s| + |s|^2$. To estimate the difference of $ \nu(f) (s, X(s), V(s ))$, we use \eqref{full nu v}, \eqref{full nu v} in Lemma \ref{lem_Gamma}. Then, we obtain 
 	\begin{equation*} 
	\begin{split}
	& e^{-\varpi \langle v \rangle^{2} t} 
	\int_{\R_{\zeta}^{3}} k_{c}(v, v+\zeta) \frac{ |\eqref{bbasic f-f3} | +\eqref{bbasic f-f4}}{|x - \bx|^{2\b}} d\zeta   
	\lesssim_{\b} \frac{1}{\varpi} 
	\left[ \sup_{0\leq s \leq T}\mathfrak{H}_{sp}^{2\b}(s) + \sup_{0\leq s \leq T}\mathfrak{H}_{vel}^{2\b}(s) \right] 
	\mathcal{P}_{3}( \|w_{0}f_{0}\|_{\infty}), 
	\end{split} 
	\end{equation*} where $\mathcal{P}_{3}(s) = |s| + |s|^2+|s|^3$.
    Therefore, we conclude
    	\begin{align*}
	\begin{split}
		\sup_{0\leq s \leq T}\mathfrak{H}_{sp}^{2\b}(s)  
		&\lesssim_{\b} \sup_{\substack{v\in\R^{3} \\ |x - \bx|\leq 1}} 
		\langle v \rangle  \frac{|f_{0}( x, v ) - f_{0}(\bx, v)|}{|x - \bx|^{2\b}}  
		+ \sup_{ \substack{ x\in \O \\ |v - \bv|\leq 1   } }  \langle v \rangle^{2} \frac{|f_{0}( x, v ) - f_{0}( x, \bv)|}{|v - \bv|^{2\b}}  
		+ \|w_{0} f_{0}\|_{\infty} \\
		&\quad + \frac{1}{\varpi}  \left[ \sup_{0\leq s \leq T}\mathfrak{H}_{sp}^{2\b}(s) + \sup_{0\leq s \leq T}\mathfrak{H}_{vel}^{2\b}(s) \right] \mathcal{P}_{3}( \|w_{0}f_{0}\|_{\infty})
	\end{split} 
	\end{align*} for some $\varpi \gg_{f_0, \b} 1$ and $T \leq T^{*}$, where $T^{*}$ is local existence time in Lemma \ref{lem loc}. Similarly, if we apply Lemma \ref{lem:f-f:tvbv}, Lemma \ref{lem:f-f:tvv} instead of Lemma \ref{lem:f-f:txbx}, Lemma \ref{lem:f-f:txx}, then 
 	\begin{align*}
	\begin{split}
		\sup_{0\leq s \leq T}\mathfrak{H}_{vel}^{2\b}(s)  
		&\lesssim_{\b} \sup_{\substack{v\in\R^{3} \\ |x - \bx|\leq 1}} 
		\langle v \rangle  \frac{|f_{0}( x, v ) - f_{0}(\bx, v)|}{|x - \bx|^{2\b}}  
		+ \sup_{ \substack{ x\in \O \\ |v - \bv|\leq 1   } }  \langle v \rangle^{2} \frac{|f_{0}( x, v ) - f_{0}( x, \bv)|}{|v - \bv|^{2\b}}  
		+ \|w_{0} f_{0}\|_{\infty} \\
		&\quad + \frac{1}{\varpi}  \left[ \sup_{0\leq s \leq T}\mathfrak{H}_{sp}^{2\b}(s) + \sup_{0\leq s \leq T}\mathfrak{H}_{vel}^{2\b}(s) \right] \mathcal{P}_{3}( \|w_{0}f_{0}\|_{\infty}).
	\end{split} 
	\end{align*} Lastly, we choose sufficiently large $\varpi \gg_{f_{0}, \b} 1$. 
\end{proof} 

\section{\texorpdfstring{$C^{0,\frac{1}{2}}_{x,v}$}{} estimates of trajectory}

In section 6, we let $f : \overline{\O} \times \R^{3} \rightarrow \R_{+} \cup \{0\}$ be a function which satisfies specular reflection \eqref{specular}, where $\O$ is a domain as in Definition \ref{def:domain} and $w(v) = e^{\vartheta|v|^{2}}$ for some $0 < \vartheta$. Let fix any point $t,s \in \mathbb{R}^{+}$, $v, \bv, \zeta \in \mathbb{R}^3$, $x,\bx \in \O$. We assume \eqref{assume_x}, \eqref{assume_v}, and recall $\tx(x,\bx,v+\zeta)$ in \eqref{def:tx}, $\X(\tau)$ in \eqref{def:x_para}, $\tv(v,\bv,\zeta)$ in \eqref{def:tv} and $\V(\tau)$ in \eqref{def:v_para}. 

\subsection{$C^{0,\frac{1}{2}}_{x,v}$ estimates of trajectory for $\mathcal{C}_1$ and $\mathcal{C}_2$}
\begin{lemma}\label{est:angle,time:geo} 
Assume that $(x,v+\zeta), (\tx,v+\zeta) \in \mathcal{C}_1$ or $(x,v+\zeta), (\tx,v+\zeta) \in \mathcal{C}_2$ with moving clockwise.
For $|x-\tx| \leq 1$ and $i=b,f,*$,  we have
    \begin{align} \label{est:t:txx:geo:case1,2}
        \begin{split}
            |t_i(x,v+\zeta)-t_i(\tx,v+\zeta)| \lesssim \frac{1}{|v+\zeta|_p} |x-\tx|^{1/2}
        \end{split} 
    \end{align} and 
    \begin{align} \label{est:a:txx:geo:case1,2}
        \begin{split}
            &|a(x,v+\zeta)-a(\tx,v+\zeta)|+|b(x,v+\zeta)-b(\tx,v+\zeta)|\lesssim  |x-\tx|^{1/2},
        \end{split} 
    \end{align}
    where $a(x,v), b(x,v)$ are in \eqref{def:angle:1} or \eqref{def:angle:2}. Assume that $(x,v+\zeta), (x,\tv+\zeta) \in \mathcal{C}_1$ or $(x,v+\zeta), (x,\tv+\zeta) \in \mathcal{C}_2$ with moving clockwise.
For $|v-\tv| \leq 1$ and $i=b,f,*$,
    \begin{align} \label{est:t:tvv:geo:casse1,2}
        \begin{split}
            &|t_i(x,v+\zeta)-t_i(x,\tv+\zeta)| \lesssim \frac{1}{|v+\zeta|^{2}_p}(1+|v+\zeta|^{1/2}_p) |v-\tv|^{1/2}
        \end{split} 
    \end{align} and
    \begin{align} \label{est:a:tvv:geo:case1,2}
        \begin{split}
            &|a(x,v+\zeta)-a(x,\tv+\zeta)|\lesssim  \frac{1}{|v+\zeta|_p}|v-\tv|^{1/2}
        \end{split} 
    \end{align} 
    and
    \begin{align} \label{est:b:tvv:geo:case1,2}
        \begin{split}
            &|b(x,v+\zeta)-b(x,\tv+\zeta)|\lesssim \frac{1}{|v+\zeta|_p} (1+|v+\zeta|^{1/2}_p ) |v-\tv|^{1/2}.
        \end{split} 
    \end{align} 
\end{lemma}

\begin{proof}
     We let $v=v+\zeta=(v_1,v_2,v_3)$ and $x=(x_1,x_2,x_3)$. We assume that $(x,v),(\tx,v) \in \mathcal{C}_1$ with moving clockwise. Because $|X_1(x,v)|_p=|(x_1-t_b(x,v)v_1,x_2-t_b(x,v)v_2,0)|=r$,
    \begin{align} \label{tb_eq:case1}
        |v|_p^2 t_b(x,v) = (x_p \cdot v_p) - \sqrt{(x_p \cdot v_p)^2-|v|_p^2(|x|_p^2-r^2)}.
    \end{align} From \eqref{tb_eq:case1}, 
    \begin{align} \label{tb-tb:geo:case1}
    \begin{split}
        &\quad |v|_p^2 |t_b(x,v)-t_b(\tx,v)|  \\
        &\leq \left|\sqrt{(x_p \cdot v_p)^2-|v|_p^2(|x|_p^2-r^2)}-\sqrt{(\tx_p \cdot v_p)^2-|v|_p^2(|\tx|_p^2-r^2)}\right|  \\
        &\leq \left|(x_p \cdot v_p)^2-(\tx_p \cdot v_p)^2-|v|_p^2(|x|_p^2-r^2-|\tx|_p^2+r^2) \right|^{1/2}\\
        &\lesssim |v|_p|x-\tx|^{1/2},
    \end{split} 
    \end{align} since $(x_p-\tx_p) \cdot v_p =0$.
    From \eqref{cosb_eq:case1},
    \begin{align*} 
        r|v|_p|\cos b(x,v) - \cos b(\tx,v)| &\leq
       |v|_p^2|t_b(x,v)-t_b(\tx,v)| \lesssim |v|_p|x-\tx|^{1/2}
    \end{align*} Using \eqref{est:sinb:x},
\begin{align*} 
    |b(x,v)-b(\tx,v)| \leq \frac{\pi}{2}|\cos b(x,v) - \cos b(\tx,v)|+\frac{\pi}{2} |\sin b(x,v) - \sin b(\tx,v)|\lesssim|x-\tx|^{1/2}.
\end{align*} 
    From
    \begin{align*}
         |v|_p^2 t_f(x,v) = -(x_p \cdot v_p) + \sqrt{(x_p \cdot v_p)^2+|v|_p^2(R^2-|x|_p^2)}
    \end{align*} and \eqref{cosa_eq:case1:geo}, we also estimate 
    $|t_f(x,v)-t_f(\tx,v)|$ and $|a(x,v)-a(\tx,v)|$. From \eqref{tb_eq:case1}, we have 
    \begin{align*}
        &\quad |v|_p^2 |t_b(x,v)-t_b(x,\tv)| \notag \\
        &\leq |(v_p-\tv_p)\cdot x_p|+\left|\sqrt{(x_p \cdot v_p)^2-|v|_p^2(|x|_p^2-r^2)}-\sqrt{(x_p \cdot \tv_p)^2-|v|_p^2(|x|_p^2-r^2)} \right| \notag \\
        &\leq |(v_p-\tv_p)\cdot x_p|+|(x_p \cdot v_p)^2-(x_p \cdot \tv_p)^2|^{1/2} \notag \\
        &\lesssim |v-\tv|^{1/2}(1+|v|_p^{1/2}).
    \end{align*} From \eqref{cosb_eq:case1} and \eqref{est:sinb:v}, we have \eqref{est:b:tvv:geo:case1,2}, and omit details. 
\end{proof}

\begin{lemma}\label{geo:x-tx}
Assume that $(x,v+\zeta), (\tx,v+\zeta) \in \mathcal{C}_1$ or $(x,v+\zeta), (\tx,v+\zeta) \in \mathcal{C}_2$ with moving clockwise. For $|x-\tx| \leq  1$, we have
    \begin{align*}
    \begin{split}
          &\quad |f(s, X(s;t,x,v+\zeta), V(s;t,x,v+\zeta)) - f(s, X(s;t, \tilde{x}, v+\zeta ), V(s;t, \tilde{x}, v+\zeta ))|\\
          &\lesssim \;  \Big[1+|v+\zeta|_p(t-s)\Big]^{2\gamma} \big(\mathcal{X}(2\gamma,s,v,\zeta)  + |v+\zeta|_p^{2\gamma} \mathcal{V}(2\gamma,s,v,\zeta)  \big)|x-\bx|^{2\gamma}.
    \end{split}
    \end{align*}
\end{lemma}
\begin{proof}
In the proof of Lemma \ref{x-tx}, we apply difference estimates of $t_b(x,v), t_f(x,v), t_*(x,v)$ in Lemma \ref{est:time} and $a(x,v), b(x,v)$ in Lemma \ref{est:angle}. Here, we follow the proof of Lemma \ref{x-tx}, but we apply difference estimates in Lemma \ref{est:angle,time:geo}, not in Lemma \ref{est:time}  and Lemma \ref{est:angle}.
\end{proof}

\begin{lemma} \label{geo:v-tv}
Assume $(x,v+\zeta), (x,\tv+\zeta) \in \mathcal{C}_1$ or $(x,v+\zeta), (x,\tv+\zeta) \in \mathcal{C}_2$ with moving clockwise. For $|v-\tv| \leq 1$, we have
    \begin{align*} 
    \begin{split}
          &\quad |f(s, X(s;t,x,v+\zeta), V(s;t,x,v+\zeta)) - f(s, X(s;t, x, \tv+\zeta ), V(s;t, x, \tv+\zeta ))|  \\
          &\lesssim
          \Big[ \frac{1}{|v+\zeta|_p}+(t-s)\Big]^{2\gamma}\left(\mathcal{X}(2\gamma,s,v,\zeta)  + |v+\zeta|_p^{2\gamma} \mathcal{V}(2\gamma,s,v,\zeta)  \right)|v-\bv|^{2\gamma}.
   \end{split}
    \end{align*}
\end{lemma}
\begin{proof}
We follow the proof of Lemma \ref{v-tv}, but we apply difference estimates in Lemma \ref{est:angle,time:geo}, not in Lemma \ref{est:time}  and Lemma \ref{est:angle}.
\end{proof}

\subsection{$C^{0,\frac{1}{2}}_{x,v}$ estimates of trajectory for $\mathcal{C}_3$}
\begin{lemma}\label{1/2:x:est:angle,time:case3} Assume that $(x,v+\zeta), (\tx,v+\zeta) \in \mathcal{C}_3$ with moving clockwise.
For  $|x-\tx| \leq 1$ and $i=b,f,*$, we have
    \begin{align} \label{1/2:x:est:t:case3}
        \begin{split}
            &|t_i(x,v+\zeta)-t_i(\tx,v+\zeta)|  \lesssim  \frac{1}{|v+\zeta|_p}|x-\tx|^{1/2}
        \end{split} 
    \end{align} and
    \begin{align} \label{1/2:x:est:a:case3}
        \begin{split}
            |a(x,v+\zeta)-a(\tx,v+\zeta)|\lesssim |x-\tx|^{1/2},
        \end{split} 
    \end{align} where $a(x,v)$ is in \eqref{def:angle:3}. 
\end{lemma}
    \begin{proof} We let $v=v+\zeta$.
       It holds that 
          \begin{align*} 
        |v|_p^2 t_b(x,v) = (x_p \cdot v_p) - \sqrt{(x_p \cdot v_p)^2-|v|_p^2(|x|_p^2-R^2)}
    \end{align*} and \eqref{tf_eq:geo:case1,2}. We have \eqref{1/2:x:est:t:case3} by same argument to \eqref{tb-tb:geo:case1}. From \eqref{cosa_eq:case1:geo} and \eqref{est:sinb:x}, we have \eqref{1/2:x:est:a:case3}.
    \end{proof}

  \begin{lemma}\label{1/2:x-tx:3}
Assume that $(x,v+\zeta), (\tx,v+\zeta) \in \mathcal{C}_3$ with moving clockwise. For $|x-\tx| \leq 1$, we have
    \begin{align} \label{1/2:f-f:tx:3}
    \begin{split}
          &|f(s, X(s;t,x,v+\zeta), V(s;t,x,v+\zeta)) - f(s, X(s;t, \tilde{x}, v+\zeta ), V(s;t, \tilde{x}, v+\zeta ))|  \\
          &\lesssim \left[\Big(1+|v+\zeta|_p(t-s) \Big) \times \min \left \{\frac{1}{\cos a(x,v+\zeta)}, \frac{1}{\cos a(\tx,v+\zeta)}\right \}\right]^{2\gamma}  \\ &\quad \quad  \times \Big(\mathcal{X}(2\gamma,s,v,\zeta)+|v+\zeta|^{2\gamma}_p \mathcal{V}(2\gamma,s,v,\zeta) \Big) |x-\bx|^{\gamma},
    \end{split}
    \end{align}  where $a(x,v)$ is in \eqref{def:angle:3}.
\end{lemma}
\begin{proof}
In the proof of Lemma \ref{v-tv:3}, we apply difference estimates of $t_b(x,v), t_f(x,v), t_*(x,v)$ and $a(x,v)$ in Lemma \ref{v:est:angle,time:case3}. Here, we follow the proof of  Lemma \ref{v-tv:3}, but we estimate the difference in $x$, not $v$. Also, we apply difference estimates in Lemma \ref{1/2:x:est:angle,time:case3}, not in Lemma \ref{v:est:angle,time:case3}. Because we do not need the assumption that $x_p \cdot v$ and $x_p \cdot \tv$ have the same sign in Lemma \ref{v:est:angle,time:case3} and \eqref{sub:case3:Vv:sing:2}, there is
\begin{align*}
    \min \left \{\frac{1}{\cos a(x,v+\zeta)}, \frac{1}{\cos a(\tx,v+\zeta)}\right \}
\end{align*} in \eqref{1/2:f-f:tx:3}, not maximum. 
\end{proof}

\begin{lemma}\label{1/2:v:est:angle,time:case3} Assume that $(x,v+\zeta), (x,\tv+\zeta) \in \mathcal{C}_3$ with moving clockwise, and $x_p \cdot v$ and $x_p \cdot \tv$ have the same sign.
For $|v-\tv| \leq 1$,
    \begin{align*} 
        \begin{split}
            &|t_{*}(x,v+\zeta)-t_{*}(x,\tv+\zeta)| \lesssim \max \left\{\sqrt{\cos a(x,v+\zeta)}, \sqrt{\cos a(x,\tv+\zeta)} \right\}  \frac{1}{|v+\zeta|^{2}_p} |v-\tv|^{1/2}
        \end{split} 
    \end{align*} and
    \begin{align*} 
        \begin{split}
         |a(x,v+\zeta)-a(x,\tv+\zeta)\lesssim \max \left\{\sqrt{\cos a(x,v+\zeta)}, \sqrt{\cos a(x,\tv+\zeta)} \right\}   \frac{1}{|v+\zeta|_p}|v-\tv|^{1/2},
        \end{split} 
    \end{align*} where $a(x,v)$ is in \eqref{def:angle:3}. 
\end{lemma}
\begin{proof} We let $v=v+\zeta$ and $\tv=\tv+\zeta$. From \eqref{v:est:a:case3}, we have
\begin{align*}
    |v|_p|t_*(x,v)-t_*(x,\tv)| &\leq R   |\cos a(x,v)-\cos a(x,\tv)|\\
     &\lesssim \;\max \left\{\sqrt{\cos a(x,v)}, \sqrt{\cos a(x,\tv)} \right\}   |\cos a(x,v)-\cos a(x,\tv)|^{1/2}
   \\ &\lesssim \;\max \left\{\sqrt{\cos a(x,v)}, \sqrt{\cos a(x,\tv)} \right\} \frac{1}{|v|_p}|v-\tv|^{1/2}.
\end{align*}
Since $a(x,v),a(x,\tv) \in  [0,\pi/2)$, we have 
\begin{align} \label{a-a:Case3:1/2}
\begin{split}
    |a(x,v)-a(x,\tv)|
    &\leq \sin \left(\frac{\pi}{2}-a(x,v) \right)+\sin \left(\frac{\pi}{2}-a(x,\tv) \right)
    \\ &\leq 2\;\max \left\{\cos a(x,v), \cos a(x,\tv) \right\}.
\end{split}
\end{align} From \eqref{v:est:a:case3} and \eqref{a-a:Case3:1/2}, we obtain
\begin{align*}
    |a(x,v)-a(x,\bv)| &\leq \sqrt{2}\;\max \left\{\sqrt{\cos a(x,v)}, \sqrt{\cos a(x,\tv)} \right\} |a(x,v)- a(x,\tv)|^{1/2} \\
    &\leq  \sqrt{2}\;\max \left\{\sqrt{\cos a(x,v)}, \sqrt{\cos a(x,\tv)} \right\} \frac{1}{|v+\zeta|_p}|v-\tv|^{1/2}. 
\end{align*} 
\end{proof}

\begin{lemma}\label{1/2:v-tv:3}
Assume that $(x,\V(\tau)) \in \mathcal{C}_3$ with moving clockwise for all $\tau \in [0,1]$. For $|v-\tv| \leq 1$,
   \begin{align*} 
   \begin{split}
          &|f(s, X(s;t,x,v+\zeta), V(s;t,x,v+\zeta)) - f(s, X(s;t, x, \tv+\zeta ), V(s;t, x, \tv+\zeta ))|\\
          &\lesssim
          \bigg(\left[\Big( \frac{1}{|v+\zeta|_p}+(t-s)\Big) \right]^{2\gamma}   \mathcal{X}(2\gamma,s,v,\zeta)\\
	&\quad+ \left[ \Big(1+|v+\zeta|_p(t-s)\Big) \times \max \left\{\frac{1}{\cos^{1/2} a(x,v+\zeta)},\frac{1}{\cos^{1/2} a(x,\tv+\zeta)}\right \}\right]^{2\gamma}  \mathcal{V}(2\gamma,s,v,\zeta) \bigg) |v-\bv|^{\gamma},
    \end{split}
    \end{align*}  where $a(x,v)$ is in \eqref{def:angle:3}.
\end{lemma}
\begin{proof}
We follow the proof of  Lemma \ref{v-tv:3}. In \eqref{a-a:tvv:3},\eqref{t-t:tvv;3}, we apply difference estimates in Lemma  \ref{1/2:v:est:angle,time:case3}, not in Lemma \ref{v:est:angle,time:case3}. 
\end{proof}
\subsection{$C^{0,\frac{1}{2}}_{x,v}$ estimates of trajectory for all cases}
\begin{lemma}\label{lem:f-f:txx:sec7}
      For $|x-\bx| \leq 1$,  we have
    \begin{align*}
  \begin{split}
         \eqref{split1}+\eqref{split2}
         &\lesssim \Big[\Big(1+|v+\zeta|_p(t-s)\Big)\times \mathcal{H}_{sp}(x,\tx,v+\zeta)\Big]^{2\b} \\  
         &\quad \times \Big(\mathcal{X}(2\b,s,v,\zeta)+|v+\zeta|^{2\b}_p\mathcal{V}(2\b,s,v,\zeta)\Big) |x-\bx|^{\b},
    \end{split}
    \end{align*} where $\mathcal{H}_{sp}(x,\tx,v+\zeta)$ is defined as
    \begin{align*}
        \mathcal{H}_{sp}(x,\tx,v+\zeta)=\min \left\{\frac{1}{\cos a(x,v+\zeta)} \mathbf{1}_{\{{(x,v+\zeta)}\in\text{$\mathcal{C}_3$}\}},\frac{1}{\cos a(\tx,v+\zeta)} \mathbf{1}_{\{{(\tx,v+\zeta)}\in\text{$\mathcal{C}_3$}\}} \right\}
    \end{align*} for $a(x,v)$ is in \eqref{def:angle:3}.
\end{lemma}
\begin{proof}
In the proof of Lemma \ref{lem:f-f:txbx} and Lemma \ref{lem:f-f:txx}, we apply difference estimates of singular parts of $\mathcal{C}_{1.2}$ in Lemma \ref{x-tx} and $\mathcal{C}_{3}$ in Lemma \ref{x-tx:3}. We follow the proof of Lemma \ref{lem:f-f:txbx} and  Lemma \ref{lem:f-f:txx}, but we apply difference estimates of singular parts of $\mathcal{C}_{1.2}$ in Lemma \ref{geo:x-tx} and $\mathcal{C}_{3}$ in Lemma \ref{1/2:x-tx:3}, not in Lemma \ref{x-tx} and Lemma \ref{x-tx:3}.
\end{proof}

\begin{lemma} \label{lem:f-f:tvv:sec7}
For $|v-\bv| \leq 1$,
    \begin{align*}
    \begin{split}
        \eqref{split4}
         &\lesssim \bigg(\left[\Big( \frac{1}{|v+\zeta|_p}+(t-s)\Big)\right]^{2\b} 
         \mathcal{X}(2\b,s,v,\zeta) \\
	&\quad+ \Big[\Big(1+|v+\zeta|_p(t-s)\Big) \times \mathcal{H}_{vel}(\bx,v,\tv,\zeta)\Big]^{2\b}  \mathcal{V}(2\b,s,v,\zeta) \bigg) |v-\bv|^{\b},
    \end{split} 
    \end{align*} where $\mathcal{H}_{vel}(\bx, v,\tv,\zeta)$ is defined as
    \begin{align*}
        \mathcal{H}_{vel}(\bx, v,\tv,\zeta)=\max \left\{\frac{1}{\cos^{1/2} a(\bx,v+\zeta)} \mathbf{1}_{\{{(\bx,v+\zeta)}\in\text{$\mathcal{C}_3$}\}},\frac{1}{\cos^{1/2}  a(\bx,\tv+\zeta)} \mathbf{1}_{\{{(\bx,\tv+\zeta)}\in\text{$\mathcal{C}_3$}\}} \right\}
    \end{align*} for $a(x,v)$ is in \eqref{def:angle:3}.
\end{lemma}
\begin{proof}
In the proof of Lemma \ref{lem:f-f:tvv}, we apply difference estimates of singular parts of $\mathcal{C}_{1.2}$ in Lemma \ref{v-tv} and $\mathcal{C}_{3}$ in Lemma \ref{v-tv:3}. We follow the proof of Lemma \ref{lem:f-f:tvv}, but we apply difference estimates of singular parts of $\mathcal{C}_{1.2}$ in Lemma \ref{geo:v-tv} and $\mathcal{C}_{3}$ in Lemma \ref{1/2:v-tv:3}, not in Lemma \ref{v-tv} and Lemma \ref{v-tv:3}.
\end{proof}

\section{H\"older type regularity with singular weight}

 \begin{proof} [Proof of Theorem \ref{theo:Holder}] 
 Fix $t>0,\;x,\bx \in \O$ and $v,\bv,\zeta \in \mathbb{R}^3$. We rewrite
 \begin{align} 
& |f(t,x,v+ \zeta)-f(t,\bar{x}, v+ \zeta)|  \notag \\
&\leq   
|f(0,X(0 ), V(0 ))-f(0,{X}^{*}(0 ), {V}^{*}(0 ))| 
 \label{basic f-f1} \\
&\quad + \int^t_0 
|\Gamma_{\text{gain}}(f,f)(s,X(s ), V(s ))
-\Gamma_{\text{gain}}(f,f)(s,{X}^{*}(s ), {V}^{*}(s ))|  ds 
 \label{basic f-f2} \\
&\quad +  \|w_{0}f_{0}\|_{\infty} \frac{1}{w_{0}(\bv+\zeta)} 
\int_{0}^{t } | \nu(f) (s, X(s), V(s ))  - \nu(f) (s, {X}^{*}(s), {V}^{*}(s)) | ds
 \label{basic f-f3} \\
&\quad + t \sup_{0\leq s \leq t} \|wf(s)\|^{2}_{\infty}
\frac{1}{\sqrt{w(\bv+\zeta)}}  
\int_{0}^{t } | \nu(f) (s, X(s), V(s ))  - \nu(f) (s, {X}^{*}(s), {V}^{*}(s)) | ds , \label{basic f-f4}  
\end{align}
where 
\begin{align*}
	X(s)= X(s;t, x, v+\zeta),\quad V(s)= V(s;t, x, v+\zeta),\quad {X}^{*}(s)= X(s;t, \bx, v+\zeta), \quad {V}^{*}(s) = V(s;t, \bx, v+\zeta),
\end{align*} and
\begin{align} 
& |f(t,\bx,v+ \zeta)-f(t,\bx, \bv+ \zeta)|  \notag \\
&\leq   
|f(0,{X}^{*}(0), {V}^{*}(0))-f(0,\bar{X}(0 ), \bar{V}(0))| 
 \label{v:basic f-f1} \\
&\quad + \int^t_0 
|\Gamma_{\text{gain}}(f,f)(s,{X}^{*}(s), {V}^{*}(s))
-\Gamma_{\text{gain}}(f,f)(s,\bar{X}(s ), \bar{V}(s ))|  ds 
 \label{v:basic f-f2} \\
&\quad +  \|w_{0}f_{0}\|_{\infty} \frac{1}{w_{0}(\bv+\zeta)} 
\int_{0}^{t } | \nu(f) (s, {X}^{*}(s), {V}^{*}(s))  - \nu(f) (s, \bar{X}(s), \bar{V}(s)) | ds
 \label{v:basic f-f3} \\
&\quad + t \sup_{0\leq s \leq t} \|wf(s)\|^{2}_{\infty}
\frac{1}{\sqrt{w(\bv+\zeta)}}  
\int_{0}^{t } | \nu(f) (s, {X}^{*}(s), {V}^{*}(s))  - \nu(f) (s, \bar{X}(s), \bar{V}(s)) | ds , \label{v:basic f-f4}  
\end{align}
where 
\begin{align*}
   \bar{X}(s)= {X}(s;t, \bx, \bv+\zeta), \quad \bar{V}(s)= {V}(s;t, \bx, \bv+\zeta).
\end{align*} \\
Let $|x-\bx| \leq 1$ and $t<1$. We assume that \eqref{assume_x}, and recall $\tx(x,\bx,v+\zeta)$ in \eqref{def:tx}. We replace $s=0$ to \eqref{split1} and \eqref{split2}. By Lemma \ref{lem:f-f:txx:sec7}, we have
	\begin{align} \label{est : f-f1-half}
 	\begin{split}
 	\eqref{basic f-f1}
 	&\lesssim 
 	|x-\bx|^{\b}\Big[\langle (v+\zeta)_p \rangle \times \mathcal{H}_{sp}(x,\tx,v+\zeta) \Big]^{2\b} \times \Big(\mathcal{X}(2\b,0,v,\zeta)+\langle (v+\zeta)_p \rangle^{2\b}\mathcal{V}(2\b,0,v,\zeta)\Big)\\ 
        &\lesssim |x-\bx|^{\b}\left[\min \left\{\frac{1}{\cos a(x,v+\zeta)} ,\frac{1}{\cos a(\bx,v+\zeta)} \right \}\right]^{2\b} \\
 	&\quad \times
 	\left[\sup_{\substack{v\in\R^{3} \\ 0 < |x - \bx|\leq 1}} 
 	\langle v \rangle \frac{|f_{0}( x, v ) - f_{0}(\bx, v)|}{|x - \bx|^{2\b}}  
 	+ \sup_{ \substack{ x\in \O \\ 0 < |v - \bv|\leq 1   } }  \langle v \rangle^{2} \frac{|f_{0}( x, v ) - f_{0}( x, \bv)|}{|v - \bv|^{2\b}}  
 	+ \|w_{0} f_{0}\|_{\infty}\right].
 	\end{split}
 	\end{align} 
  Let $|v-\bv| \leq 1$ and $t<1$. We assume that \eqref{assume_v}, and recall $\tv(v,\bv,\zeta)$ in \eqref{def:tv}. We replace $s=0$ to \eqref{split3} and \eqref{split4}. By Lemma \ref{lem:f-f:tvbv} and Lemma \ref{lem:f-f:tvv:sec7}, we have
	\begin{align} \label{v:est : f-f1-half}
 	\begin{split}
 	\eqref{v:basic f-f1}
 	&\lesssim 
 	|v-\bv|^{\b}\Big[\Big( |v+\zeta|^{-1}_p+(t-s)\Big)\Big]^{2\b}   \mathcal{X}(2\b,0,v,\zeta) +
 	\Big[ \langle (v+\zeta)_p \rangle \times \mathcal{H}_{vel}(\bx, v,\tv,\zeta) \Big]^{2\b}  \mathcal{V}(2\b,0,v,\zeta) \\ 
        &\lesssim |v-\bv|^{\b} |v+\zeta|_p^{-2\b}	\left[\max \left\{\frac{1}{\cos^{1/2} a(\bx,v+\zeta)} ,\frac{1}{\cos^{1/2} a(\bx,\bv+\zeta)} \right\}\right]^{2\b}\\
 	&\quad \times
 	\left[\sup_{\substack{v\in\R^{3} \\ 0 < |x - \bx|\leq 1}} 
 	\langle v \rangle \frac{|f_{0}( x, v ) - f_{0}(\bx, v)|}{|x - \bx|^{2\b}}  
 	+ \sup_{ \substack{ x\in \O \\ 0 < |v - \bv|\leq 1   } }  \langle v \rangle^{2} \frac{|f_{0}( x, v ) - f_{0}( x, \bv)|}{|v - \bv|^{2\b}}  
 	+ \|w_{0} f_{0}\|_{\infty}\right], 
 	\end{split}
 	\end{align} Using \eqref{full k v}, \eqref{full k x}, and \eqref{specular Gamma} in Lemma \ref{lem_Gamma}, we have
  \begin{align} \label{est : f-f2-half}
\begin{split}
 	 \eqref{basic f-f2}
 	&\lesssim |x-\bx|^{\b}\int_{0}^{t} \Big[ \langle (v+\zeta)_p \rangle  \times  \mathcal{H}_{sp}(x,\tx,v+\zeta) \Big]^{2\b} e^{\varpi \langle v+\zeta \rangle^{2} s}
 	\mathfrak{H}_{sp}^{2\b}(s) \;ds \\
 	& \quad+ |x-\bx|^{\b}\int_{0}^{t} 
 	\Big[\langle (v+\zeta)_p \rangle^{2}  \times  \mathcal{H}_{sp}(x,\tx,v+\zeta) \Big]^{2\b} e^{\varpi\langle v+\zeta \rangle^{2}s}\mathfrak{H}_{vel}^{2\b}(s)\; ds \\
 	&\lesssim  |x-\bx|^{\b}\langle (v+\zeta)_p \rangle^{4\b} e^{\varpi\langle v+\zeta \rangle^{2}t} \left[\min \left\{\frac{1}{\cos a(x,v+\zeta)} ,\frac{1}{\cos a(\bx,v+\zeta)} \right\} \right]^{2\b}\\
 	&\quad \times
 	\left\{ \|w_{0}f_{0}\|_{\infty}\left[ \sup_{0\leq s \leq T}\mathfrak{H}_{sp}^{2\b}(s) + \sup_{0\leq s \leq T}\mathfrak{H}_{vel}^{2\b}(s) \right] + \|w_{0}f_{0}\|_{\infty}^{2} \right\},
 \end{split} 
 \end{align} and similarly,
 \begin{align} \label{v:est : f-f2-half}
\begin{split}
 	\eqref{v:basic f-f2}
 	&\lesssim |v-\bv|^{\b} |v+\zeta|_p^{-2\b}\langle (v+\zeta)_p \rangle^{2\b}e^{\varpi\langle v+\zeta \rangle^{2}t}\left[\max \left\{\frac{1}{\cos^{1/2} a(\bx,v+\zeta)} ,\frac{1}{\cos^{1/2} a(\bx,\bv+\zeta)} \right\}\right]^{2\b}\\
 	&\quad \times
 	\left\{ \|w_{0}f_{0}\|_{\infty}\left[ \sup_{0\leq s \leq T}\mathfrak{H}_{sp}^{2\b}(s) + \sup_{0\leq s \leq T}\mathfrak{H}_{vel}^{2\b}(s) \right] + \|w_{0}f_{0}\|_{\infty}^{2} \right\},
 \end{split} 
 \end{align} where $a(x,v)$ is in \eqref{detail a,b:1}. Using \eqref{full nu v}, \eqref{full nu x}, and \eqref{specular Gamma} in Lemma \ref{lem_Gamma}, we have
\begin{align}  \label{est : f-f3-half}
    \begin{split}
        \eqref{basic f-f3}+\eqref{basic f-f4} 
        &\lesssim  |x-\bx|^{\b} \langle (v+\zeta)_p \rangle^{4\b} e^{\varpi\langle v+\zeta \rangle^{2}t} \left[\min \left\{\frac{1}{\cos a(x,v+\zeta)} ,\frac{1}{\cos a(\bx,v+\zeta)} \right \}\right]^{2\b} \\
        &\quad \times
 	\left\{ \|w_{0}f_{0}\|_{\infty}\left[ \sup_{0\leq s \leq T}\mathfrak{H}_{sp}^{2\b}(s) + \sup_{0\leq s \leq T}\mathfrak{H}_{vel}^{2\b}(s) \right] + \|w_{0}f_{0}\|_{\infty}^{3} \right\}
    \end{split} 
\end{align} and 
\begin{align} \label{v:est : f-f3-half}
    \begin{split}
        \eqref{v:basic f-f3}+\eqref{v:basic f-f4}
        &\lesssim |v-\bv|^{\b} |v+\zeta|_p^{-2\b}\langle (v+\zeta)_p \rangle^{2\b}e^{\varpi\langle v+\zeta \rangle^{2}t}\left[\max \left\{\frac{1}{\cos^{1/2} a(\bx,v+\zeta)} ,\frac{1}{\cos^{1/2} a(\bx,\bv+\zeta)} \right\}\right]^{2\b}\\
        &\quad \times
 	\left\{ \|w_{0}f_{0}\|_{\infty}\left[ \sup_{0\leq s \leq T}\mathfrak{H}_{sp}^{2\b}(s) + \sup_{0\leq s \leq T}\mathfrak{H}_{vel}^{2\b}(s) \right] + \|w_{0}f_{0}\|_{\infty}^{3} \right\}.
    \end{split}
\end{align}
 Let $P_3(s)=|s|+|s|^2+|s|^3$. From \eqref{est : f-f1-half}, \eqref{est : f-f2-half}, \eqref{est : f-f3-half} and Proposition \ref{prop_unif H}, for $0<\beta<1/4$, we have
 	\begin{align} \label{thm_sub}
 	\begin{split}
 	& \max \{\ \cos a(x,v+\zeta), \cos a(\bx,v+\zeta) \}^{2\b}\times \langle (v+\zeta)_p \rangle^{-4\b} e^{-\varpi\langle v+\zeta \rangle^{2}t}| f(t,x,v+\zeta) - f(t, \bx, v+\zeta) |  \\
 	&\lesssim   |x-\bx|^{\b} \|w_{0}f_{0}\|^2_{\infty} \left[ \sup_{0\leq s \leq T}\mathfrak{H}_{sp}^{2\b}(s) + \sup_{0\leq s \leq T}\mathfrak{H}_{vel}^{2\b}(s) \right] +  \mathcal{P}_{3} (\|w_{0}f_{0}\|_{\infty})  \\
 	&\lesssim  |x-\bx|^{\b}
 	\|w_{0}f_{0}\|^2_{\infty}
 	\left[
 	\sup_{\substack{v\in\R^{3} \\ 0 < |x - \bx|\leq 1}} 
 	\langle v \rangle  \frac{|f_{0}( x, v ) - f_{0}(\bx, v)|}{|x - \bx|^{2\b}}  
 	+ \sup_{ \substack{ x\in \O \\ 0 < |v - \bv|\leq 1   } } \langle v \rangle^{2} \frac{|f_{0}( x, v ) - f_{0}( x, \bv)|}{|v - \bv|^{2\b}}  
 	\right] 
 	+ \mathcal{P}_{3}(\|w_{0}f_{0}\|_{\infty}).  
 	\end{split}
 	\end{align}  From \eqref{v:est : f-f1-half}, \eqref{v:est : f-f2-half}, \eqref{v:est : f-f3-half} and Proposition \ref{prop_unif H}, for $0<\beta<1/4$, we have
 	\begin{align} \label{v:thm_sub}
 	\begin{split}
 	& \min \{\ \cos^{1/2} a(\bx,v+\zeta), \cos^{1/2} a(\bx,\bv+\zeta) \}^{2\b} \\ &\quad \quad \times|v+\zeta|^{2\b}_p \langle (v+\zeta)_p \rangle^{-2\b} e^{-\varpi\langle v+\zeta \rangle^{2}t}| f(t,\bx,v+\zeta) - f(t, \bx, \bv+\zeta) | \\
 	&\lesssim  |v-\bv|^{\b} \|w_{0}f_{0}\|^2_{\infty} \left[ \sup_{0\leq s \leq T}\mathfrak{H}_{sp}^{2\b}(s) + \sup_{0\leq s \leq T}\mathfrak{H}_{vel}^{2\b}(s) \right] +  \mathcal{P}_{3} (\|w_{0}f_{0}\|_{\infty})  \\
 	&\lesssim 
 	|v-\bv|^{\b} \|w_{0}f_{0}\|^2_{\infty}
 	\left[
 	\sup_{\substack{v\in\R^{3} \\ 0 < |x - \bx|\leq 1}} 
 	\langle v \rangle  \frac{|f_{0}( x, v ) - f_{0}(\bx, v)|}{|x - \bx|^{2\b}}  
 	+ \sup_{ \substack{ x\in \O \\ 0 < |v - \bv|\leq 1   } } \langle v \rangle^{2} \frac{|f_{0}( x, v ) - f_{0}( x, \bv)|}{|v - \bv|^{2\b}}  
 	\right]
 	+ \mathcal{P}_{3}(\|w_{0}f_{0}\|_{\infty}).
 	\end{split}
 	\end{align} From \eqref{detail a,b:1}, 
  \begin{align}\label{detail:cos a}
      \cos a(x,v) = \sqrt{1-\frac{|x|^2_p}{R^2}+\frac{(x_p \cdot \hat{v}_p)^2}{R^2}}>0
  \end{align}  for any $x \in \O,  v \in \mathbb{R}^3$. Therefore,  we conclude \eqref{est:Holder} by \eqref{thm_sub} and \eqref{v:thm_sub}. If \eqref{assume_x} or \eqref{assume_v} does not hold, we let $x=\tx$ or $v+\zeta=\tv+\zeta$, and do not consider the singularity part of space or velocity, and obtain \eqref{est:Holder} more easily.
\end{proof} 

\section{Acknowledgements}
GA and DL are supported by the National Research Foundation of Korea(NRF) grant funded by the Korea government(MSIT)(No.RS-2023-00212304 and No.RS-2023-00219980). 

\phantom{\cite{CKL1} \cite{CKL2} \cite{CK1} \cite{CIP} \cite{DV} \cite{DiLion} \cite{dC} \cite{DHWY} \cite{DHWZ} \cite{DKL2020} \cite{DW} \cite{gl} \cite{Guo_Classic} \cite{Guo_VMB} \cite{Guo10} \cite{CyrilLouis} \cite{Kim11} \cite{K2} \cite{KY} \cite{KLP} \cite{SA} \cite{GKTT2}}.

\bibliographystyle{abbrv}
\nocite{*}
\bibliography{reference.bib}

\end{document}